\expandafter\def\csname ver@fixltx2e.sty\endcsname{}
 \documentclass[journal,9pt]{IEEEtran}

\usepackage{amsthm}
\usepackage{amsfonts}
\usepackage{amssymb,bm}
\usepackage{mathrsfs}
\usepackage{mathtools}
\usepackage{amsmath}

\usepackage{cite}
\usepackage{balance}

\let\labelindent\relax
\usepackage[inline]{enumitem}

\usepackage{verbatim}

\usepackage{dblfloatfix}
\usepackage{url}
\usepackage{graphicx}
\usepackage{textcomp}

\usepackage{accents}

\newlength{\dhatheight}
\newcommand{\doublehat}[1]{%
    \settoheight{\dhatheight}{\ensuremath{\hat{#1}}}%
    \addtolength{\dhatheight}{-0.25ex}%
    \hat{\vphantom{\rule{1pt}{\dhatheight}}%
    \smash{\hat{#1}}}}

\usepackage{array}
\newcolumntype{P}[1]{>{\centering\arraybackslash}p{#1}}

\usepackage{xcolor}

\newtheorem{theorem}{Theorem}
\newtheorem{definition}{Definition}
\newtheorem{lemma}{Lemma}

\newtheorem{remark}{Remark}

\def\BibTeX{{\rm B\kern-.05em{\sc i\kern-.025em b}\kern-.08em
    T\kern-.1667em\lower.7ex\hbox{E}\kern-.125emX}}

\begin{document}

\title{Inverse Extended Kalman Filter --- Part I: Fundamentals}

\author{Himali Singh, Arpan Chattopadhyay$^\ast$ and Kumar Vijay Mishra$^\ast$\vspace{-24pt}
\thanks{$^\ast$A. C. and K. V. M. have made equal contributions.}
\thanks{H. S. and A. C. are with the Electrical Engineering Department, Indian Institute of Technology (IIT) Delhi, India. A. C. is also associated with the Bharti School of Telecommunication Technology and Management, IIT Delhi. Email: \{eez208426, arpanc\}@ee.iitd.ac.in.} 
\thanks{K. V. M. is with the United States DEVCOM Army Research Laboratory, Adelphi, MD 20783 USA. E-mail: kvm@ieee.org.}
\thanks{A. C. acknowledges support via the professional development fund and professional development  allowance from IIT Delhi, grant no. GP/2021/ISSC/022 from I-Hub Foundation for Cobotics and grant no. CRG/2022/003707 from Science and Engineering Research Board (SERB), India. H. S. acknowledges support via Prime Minister Research Fellowship. K. V. M. acknowledges support from the National Academies of Sciences, Engineering, and Medicine via Army Research Laboratory Harry Diamond Distinguished Fellowship.}
\thanks{The conference precursor of this work has been published in the 2022 Asilomar Conference on Signals, Systems, and Computers.}
}

\maketitle

\begin{abstract}
Recent advances in counter-adversarial systems have garnered significant research attention to inverse filtering from a Bayesian perspective. For example, interest in estimating the adversary’s Kalman filter tracked estimate with the purpose of predicting the adversary's future steps has led to recent formulations of \textit{inverse Kalman filter} (I-KF). In this context of inverse filtering, we address the key challenges of non-linear process dynamics and unknown input to the forward filter by proposing an \textit{inverse extended Kalman filter} (I-EKF). The purpose of this paper and the companion paper (Part II) is to develop the theory of I-EKF in detail. In this paper, we assume perfect system model information and derive I-EKF with and without an unknown input when both forward and inverse state-space models are non-linear. In the process, I-KF-with-unknown-input is also obtained. We then provide theoretical stability guarantees using both bounded non-linearity and unknown matrix approaches and prove the I-EKF's consistency. Numerical experiments validate our methods for various proposed inverse filters using the recursive Cram\'{e}r-Rao lower bound as a benchmark. In the companion paper (Part II), we further generalize these formulations to highly non-linear models and propose reproducing kernel Hilbert space-based EKF to handle incomplete system model information.
\end{abstract}

\begin{IEEEkeywords}
Bayesian filtering, counter-adversarial systems, extended Kalman filter, inverse filtering, non-linear processes.
\end{IEEEkeywords}

\vspace{-8pt}
\section{Introduction}
\label{sec:introduction}
In many engineering applications, it is desired to infer the parameters of a filtering system by observing its output. This \textit{inverse filtering} is useful in applications such as system identification, fault detection, image deblurring, and signal deconvolution \cite{idier2013bayesian,gustafsson2007statistical}. Conventional inverse filtering is limited to non-dynamic systems. However, applications such as cognitive and counter-adversarial systems \cite{haykin2006cognitive,bell2015cognitive,mattila2020inverse} have recently been shown to require designing the inverse of classical stochastic filters such as hidden Markov model (HMM) filter \cite{elliott2008hidden} and Kalman filter (KF) \cite{kalman1960new}. The cognitive systems are intelligent units that sense the environment, learn relevant information about it, and then adapt themselves in real-time to optimally enhance their performance. For example, a cognitive radar \cite{mishra2020toward} adapts both transmitter and receiver processing in order to achieve desired goals such as improved target detection \cite{mishra2017performance} and tracking \cite{bell2015cognitive,sharaga2015optimal}. In this context, \cite{krishnamurthy2019how} recently introduced \textit{inverse cognition}, in the form of inverse stochastic filters, to detect cognitive sensor and further estimate the information that the same sensor may have learnt. In this two-part paper, we focus on inverse stochastic filtering for such inverse cognition applications.

At the heart of inverse cognition are two agents: `defender' (e.g., an intelligent target) and an `adversary' (e.g., a sensor or radar) equipped with a Bayesian tracker. The adversary infers an estimate of the defender's kinematic state and cognitively adapts its actions based on this estimate. The defender observes adversary's actions with the goal to predict its future actions in a Bayesian sense. 
In particular, \cite{krishnamurthy2020identifying} developed stochastic revealed preferences-based algorithms to ascertain if the adversary's actions are consistent with optimizing a utility function; and if so, estimate that function. On the other hand, \cite{krishnamurthy2021adversarial,kang2023} deal with smart interference design to force an adversary to change its actions. The motivation for the problem lies in developing counter-adversarial systems. For instance, an intelligent target can sense the cognitive radar's waveform adaptations and employ an inverse filter to infer the radar's estimate of its state\cite{krishnamurthy2019how}. Similar examples abound in interactive learning\cite{krishnamurthy2019how}, fault diagnosis, cyber-physical security\cite{mattila2017inverse}, and inverse reinforcement learning\cite{ng2000algorithms}.

If the defender aims to guard against the adversary's future actions, it requires an estimate of the adversary's inference. This is precisely the objective of inverse Bayesian filtering. In (forward) Bayesian filtering, given noisy observations, a posterior distribution of the underlying state is obtained. An example is the KF, which provides optimal estimates of the underlying state in linear system dynamics with Gaussian measurement and process noises. The inverse filtering problem, on the other hand, is concerned with estimating this posterior distribution of a Bayesian filter given the noisy measurements of the posterior. An example of such a system is the recently introduced inverse Kalman filter (I-KF) \cite{krishnamurthy2019how}. Note that, historically, the Wiener filter -- a special case of KF when the process is stationary -- has long been used for frequency-domain inverse filtering for deblurring in image processing \cite{biemond1990iterative}. Further, some early works \cite{kalman1964linear} have investigated the inverse problem of finding cost criterion for a control policy. 

Although KF and its continuous-time variant Kalman-Bucy filter \cite{kalman1961new} are highly effective in many practical applications, they are optimal for only linear and Gaussian models. In practice, many engineering problems involve non-linear processes \cite{haykin2004kalman,simon2006optimal}. In these cases, a \textit{linearized KF} is used, wherein the states of a linear system represent the deviations from a nominal trajectory of a non-linear system. The KF estimates the deviations from the nominal trajectory and obtains an estimate of the states of the non-linear system. The linearized KF is extended to directly estimate the states of a non-linear system in the extended KF (EKF) \cite{schmidt1966application}. The linearization is 
locally at the state estimates through Taylor series expansion. This is very similar to the Volterra series filters  \cite{zaknich2005principles} that are non-linear counterparts of adaptive linear filters. Besides the traditional state estimation applications, EKF has also been considered in learning applications like dual and joint estimation of state and parameters\cite{haykin2004kalman}, parameter optimization for fuzzy logic systems\cite{khanesar2011extended} and training neural networks\cite{simon2002training,wang2011convergence}. The EKF is further connected to the general approximate Bayesian inference approaches in machine learning, wherein the EKF may be viewed as a member of a general class of Gaussian filters that assume a Gaussian distributed conditional probability density. The mean and covariance of the assumed density are then updated recursively using the observations. These Gaussian filtering approaches are a special case of assumed density filtering (ADF)\cite{maybeck1982stochastic} or online Bayesian learning\cite{opper1999bayesian}, which sequentially computes the approximate posterior distribution of the underlying state. Expectation propagation is a further extension of ADF where new observations are also used to refine the previous approximations iteratively \cite{minka2013expectation}.

While inverse non-linear filters have been studied for adaptive systems in some previous works \cite{broomhead1996nonlinear,shen2001robust}, the inverse of non-linear stochastic filters such as EKF remain unexamined so far. To address the aforementioned non-linear inverse cognition scenarios, contrary to prior works which focus on only linear I-KF \cite{krishnamurthy2019how}, our goal is to derive and analyze inverse EKF (I-EKF). Note that the I-EKF is different from the \textit{inversion of EKF} \cite{zhengyu2021iterated}, which may not take the same form as EKF, is employed on the adversary's side, and is unrelated to our inverse cognition problem. Similarly, the non-linear extended information filter (EIF) proposed in \cite{mutambara1999information} used inverse of covariance matrix and was compared with KF for estimation of the same states. Our inverse EKF has a different formulation that is focused on estimating the inference of an adversary who is also using an EKF to estimate the defender's state. Further, the adversary does not attempt to hide its strategy from the defender, which is a more challenging problem recently addressed in \cite{lourencco2020protect,lourencco2021hidden}. If the adversary also guards itself against the defender, the adversary-defender interaction then requires an inverse-inverse reinforcement learning-based representation of the problem, which is not the focus of our current work and has been addressed in other recent works \cite{pattanayak2022inverse,pattanayak2022meta}.

Preliminary results of this work appeared in our conference publication \cite{singh2022inverse}, where only I-EKF-without-unknown-inputs was formulated. In this paper, we present inverses of many other EKF formulations for systems with unknown inputs and provide their stability analyses. The companion paper (Part II) \cite{singh2022inverse_part2} further develops the I-EKF theory for highly non-linear systems where first-order EKF does not sufficiently addresses the linear approximation. Our main contributions in this paper (Part I) are:\\
\textbf{1) I-KF and I-EKF with unknown inputs.} In the inverse cognition scenario, the target may introduce additional motion or jamming that is known to the target but not to the adversarial cognitive sensor. In this context, while deriving I-EKF, we consider a more general non-linear system model with unknown input. Unknown inputs refer to exogenous excitations to the system which affect the state transition and observations but are not known to the agent employing the stochastic filter. In the process, we also obtain I-KF-with-unknown-input that was not examined in the I-KF developed in \cite{krishnamurthy2019how}. Here, similar to the inverse cognition frameworks investigated in \cite{krishnamurthy2019how,mattila2020inverse}, we assume that the adversary's filter is known to the defender. In the companion paper (Part II) \cite{singh2022inverse_part2}, we consider the case when no prior information about the adversary's filter is available.\\
\textbf{2) Augmented states for I-EKF.} For systems with unknown inputs, the adversary's state estimate depends on its estimate of the unknown input. As a result, the adversary's forward filters vary with system models. We overcome this challenge by considering augmented states in the inverse filter so that the unknown input estimation is performed jointly with state estimation, including for KF with direct feed-through. For different inverse filters, separate augmented states are considered depending on the state transitions for the inverse filter.\\
\textbf{3) Stability of I-EKF.} The treatment of linear filters includes filter stability and model error sensitivity. But, in general, stability and convergence results for non-linear KFs, and more so for their inverses, are difficult to obtain. In this work, we show the stability of I-EKF using two techniques. The first approach is based on bounded non-linearities, which has been earlier employed for proving stochastic stability of discrete-time \cite{reif1999stochastic} and continuous-time \cite{reif2000continuousEKFstability} EKFs. Here, the estimation error was shown to be exponentially bounded in the mean-squared sense. The second method relaxes the bound on the initial estimation error by introducing unknown matrices to model the linearization errors \cite{xiong2006performance_ukf}. Besides providing the sufficient conditions for error boundedness, this approach also rigorously justifies the enlarging of the noise covariance matrices to stabilize the filter \cite{wu2007comments}. Since the I-EKF's error dynamics depends on the forward filter's recursive updates, the derivations of these theoretical guarantees are not straightforward. In the process, we also obtain novel stability results for forward EKF using the unknown matrix approach. We further show the consistency of I-EKF's estimates. We validate the estimation errors of all inverse filters through extensive numerical experiments with recursive Cram\'{e}r-Rao lower bound (RCRLB) \cite{tichavsky1998posterior} as the performance metric.

The rest of the paper is organized as follows. In the next section, we provide the background of inverse cognition model. The inverse EKF with unknown input is then derived in Section~\ref{sec:ekfunknown} for the case of the forward EKF with and without direct feed-through. Here, we also obtain the standard I-EKF in the absence of unknown input. Then, similar cases are considered for inverse KF with unknown input in Section~\ref{sec:kfunknown}. We then derive the stability conditions in Section~\ref{sec:stability}. In Section~\ref{sec:simulations}, we corroborate our results with numerical experiments before concluding in Section~\ref{sec:summary}.

Throughout the paper, we reserve boldface lowercase and uppercase letters for vectors (column vectors) and matrices, respectively. The transpose operation and $l_{2}$ norm (for a vector) are denoted by $(\cdot)^T$ and $||\cdot||_{2}$, respectively. The notation $\textrm{Tr}(\mathbf{A})$, $\textrm{rank}(\mathbf{A})$, and $||\mathbf{A}||$, respectively, denote the trace, rank, and spectral norm of $\mathbf{A}$. For matrices $\mathbf{A}$ and $\mathbf{B}$, the inequality $\mathbf{A}\preceq\mathbf{B}$ means that $\mathbf{B}-\mathbf{A}$ is a positive semidefinite (p.s.d.) matrix. For a function $f:\mathbb{R}^{n}\rightarrow\mathbb{R}^{m}$, $\nabla f$ denotes the $\mathbb{R}^{m \times n}$ Jacobian matrix. Similarly, for a function $f:\mathbb{R}^{n}\rightarrow\mathbb{R}$, $\nabla f$ denote the gradient vector ($\mathbb{R}^{n\times 1}$). A $n\times n$ identity matrix is denoted by $\mathbf{I}_{n}$ and a $n\times m$ all zero matrix is denoted by $\mathbf{0}_{n\times m}$. The notation $\lbrace a_{i}\rbrace_{i_{1}\leq i\leq i_{2}}$ denotes a set of elements indexed by integer $i$. 
The notation $\mathbf{x} \sim \mathcal{N}(\boldsymbol{\mu},\mathbf{Q})$ and $x \sim \mathcal{U}[u_l,u_u]$, respectively, represent a random variable drawn from a normal distribution with mean  $\boldsymbol{\mu}$ and covariance matrix $\mathbf{Q}$, and the uniform distribution over $[u_l,u_u]$.

\vspace{-8pt}
\section{Desiderata for Inverse Cognition}
\label{sec:background}
Consider a discrete-time stochastic dynamical system as the defender's state evolution process $\{\mathbf{x}_k\}_{k \geq 0}$, where $\mathbf{x}_k \in \mathbb{R}^{n \times 1}$ is the state at the $k$-th time instant. The defender perfectly knows its current state $\mathbf{x}_{k}$. The control input $\mathbf{u}_k \in \mathbb{R}^{m \times 1}$ is known to the defender but not to the adversary. In a linear state-space model, we denote the state-transition and control input matrices by $\mathbf{F} \in \mathbb{R}^{n\times n}$ and $\mathbf{B} \in \mathbb{R}^{n\times m}$, respectively. The defender's state evolves as
\par\noindent\small
\begin{align}
\mathbf{x}_{k+1}=\mathbf{Fx}_{k}+\mathbf{Bu}_{k}+\mathbf{w}_{k},\label{eqn: linear x with input}
\end{align}
\normalsize
where $\mathbf{w}_{k}\sim\mathcal{N}(\mathbf{0}_{n\times 1},\mathbf{Q})$ is the process noise with covariance matrix $\mathbf{Q} \in \mathbb{R}^{n\times n}$. At the adversary, the observation and control input matrices are given by $\mathbf{H} \in \mathbb{R}^{p\times n}$ and $\mathbf{D} \in \mathbb{R}^{p\times m}$, respectively. The adversary makes a noisy observation $\mathbf{y}_k \in \mathbb{R}^{p \times 1}$ at time $k$ as
\par\noindent\small
\begin{align}
\mathbf{y}_{k}=\mathbf{Hx}_{k}+\mathbf{Du}_{k}+\mathbf{v}_{k},
\label{eqn: linear y withdf}
\end{align}
\normalsize
where $\mathbf{v}_{k}\sim\mathcal{N}(\mathbf{0}_{p\times 1},\mathbf{R})$ is the adversary's measurement noise with covariance matrix $\mathbf{R} \in \mathbb{R}^{p\times p}$.

The adversary uses $\{\mathbf{y}_j\}_{1 \leq j \leq k}$ to compute the estimate $\hat{\mathbf{x}}_k$ of the defender's state $\mathbf{x}_{k}$ using a (forward) stochastic filter. The adversary then uses this estimate to administer an action matrix $\mathbf{G} \in \mathbb{R}^{n_a \times n}$ on $\hat{\mathbf{x}}_{k}$. The defender makes noisy observations of this action as
\par\noindent\small
\begin{align}
\mathbf{a}_{k}=\mathbf{G}\hat{\mathbf{x}}_{k}+\bm{\epsilon}_{k}\;\; \in \mathbb{R}^{n_{a} \times 1},
\label{eqn: linear a}
\end{align}
\normalsize
where $\bm{\epsilon}_{k}\sim \mathcal{N}(\mathbf{0}_{n_{a}\times 1},\bm{\Sigma_{\epsilon}})$ is the defender's measurement noise with covariance matrix $\bm{\Sigma}_{\epsilon} \in \mathbb{R}^{n_{a}\times n_{a}}$. Finally, the defender uses $\{\mathbf{a}_j, \mathbf{x}_j,\mathbf{u}_{j}\}_{1 \leq j \leq k}$ to compute the estimate $\doublehat{\mathbf{x}}_k \in \mathbb{R}^{n \times 1}$ of $\hat{\mathbf{x}}_k $ in the (inverse) stochastic filter. Define $\hat{\mathbf{u}}_{k}$ to be the estimate of $\mathbf{u}_k$ as computed in the adversary's forward filter, while $\doublehat{\mathbf{u}}_{k}$ is an estimate of $\hat{\mathbf{u}}_{k}$ as computed by the defender's inverse filter. The noise processes $\{\mathbf{w}_{k}\}_{k \geq 0}$, $\{\mathbf{v}_{k}\}_{k \geq 1}$ and $\{\bm{\epsilon}_{k}\}_{k \geq 1}$ are mutually independent and i.i.d. across time. These noise distributions are known to the defender as well as the adversary. The adversary and defender are entirely different agents employing independent sensors to observe each other. Furthermore, in the inverse filtering problem, the adversary is unaware that the defender is observing the former. Hence, the defender's measurements noise $\bm{\epsilon}_{k}$ is independent of the adversary's state estimate and the administered action. When the unknown input is absent, either $\mathbf{B}=\mathbf{0}_{n \times m}$ or $\mathbf{D}=\mathbf{0}_{p \times m}$ or both vanish. Throughout the paper, we assume that both parties (adversary and defender) have perfect knowledge of the system model and parameters. Additionally, the defender is assumed to know the forward filter employed by the adversary. The companion paper (Part II) \cite{singh2022inverse_part2} considers the case when this perfect knowledge is not available and also, numerically analyzes the mismatched forward and inverse filters case. In particular, the proposed inverse filters provide reasonably accurate estimates even when the defender assumes an incorrect forward filter. In some cases, a sophisticated inverse filter may even provide better estimates.

When the system dynamics are non-linear, then the matrix pairs $\{\mathbf{F, B}\}$, $\{\mathbf{H, D}\}$, and the matrix $\mathbf{G}$ are replaced by non-linear functions $f(\cdot, \cdot)$, $h(\cdot,\cdot)$, and $g(\cdot)$, respectively, as
\par\noindent\small
\begin{align}
\mathbf{x}_{k+1}&=f(\mathbf{x}_{k},\mathbf{u}_{k})+\mathbf{w}_{k},\label{eqn: non x with input}\\
\mathbf{y}_{k}&=h(\mathbf{x}_{k},\mathbf{u}_{k})+\mathbf{v}_{k},\label{eqn: non y withdf}\\
\mathbf{a}_{k}&=g(\hat{\mathbf{x}}_{k})+\bm{\epsilon}_{k}.\label{eqn: non a}
\end{align}
\normalsize
This is a \textit{direct feed-through} (DF) model, wherein $\mathbf{y}_{k}$ depends on the unknown input. Without DF, observations \eqref{eqn: non y withdf} becomes
\par\noindent\small
\begin{align}
\mathbf{y}_{k}=h(\mathbf{x}_{k})+\mathbf{v}_{k}.\label{eqn: non y withoutdf}
\end{align}
\normalsize

We show in the following Section~\ref{sec:ekfunknown}, the presence or absence of the unknown input leads to different solution approaches towards forward  and inverse filters. For simplicity, the presence of known exogenous inputs is also ignored in state evolution and observations. However, it is trivial to extend the inverse filters developed in this paper for these modifications in the system model. Throughout the paper, we focus on discrete-time models.

\vspace{-8pt}
\section{I-EKF with Unknown Input}
\label{sec:ekfunknown} 
One of the earliest approaches to treat the unknown input was to model the inputs as a stochastic process with known evolution dynamics and jointly estimate the state and inputs. Relaxing the known input dynamics assumption, \cite{kitanidis1987unbiased,gillijns2007unknownkf,gillijns2007kfb,zhang2022boundedness} developed and analyzed unbiased minimum variance linear filters with unknown inputs. Recently, \cite{marco2022regularized,kong2021kalman} have also considered non-persistent and norm-constrained unknown input estimation in linear systems. Various EKF variants to handle unknown inputs in non-linear systems have also been proposed\cite{yang2007adaptive,pan2010applying,xiao2018adaptive,meyer2020unknown,kim2020simultaneous}. We consider a more general EKF with unknown inputs based on a weighted least squared error criterion in case of both without \cite{pan2010applying} and with \cite{yang2007adaptive} DF. We do not make any other assumption on the inputs.

The EKF linearizes the model about the nominal values of the state vector and control input. It is similar to the iterated least squares (ILS) method except that the former is for dynamical systems and the latter is not \cite{mendel1995lessons}.
\begin{remark}\label{remark:with and without diff}
Note that the optimal forward EKFs with and without DF are conceptually different. In the latter case, while the observation $\mathbf{y}_{k}$ is unaffected by the unknown input $\mathbf{u}_{k}$, it is still dependent on $\mathbf{u}_{k-1}$ through $\mathbf{x}_k$; this induces a one-step delay in the adversary's estimate of $\mathbf{u}_k$. On the other hand, with DF, 
there is no such delay in estimating $\mathbf{u}_k$. 
\end{remark}

We now show that this difference results in different inverse filters for these two cases.

\subsection{I-EKF-without-DF unknown input}
\label{subsec:ekfwithoutdf}
Consider the non-linear system without DF given by \eqref{eqn: non x with input} and \eqref{eqn: non y withoutdf}. Linearize the model functions as $\mathbf{F}_{k}\doteq\nabla_{\mathbf{x}}f(\mathbf{x},\hat{\mathbf{u}}_{k-1})|_{\mathbf{x}=\hat{\mathbf{x}}_{k}}$, $\mathbf{B}_{k} \doteq\nabla_{\mathbf{u}}f(\hat{\mathbf{x}}_{k},\mathbf{u})|_{\mathbf{u}=\hat{\mathbf{u}}_{k-1}}$ and $\mathbf{H}_{k+1}\doteq\nabla_{\mathbf{x}}h(\mathbf{x})|_{\mathbf{x}=\hat{\mathbf{x}}_{k+1|k}}$. 
\subsubsection{Forward filter}\label{subsubsec:forward EKF without DF}
The forward filter's recursive state estimation procedure first obtains the prediction $\hat{\mathbf{x}}_{k+1|k}$ of the current state using the previous state and input estimates, with $\bm{\Sigma}^{x}_{k+1|k}$ as the associated state prediction error covariance matrix of $\hat{\mathbf{x}}_{k+1|k}$. Then, the state and input gain matrices $\mathbf{K}^{x}_{k+1}$ and $\mathbf{K}^{u}_{k}$, respectively, are computed along with the input estimation (with delay) covariance matrix $\bm{\Sigma}^{u}_{k}$. Finally, the state $\hat{\mathbf{x}}_{k+1}$, input $\hat{\mathbf{u}}_{k}$, and covariance matrix $\bm{\Sigma}^{x}_{k+1}$ are updated using current observation $\mathbf{y}_{k+1}$, and gain matrices $\mathbf{K}^{x}_{k+1}$ and $\mathbf{K}^{u}_{k}$. Note that the current observation $\mathbf{y}_{k+1}$ provides an estimate $\hat{\mathbf{u}}_{k}$ of the input $\mathbf{u}_{k}$ at the previous time step. The adversary's forward EKF's recursions are\cite{pan2010applying}:
\par\noindent\small
\begin{align}
&\textit{Prediction:}\;\hat{\mathbf{x}}_{k+1|k}=f(\hat{\mathbf{x}}_{k},\hat{\mathbf{u}}_{k-1}),\label{eqn: ekfwithoutdf predict}\\
&\textit{Gain computation:}\;\bm{\Sigma}^{x}_{k+1|k}=\mathbf{F}_{k}\bm{\Sigma}^{x}_{k}\mathbf{F}_{k}^{T}+\mathbf{Q},\nonumber\\
&\mathbf{K}^{x}_{k+1}=\bm{\Sigma}^{x}_{k+1|k}\mathbf{H}_{k+1}^{T}\left(\mathbf{H}_{k+1}\bm{\Sigma}^{x}_{k+1|k}\mathbf{H}_{k+1}^{T}+\mathbf{R}\right)^{-1},\nonumber\\
&\bm{\Sigma}^{u}_{k}=\left(\mathbf{B}_{k}^{T}\mathbf{H}_{k+1}^{T}\mathbf{R}^{-1}(\mathbf{I}_{p\times p}-\mathbf{H}_{k+1}\mathbf{K}^{x}_{k+1})\mathbf{H}_{k+1}\mathbf{B}_{k}\right)^{-1},\nonumber\\
&\mathbf{K}^{u}_{k}=\bm{\Sigma}^{u}_{k}\mathbf{B}_{k}^{T}\mathbf{H}_{k+1}^{T}\mathbf{R}^{-1}(\mathbf{I}_{p\times p}-\mathbf{H}_{k+1}\mathbf{K}^{x}_{k+1}),\nonumber\\
&\textit{Update:}\;\hat{\mathbf{x}}_{k+1}=\hat{\mathbf{x}}_{k+1|k}+\mathbf{K}^{x}_{k+1}(\mathbf{y}_{k+1}-h(\hat{\mathbf{x}}_{k+1|k})),\label{eqn: ekfwithoutdf update x}\\
&\hat{\mathbf{u}}_{k}=\mathbf{K}^{u}_{k}(\mathbf{y}_{k+1}-h(\hat{\mathbf{x}}_{k+1|k})+\mathbf{H}_{k+1}\mathbf{B}_{k}\hat{\mathbf{u}}_{k-1}),\label{eqn: ekfwithoutdf update u}\\
&\textit{Covariance matrix update:}\;\bm{\Sigma}^{x}_{k+1}=\nonumber\\
&\hspace{-0.25cm}(\mathbf{I}_{n\times n}-\mathbf{K}^{x}_{k+1}\mathbf{H}_{k+1})\left(\bm{\Sigma}^{x}_{k+1|k}+\mathbf{B}_{k}\bm{\Sigma}^{u}_{k}\mathbf{B}_{k}^{T}(\mathbf{I}_{n\times n}-\mathbf{K}^{x}_{k+1}\mathbf{H}_{k+1})^{T}\right).\nonumber
\end{align}
\normalsize
Forward filter exists if $\textrm{rank}(\bm{\Sigma}^{u}_{k})=m$, for all $k\geq 0$, and $p \geq m$ \cite{pan2010applying}. We provide a detailed derivation of the forward EKF-without-DF recursions (omitted in \cite{pan2010applying}) in Appendix~\ref{App-forward-EKF-without-DF-recursions}.
\subsubsection{Inverse filter}
Consider an augmented state vector $\mathbf{z}_{k}=\begin{bmatrix}
\hat{\mathbf{x}}_{k}^{T} & \hat{\mathbf{u}}_{k-2}^{T}
\end{bmatrix}^{T}$. The defender's observation $\mathbf{a}_{k}$ in \eqref{eqn: non a} is the first observation that contains the information about unknown input estimate $\hat{\mathbf{u}}_{k-2}$, because of the delay in forward filter input estimate. Hence, the delayed estimate $\hat{\mathbf{u}}_{k-2}$ is considered in the augmented state $\mathbf{z}_{k}$. Define $\widetilde{\phi}_{k}(\hat{\mathbf{x}}_{k},\hat{\mathbf{u}}_{k-1},\mathbf{x}_{k+1},\mathbf{v}_{k+1})=f(\hat{\mathbf{x}}_{k},\hat{\mathbf{u}}_{k-1})-\mathbf{K}^{x}_{k+1}h(f(\hat{\mathbf{x}}_{k},\hat{\mathbf{u}}_{k-1}))+\mathbf{K}^{x}_{k+1}h(\mathbf{x}_{k+1})+\mathbf{K}^{x}_{k+1}\mathbf{v}_{k+1}$.
From \eqref{eqn: non y withoutdf}-\eqref{eqn: ekfwithoutdf update u}, state transition equations of augmented state vector are $\hat{\mathbf{x}}_{k+1}=\widetilde{f}_{k}(\hat{\mathbf{x}}_{k},\hat{\mathbf{u}}_{k-2},\hat{\mathbf{x}}_{k-1},\mathbf{x}_{k},\mathbf{x}_{k+1},\mathbf{v}_{k},\mathbf{v}_{k+1})$ and 
$\hat{\mathbf{u}}_{k-1}=\widetilde{h}_{k}(\hat{\mathbf{u}}_{k-2},\hat{\mathbf{x}}_{k-1},\mathbf{x}_{k},\mathbf{v}_{k})$, where
\par\noindent\small
\begin{align}
&\widetilde{h}_{k}(\hat{\mathbf{u}}_{k-2},\hat{\mathbf{x}}_{k-1},\mathbf{x}_{k},\mathbf{v}_{k})\nonumber\\
&\;\;\;\;=\mathbf{K}^{u}_{k-1}(\mathbf{H}_{k}\mathbf{B}_{k-1}\hat{\mathbf{u}}_{k-2}-h(f(\hat{\mathbf{x}}_{k-1},\hat{\mathbf{u}}_{k-2}))+h(\mathbf{x}_{k})+\mathbf{v}_{k}),\label{eqn: state transition ekf without df input}\\
&\widetilde{f}_{k}(\hat{\mathbf{x}}_{k},\hat{\mathbf{u}}_{k-2},\hat{\mathbf{x}}_{k-1},\mathbf{x}_{k},\mathbf{x}_{k+1},\mathbf{v}_{k},\mathbf{v}_{k+1})\nonumber\\
&\;\;\;\;=\widetilde{\phi}_{k}(\hat{\mathbf{x}}_{k},\widetilde{h}_{k}(\hat{\mathbf{u}}_{k-2},\hat{\mathbf{x}}_{k-1},\mathbf{x}_{k},\mathbf{v}_{k}),\mathbf{x}_{k+1},\mathbf{v}_{k+1}).\label{eqn: state transition ekf without df state}
\end{align}
\normalsize

In these state transition equations, the actual states $\mathbf{x}_{k}$ and $\mathbf{x}_{k+1}$ are perfectly known to the defender and henceforth treated as known exogenous inputs. Note that, unlike the forward filter, the process noise terms $\mathbf{v}_{k}$ and $\mathbf{v}_{k+1}$ are non-additive because the filter gains $\mathbf{K}^{x}_{k+1}$ and $\mathbf{K}^{u}_{k-1}$ depend on the previous estimates (through the Jacobians).

Denote $\hat{\mathbf{z}}_{k+1} \doteq \begin{bmatrix}
\doublehat{\mathbf{x}}_{k+1}^{T} & \doublehat{\mathbf{u}}_{k-1}^{T}
\end{bmatrix}^{T}$. The state transition of the augmented state $\mathbf{z}_{k+1}$ depends on the estimate $\hat{\mathbf{x}}_{k-1}$ which the defender approximates by its previous estimate $\doublehat{\mathbf{x}}_{k-1}$. With this approximation, $\hat{\mathbf{x}}_{k-1}$ is treated as a known exogenous input for the inverse filter while the augmented process noise vector is $\begin{bmatrix}
\mathbf{v}_{k}^{T} & \mathbf{v}_{k+1}^{T}
\end{bmatrix}^{T}$. Define the Jacobians $\widetilde{\mathbf{F}}^{z}_{k}\doteq \begin{bmatrix}
\nabla_{\doublehat{\mathbf{x}}_{k}}\widetilde{f}_{k} & \nabla_{\doublehat{\mathbf{u}}_{k-2}}\widetilde{f}_{k}\\
\mathbf{0}_{m\times n} & \nabla_{\doublehat{\mathbf{u}}_{k-2}}\widetilde{h}_{k}
\end{bmatrix}$, and $\mathbf{G}_{k+1}\doteq\begin{bmatrix}
\nabla_{\doublehat{\mathbf{x}}_{k+1|k}}g & \mathbf{0}_{n_{a}\times m}\end{bmatrix}$ with respect to the augmented state; Jacobian $\widetilde{\mathbf{F}}^{v}_{k}\doteq\begin{bmatrix}
\nabla_{\mathbf{v}_{k}}\widetilde{f}_{k} & \nabla_{\mathbf{v}_{k+1}}\widetilde{f}_{k}\\
\nabla_{\mathbf{v}_{k}}\widetilde{h}_{k} & \mathbf{0}_{m\times p}
\end{bmatrix}$ with respect to the augmented process noise vector; and $\overline{\mathbf{Q}}_{k}=\widetilde{\mathbf{F}}^{v}_{k}\begin{bmatrix}
\mathbf{R} & \mathbf{0}_{p\times p}\\ \mathbf{0}_{p\times p} & \mathbf{R}
\end{bmatrix}(\widetilde{\mathbf{F}}^{v}_{k})^{T}$. Then, the I-EKF-without-DF's recursions yield the estimate $\hat{\mathbf{z}}_{k}$ of the augmented state and the associated covariance matrix $\overline{\bm{\Sigma}}_{k}$ as:
\par\noindent\small
\begin{align}
&\textit{Prediction:}\;\doublehat{\mathbf{x}}_{k+1|k}=\widetilde{f}_{k}(\doublehat{\mathbf{x}}_{k},\doublehat{\mathbf{u}}_{k-2},\doublehat{\mathbf{x}}_{k-1},\mathbf{x}_{k},\mathbf{x}_{k+1},\mathbf{0}_{p\times 1},\mathbf{0}_{p\times 1}),\nonumber\\
&\doublehat{\mathbf{u}}_{k-1|k}=\widetilde{h}_{k}(\doublehat{\mathbf{u}}_{k-2},\doublehat{\mathbf{x}}_{k-1},\mathbf{x}_{k},\mathbf{0}_{p\times 1}),\nonumber\\
&\hat{\mathbf{z}}_{k+1|k}=\begin{bmatrix}
\doublehat{\mathbf{x}}_{k+1|k}^{T} & \doublehat{\mathbf{u}}_{k-1|k}^{T}
\end{bmatrix}^{T},\nonumber\\
&\overline{\bm{\Sigma}}_{k+1|k}=\widetilde{\mathbf{F}}^{z}_{k}\overline{\bm{\Sigma}}_{k}(\widetilde{\mathbf{F}}^{z}_{k})^{T}+\overline{\mathbf{Q}}_{k},\label{eqn: I-EKF without DF covariance predict}\\
&\textit{Update:}\;\overline{\mathbf{S}}_{k+1}=\mathbf{G}_{k+1}\overline{\bm{\Sigma}}_{k+1|k}\mathbf{G}_{k+1}^{T}+\bm{\Sigma}_{\epsilon},\label{eqn: I-EKF without DF S compute}\\
&\hat{\mathbf{z}}_{k+1}=\hat{\mathbf{z}}_{k+1|k}+\overline{\bm{\Sigma}}_{k+1|k}\mathbf{G}_{k+1}^{T}\overline{\mathbf{S}}_{k+1}^{-1}\left(\mathbf{a}_{k+1}-g(\doublehat{\mathbf{x}}_{k+1|k})\right),\label{eqn: I-EKF without DF a predict}\\
&\overline{\bm{\Sigma}}_{k+1}=\overline{\bm{\Sigma}}_{k+1|k}-\overline{\bm{\Sigma}}_{k+1|k}\mathbf{G}_{k+1}^{T}\overline{\mathbf{S}}_{k+1}^{-1}\mathbf{G}_{k+1}\overline{\bm{\Sigma}}_{k+1|k}.\label{eqn: I-EKF without DF covariance update}
\end{align}
\normalsize

Fig.~\ref{fig:I-EKF-without-DF} provides a schematic diagram for  these updates. The I-EKF-without-DF's recursions take the same form as that of the standard EKF \cite{anderson2012optimal} but with modified system matrices. In particular, the former employs an augmented state such that the Jacobian of the state transition function with respect to the state is computed as $\widetilde{\mathbf{F}}^{z}_{k}$ while for the latter, it is simply $\mathbf{F}_{k}\doteq\nabla_{\mathbf{x}}f(\mathbf{x})\vert_{\mathbf{x}=\hat{\mathbf{x}}_{k}}$. Further, unlike standard KF or EKF, the noise terms, i.e., {$\mathbf{v}_{k}$ and $\mathbf{v}_{k+1}$ in \eqref{eqn: state transition ekf without df input} and \eqref{eqn: state transition ekf without df state} are non-additive such that linearization $\widetilde{\mathbf{F}}^{v}_{k}$ of the state transition function with respect to the noise terms yields the process noise covariance matrix approximation $\overline{\mathbf{Q}}_{k}$.
\begin{figure}
  \centering
  \includegraphics[width = 1.0\columnwidth]{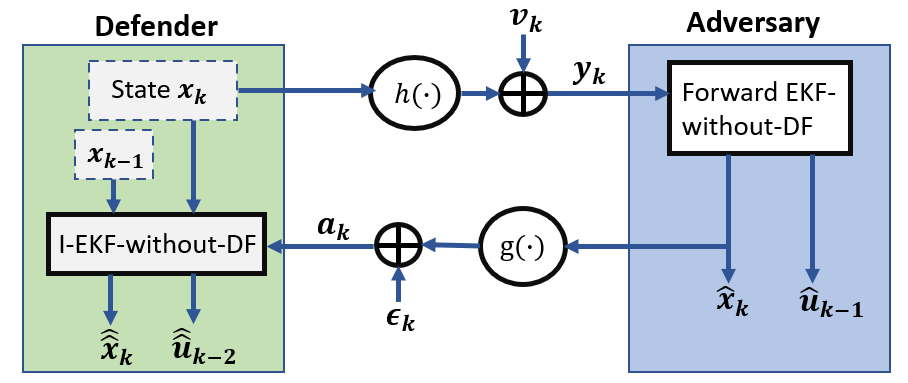}
  \caption{Graphical representation of I-EKF-without-DF recursion at $k$-th time step. The defender's true state at $k$-th and $(k-1)$-th time step are, respectively, $\mathbf{x}_{k}$ and $\mathbf{x}_{k-1}$. The adversary observes $\mathbf{x}_{k}$ as $\mathbf{y}_{k}$ through observation function $h(\cdot)$ with measurement noise $\mathbf{v}_{k}$. With $\mathbf{y}_{k}$ as input, adversary's forward EKF-without-DF computes state estimate $\hat{\mathbf{x}}_{k}$ and (one-step delayed) input estimate $\hat{\mathbf{u}}_{k-1}$. Defender observes $\hat{\mathbf{x}}_{k}$ as $\mathbf{a}_{k}$ through observation function $g(\cdot)$ with measurement noise $\bm{\epsilon}_{k}$. Finally, I-EKF-without-DF computes estimates $\doublehat{\mathbf{x}}_{k}$ and $\doublehat{\mathbf{u}}_{k-2}$ with $\mathbf{x}_{k-1}$, $\mathbf{x}_{k}$ and $\mathbf{a}_{k}$ as inputs.}
 \label{fig:I-EKF-without-DF}
\end{figure}

\begin{remark}
The forward filter gains $\mathbf{K}^{x}_{k+1}$ and $\mathbf{K}^{u}_{k-1}$ are treated as time-varying parameters of the state transition equation and not as a function of the state and input estimates ($\hat{\mathbf{x}}_{k}$ and $\hat{\mathbf{u}}_{k-1}$) in the inverse filter. The inverse filter approximates them by evaluating their values at its own estimates ($\doublehat{\mathbf{x}}_{k}$ and $\doublehat{\mathbf{u}}_{k-1}$) recursively in the similar manner as the forward filter evaluates them using its own estimates. On the contrary, in I-KF formulation introduced in \cite{krishnamurthy2019how}, the forward Kalman gain $\mathbf{K}_{k+1}$ is deterministic, fully determined by the model parameters for a given initial covariance estimate $\bm{\Sigma}_{0}$, and computed offline independent of the current I-KF's estimate.
\end{remark}

\subsection{I-EKF-with-DF unknown input}
\label{subsec:ekfwithdf}
Consider the non-linear system with DF given by \eqref{eqn: non x with input} and \eqref{eqn: non y withdf}. Linearize the functions as $\mathbf{F}_{k}\doteq\nabla_{\mathbf{x}}f(\mathbf{x},\hat{\mathbf{u}}_{k})|_{\mathbf{x}=\hat{\mathbf{x}}_{k}}$, $\mathbf{H}_{k+1}\doteq\nabla_{\mathbf{x}}h(\mathbf{x},\hat{\mathbf{u}}_{k})|_{\mathbf{x}=\hat{\mathbf{x}}_{k+1|k}}$ and $\mathbf{D}_{k}\doteq\nabla_{\mathbf{u}}h(\hat{\mathbf{x}}_{k+1|k},\mathbf{u})|_{\mathbf{u}=\hat{\mathbf{u}}_{k}}$.
\subsubsection{Forward filter}
Denote the state and input estimation covariance and gain matrices identical to Section~\ref{subsec:ekfwithoutdf}. Here, the current observation $\mathbf{y}_{k+1}$ depends on the current unknown input $\mathbf{u}_{k+1}$ such that the forward filter infers $\hat{\mathbf{u}}_{k+1}$ without any delay. For input estimation covariance without delay, we use $\bm{\Sigma}^{u}_{k+1}$. Then, the forward EKF-with-DF's recursions are \cite{yang2007adaptive}
\par\noindent\small
\begin{align}
&\textit{Prediction:}\;\hat{\mathbf{x}}_{k+1|k}=f(\hat{\mathbf{x}}_{k},\hat{\mathbf{u}}_{k}),\;\bm{\Sigma}^{x}_{k+1|k}=\mathbf{F}_{k}\bm{\Sigma}^{x}_{k}\mathbf{F}_{k}^{T}+\mathbf{Q},\label{eqn: ekfwithdf predict}\\
&\mathbf{K}^{x}_{k+1}=\bm{\Sigma}^{x}_{k+1|k}\mathbf{H}_{k+1}^{T}(\mathbf{H}_{k+1}\bm{\Sigma}^{x}_{k+1|k}\mathbf{H}_{k+1}^{T}+\mathbf{R})^{-1},\nonumber\\
&\bm{\Sigma}^{u}_{k+1}=\left(\mathbf{D}_{k}^{T}\mathbf{R}^{-1}(\mathbf{I}_{p\times p}-\mathbf{H}_{k+1}\mathbf{K}^{x}_{k+1})\mathbf{D}_{k}\right)^{-1},\nonumber\\
&\mathbf{K}^{u}_{k+1}=\bm{\Sigma}^{u}_{k+1}\mathbf{D}_{k}^{T}\mathbf{R}^{-1}(\mathbf{I}_{p\times p}-\mathbf{H}_{k+1}\mathbf{K}^{x}_{k+1}),\nonumber\\
&\textit{Update:}\;\hat{\mathbf{u}}_{k+1}=\mathbf{K}^{u}_{k+1}\left(\mathbf{y}_{k+1}-h(\hat{\mathbf{x}}_{k+1|k},\hat{\mathbf{u}}_{k})+\mathbf{D}_{k}\hat{\mathbf{u}}_{k}\right),\label{eqn: ekfwithdf update u}\\
&\hspace{-0.8mm}\hat{\mathbf{x}}_{k+1}=\hat{\mathbf{x}}_{k+1|k}+\mathbf{K}^{x}_{k+1}\left(\mathbf{y}_{k+1}-h(\hat{\mathbf{x}}_{k+1|k},\hat{\mathbf{u}}_{k})-\mathbf{D}_{k}(\hat{\mathbf{u}}_{k+1}-\hat{\mathbf{u}}_{k})\right),\label{eqn: ekfwithdf update x}\\
&\textit{Covariance matrix update:}\;\bm{\Sigma}^{x}_{k+1}=\bm{\Sigma}^{x}_{k+1|k}\nonumber\\
&\hspace{-0.8mm}\times(\mathbf{I}_{n\times n}+\mathbf{K}^{x}_{k+1}\mathbf{D}_{k}\bm{\Sigma}^{u}_{k+1}\mathbf{D}_{k}^{T}\mathbf{R}^{-1}\mathbf{H}_{k+1})(\mathbf{I}_{n\times n}-\mathbf{K}^{x}_{k+1}\mathbf{H}_{k+1}).\nonumber
\end{align}
\normalsize
The forward filter exists if $\textrm{rank}(\mathbf{D}_{k})=m$ for all $k\geq 0$, which implies $p\geq m$\cite{yang2007adaptive}.
\subsubsection{Inverse filter}
Consider an augmented state vector $\mathbf{z}_{k}=\begin{bmatrix}
\hat{\mathbf{x}}_{k}^{T} & \hat{\mathbf{u}}_{k}^{T}
\end{bmatrix}^{T}$  (note the absence of delay in the input estimate). Define $\widetilde{\phi}_{k}(\hat{\mathbf{x}}_{k},\hat{\mathbf{u}}_{k},\hat{\mathbf{u}}_{k+1},\mathbf{x}_{k+1},\mathbf{u}_{k+1},\mathbf{v}_{k+1})=f(\hat{\mathbf{x}}_{k},\hat{\mathbf{u}}_{k})-\mathbf{K}^{x}_{k+1}h(f(\hat{\mathbf{x}}_{k},\hat{\mathbf{u}}_{k}),\hat{\mathbf{u}}_{k})-\mathbf{K}^{x}_{k+1}\mathbf{D}_{k}(\hat{\mathbf{u}}_{k+1}-\hat{\mathbf{u}}_{k})+\mathbf{K}^{x}_{k+1}h(\mathbf{x}_{k+1},\mathbf{u}_{k+1})+\mathbf{K}^{x}_{k+1}\mathbf{v}_{k+1}$.
From \eqref{eqn: non y withdf} and \eqref{eqn: ekfwithdf predict}-\eqref{eqn: ekfwithdf update x}, state transitions for inverse filter are $\hat{\mathbf{x}}_{k+1}=\widetilde{f}_{k}(\hat{\mathbf{x}}_{k},\hat{\mathbf{u}}_{k},\mathbf{x}_{k+1},\mathbf{u}_{k+1},\mathbf{v}_{k+1})$ and 
$\hat{\mathbf{u}}_{k+1}=\widetilde{h}_{k}(\hat{\mathbf{x}}_{k},\hat{\mathbf{u}}_{k},\mathbf{x}_{k+1},\mathbf{u}_{k+1},\mathbf{v}_{k+1})$, 
where 
\par\noindent\small
\begin{align}
&\widetilde{h}_{k}(\hat{\mathbf{x}}_{k},\hat{\mathbf{u}}_{k},\mathbf{x}_{k+1},\mathbf{u}_{k+1},\mathbf{v}_{k+1})\nonumber\\
&=\mathbf{K}^{u}_{k+1}(h(\mathbf{x}_{k+1},\mathbf{u}_{k+1})+\mathbf{v}_{k+1}-h(f(\hat{\mathbf{x}}_{k},\hat{\mathbf{u}}_{k}),\hat{\mathbf{u}}_{k})+\mathbf{D}_{k}\hat{\mathbf{u}}_{k})\nonumber\\
&\widetilde{f}_{k}(\hat{\mathbf{x}}_{k},\hat{\mathbf{u}}_{k},\mathbf{x}_{k+1},\mathbf{u}_{k+1},\mathbf{v}_{k+1})\nonumber\\
&=\widetilde{\phi}_{k}(\hat{\mathbf{x}}_{k},\hat{\mathbf{u}}_{k},\widetilde{h}_{k}(\hat{\mathbf{x}}_{k},\hat{\mathbf{u}}_{k},\mathbf{x}_{k+1},\mathbf{u}_{k+1},\mathbf{v}_{k+1}),\mathbf{x}_{k+1},\mathbf{u}_{k+1},\mathbf{v}_{k+1}).\label{eqn:state transition ekf with df}
\end{align}
\normalsize

Then, \textit{ceteris paribus}, following similar steps as in I-EKF-without-DF, 
the I-EKF-with-DF estimate $\hat{\mathbf{z}}_{k}=\begin{bmatrix}
\doublehat{\mathbf{x}}_{k}^{T} & \doublehat{\mathbf{u}}_{k}^{T}
\end{bmatrix}^{T}$ from observations \eqref{eqn: non a} is computed recursively. The predicted augmented state is $\hat{\mathbf{z}}_{k+1|k}=\begin{bmatrix}
\doublehat{\mathbf{x}}_{k+1|k}^{T} & \doublehat{\mathbf{u}}_{k+1|k}^{T}
\end{bmatrix}^{T}$, where
$\doublehat{\mathbf{x}}_{k+1|k}=\widetilde{f}_{k}(\doublehat{\mathbf{x}}_{k},\doublehat{\mathbf{u}}_{k},\mathbf{x}_{k+1},\mathbf{u}_{k+1},\mathbf{0}_{p\times 1})$ and 
$\doublehat{\mathbf{u}}_{k+1|k}=\widetilde{h}_{k}(\doublehat{\mathbf{x}}_{k},\doublehat{\mathbf{u}}_{k},\mathbf{x}_{k+1},\mathbf{u}_{k+1},\mathbf{0}_{p\times 1})$. 
Hereafter, the remaining steps are as in \eqref{eqn: I-EKF without DF covariance predict}-\eqref{eqn: I-EKF without DF covariance update}. For I-EKF-with-DF, the Jacobians with respect to the augmented state are $\widetilde{\mathbf{F}}^{z}_{k}\doteq\begin{bmatrix}
\nabla_{\doublehat{\mathbf{x}}_{k}}\widetilde{f}_{k} & \nabla_{\doublehat{\mathbf{u}}_{k}}\widetilde{f}_{k} \\
\nabla_{\doublehat{\mathbf{x}}_{k}}\widetilde{h}_{k} & \nabla_{\doublehat{\mathbf{u}}_{k}}\widetilde{h}_{k}
\end{bmatrix}$ and $\mathbf{G}_{k+1}\doteq\begin{bmatrix}
\nabla_{\doublehat{\mathbf{x}}_{k+1|k}}g & \mathbf{0}_{n_{a}\times m}\end{bmatrix}$; the Jacobian with respect to the process noise term is $\widetilde{\mathbf{F}}^{v}_{k}\doteq\begin{bmatrix}
\nabla_{\mathbf{v}_{k+1}}\widetilde{f}_{k}\\
\nabla_{\mathbf{v}_{k+1}}\widetilde{h}_{k}
\end{bmatrix}$; and $\overline{\mathbf{Q}}_{k}=\widetilde{\mathbf{F}}^{v}_{k}\mathbf{R}(\widetilde{\mathbf{F}}^{v}_{k})^{T}$. Fig.~\ref{fig:I-EKF-with-DF} shows these updates graphically. Note that unlike I-EKF-without-DF, I-EKF-with-DF requires the true input $\mathbf{u}_{k}$ information. Here, unlike I-EKF-without-DF, the inverse filter's prediction dispenses with any approximation of $\hat{\mathbf{x}}_{k-1}$. The absence of delay in input estimation also results in a simplified process noise term $\mathbf{v}_{k+1}$, in place of I-EKF-without-DF's augmented noise vector.

Examples of EKF with unknown inputs include fault detection with unknown excitations \cite{yang2007adaptive} 
and missile-target interception with unknown target acceleration \cite{pan2010applying}. The inverse cognition in these applications would then resort to the I-EKFs described until now.
\begin{figure}
  \centering
  \includegraphics[width = 1.0\columnwidth]{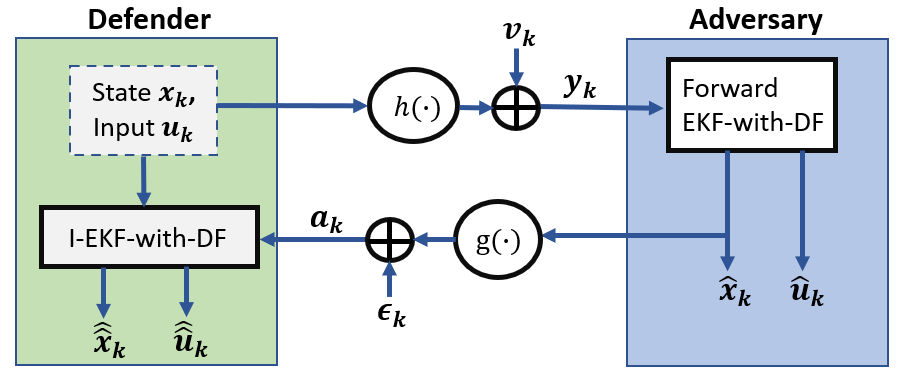}
  \caption{Graphical representation of I-EKF-with-DF recursion at $k$-th time step. The defender's true state and input (unknown to adversary) at $k$-th time step are $\mathbf{x}_{k}$ and $\mathbf{u}_{k}$, respectively. The adversary observes $(\mathbf{x}_{k},\mathbf{u}_{k})$ as $\mathbf{y}_{k}$ through observation function $h(\cdot)$ with measurement noise $\mathbf{v}_{k}$. With $\mathbf{y}_{k}$ as input, adversary's forward EKF-with-DF computes state estimate $\hat{\mathbf{x}}_{k}$ and input estimate $\hat{\mathbf{u}}_{k}$. Defender observes $\hat{\mathbf{x}}_{k}$ as $\mathbf{a}_{k}$ through observation function $g(\cdot)$ with measurement noise $\bm{\epsilon}_{k}$. Finally, I-EKF-with-DF computes estimates $\doublehat{\mathbf{x}}_{k}$ and $\doublehat{\mathbf{u}}_{k}$ with $\mathbf{x}_{k}$, $\mathbf{u}_{k}$ and $\mathbf{a}_{k}$ as inputs.}
 \label{fig:I-EKF-with-DF}
\end{figure}

\subsection{I-EKF without any unknown inputs}
\label{subsec:ekf}
Consider a non-linear system model without unknown inputs in the system equations \eqref{eqn: non x with input} and \eqref{eqn: non y withoutdf}, i.e.,
\par\noindent\small
\begin{align}
\mathbf{x}_{k+1}=f(\mathbf{x}_{k})+\mathbf{w}_{k}.\label{eqn: ekf x}
\end{align}
\normalsize
Linearize the functions as $\mathbf{F}_{k}\doteq\nabla_{\mathbf{x}}f(\mathbf{x})\vert_{\mathbf{x}=\hat{\mathbf{x}}_{k}}$ and 
$\mathbf{H}_{k+1}\doteq\nabla_{\mathbf{x}}h(\mathbf{x})\vert_{\mathbf{x}=\hat{\mathbf{x}}_{k+1|k}}$. 
Then, \textit{ceteris paribus}, setting $\mathbf{B}_{k}=\mathbf{0}_{n\times p}$ and neglecting computation of $\bm{\Sigma}^{u}_{k}$, $\mathbf{K}^{u}_{k}$ and $\hat{\mathbf{u}}_{k}$ in forward EKF-without-DF yields forward EKF-without-unknown-input whose state prediction and updates are
\par\noindent\small
\begin{align}
&\hat{\mathbf{x}}_{k+1|k}=f(\hat{\mathbf{x}}_{k}),\label{eqn: ekf predict}\\
&\hat{\mathbf{x}}_{k+1}=\hat{\mathbf{x}}_{k+1|k}+\mathbf{K}_{k+1}(\mathbf{y}_{k+1}-h(\hat{\mathbf{x}}_{k+1|k})),\label{eqn: ekf update}
\end{align}
\normalsize
with $\mathbf{K}_{k+1}=\bm{\Sigma}_{k+1|k}\mathbf{H}_{k+1}^{T}\left(\mathbf{H}_{k+1}\bm{\Sigma}_{k+1|k}\mathbf{H}_{k+1}^{T}+\mathbf{R}\right)^{-1}$. Here, we have dropped the superscript in the covariance matrix $\bm{\Sigma}^{x}_{k+1|k}$ and gain $\mathbf{K}^{x}_{k+1}$ to replace with $\bm{\Sigma}_{k+1|k}$ and $\mathbf{K}_{k+1}$, respectively (because only the state estimation covariances and gains are computed here). Thence, the I-EKF-without-DF's state transition equations and recursions yield I-EKF-without-unknown-input. Dropping the input estimate term in the augmented state $\mathbf{z}_{k}$, the state transition equations become
\par\noindent\small
\begin{align}
&\hat{\mathbf{x}}_{k+1}=\widetilde{f}_{k}(\hat{\mathbf{x}}_{k},\mathbf{x}_{k+1},\mathbf{v}_{k+1})\nonumber\\
&=f(\hat{\mathbf{x}}_{k})-\mathbf{K}_{k+1}h(f(\hat{\mathbf{x}}_{k}))+\mathbf{K}_{k+1}h(\mathbf{x}_{k+1})+\mathbf{K}_{k+1}\mathbf{v}_{k+1}.\label{eqn: inverse ekf state transition}
\end{align}
\normalsize
Denote $\widetilde{\mathbf{F}}^{x}_{k}\doteq\nabla_{\mathbf{x}}\widetilde{f}_{k}(\mathbf{x},\mathbf{x}_{k+1},\mathbf{0}_{p\times 1})|_{\mathbf{x}=\doublehat{\mathbf{x}}_{k}}$, $\mathbf{G}_{k+1}\doteq\nabla_{\mathbf{x}}g(\mathbf{x})_{\mathbf{x}=\doublehat{\mathbf{x}}_{k+1|k}}$, $\widetilde{\mathbf{F}}^{v}_{k}\doteq\nabla_{\mathbf{v}}\widetilde{f}_{k}(\doublehat{\mathbf{x}}_{k},\mathbf{x}_{k+1},\mathbf{v})|_{\mathbf{v}=\mathbf{0}_{p\times 1}}$, and $\overline{\mathbf{Q}}_{k}=\widetilde{\mathbf{F}}^{v}_{k}\mathbf{R}(\widetilde{\mathbf{F}}^{v}_{k})^{T}$. Then, the I-EKF's recursions are similar to I-EKF-without-DF except that the I-EKF's predicted state estimate and the associated prediction covariance matrix are computed, respectively, as
$\doublehat{\mathbf{x}}_{k+1|k}=\widetilde{f}_{k}(\doublehat{\mathbf{x}}_{k},\mathbf{x}_{k+1},\mathbf{0}_{p\times 1})$ and 
$\overline{\bm{\Sigma}}_{k+1|k}=\widetilde{\mathbf{F}}^{x}_{k}\overline{\bm{\Sigma}}_{k}(\widetilde{\mathbf{F}}^{x}_{k})^{T}+\overline{\mathbf{Q}}_{k}$,
followed by the update procedure in \eqref{eqn: I-EKF without DF S compute}-\eqref{eqn: I-EKF without DF covariance update}. 

Unlike I-KF \cite{krishnamurthy2019how}, the I-EKF approximates the forward gain $\mathbf{K}_{k+1}$ online at its own estimates recursively and is sensitive to the initial estimate of forward EKF's initial covariance matrix. I-EKF could be applied in various non-linear target tracking applications, where EKF is a popular forward filter\cite{ristic2003beyond}.

\begin{remark}\label{remark:non Gaussian}
So far, our system model considered the Gaussian process and measurement noises. To tackle the non-Gaussianity of the noise, Gaussian-sum EKF \cite{anderson2012optimal} and its inverse developed in our companion paper (Part II) \cite{singh2022inverse_part2} may be considered. Alternatively, one may employ the maximum correntropy criterion (MCC)-based filters\cite{cinar2012hidden}. For instance, the forward MCC-EKF in \cite{yang2019map} introduces a scalar ratio $d_{k+1}=G_{\sigma}(\|\mathbf{y}_{k+1}-h(\hat{\mathbf{x}}_{k+1|k})\|_{R^{-1}})/G_{\sigma}(\|\hat{\mathbf{x}}_{k+1|k}-f(\hat{\mathbf{x}}_{k})\|_{\bm{\Sigma}_{k+1|k}^{-1}})$, where $G_{\sigma}(\cdot)$ is the Gaussian kernel. The forward gain matrix $\mathbf{K}_{k+1}$ then becomes $\mathbf{K}_{k+1}=\bm{\Sigma}_{k+1|k}\mathbf{H}^{T}_{k+1}(\mathbf{H}_{k+1}\bm{\Sigma}_{k+1|k}\mathbf{H}^{T}_{k+1}+d_{k+1}^{-1}\mathbf{R})^{-1}$. The state prediction and update steps are the same as in forward EKF. While formulating the inverse filter, these modifications need to be taken into account in the inverse filter's state-transition equation. Also, I-EKF's gain matrix $\overline{\mathbf{K}}_{k+1}$ is similarly modified using $\overline{d}_{k+1}$ which is the counterpart of $d_{k+1}$ for the inverse filter's dynamics.
\end{remark}
\begin{remark}\label{remark:computation complexity}
Note that the proposed I-EKFs' recursions are obtained from that of a standard EKF, but with the state transition equation representing the evolution of the corresponding forward filter's state estimate. Hence, these I-EKFs have similar computational complexity as a standard EKF, i.e., $\mathcal{O}(d^3)$ where $d$ is the dimension of the estimated state vector\cite{daum2005nonlinear}. However, I-EKF with and without DF would estimate the augmented states $\mathbf{z}_{k}=[\hat{\mathbf{x}}_{k}^{T},\hat{\mathbf{u}}_{k-2}^{T}]^{T}$ and $\mathbf{z}_{k}=[\hat{\mathbf{x}}_{k}^{T},\hat{\mathbf{u}}_{k}^{T}]^{T}$, respectively. Hence, the overall computational complexity of these filters depends on the dimension of state $\mathbf{x}_{k}$ as well as the dimension of unknown inputs $\mathbf{u}_{k}$ in the system.
\end{remark}

The two-step prediction-update formulation (as discussed for EKF and I-EKF so far) infers an estimate of the current state. However, often for stability analyses, the one-step prediction formulation is analytically more useful. In this formulation, the estimate $\hat{\mathbf{x}}_{k}$ is  the one-step prediction estimate, i.e., an estimate of state $\mathbf{x}_{k}$ at $k$-th instant given the observations $\lbrace\mathbf{y}_{j}\rbrace_{1\leq j\leq k-1}$ up to time instant $k-1$ with $\bm{\Sigma}_{k}$ as the corresponding prediction covariance matrix. The forward one-step prediction EKF formulation\cite{reif1999stochastic} for the same system but with $\mathbf{F}_{k}\doteq\nabla_{\mathbf{x}}f(\mathbf{x})\vert_{\mathbf{x}=\hat{\mathbf{x}}_{k}}$ and $\mathbf{H}_{k}\doteq\nabla_{\mathbf{x}}h(\mathbf{x})\vert_{\mathbf{x}=\hat{\mathbf{x}}_{k}}$ is
\par\noindent\small
\begin{align}
&\mathbf{K}_{k}=\mathbf{F}_{k}\bm{\Sigma}_{k}\mathbf{H}_{k}^{T}(\mathbf{H}_{k}\bm{\Sigma}_{k}\mathbf{H}_{k}^{T}+\mathbf{R})^{-1},\label{eqn: one step forward ekf gain}\\
&\hat{\mathbf{x}}_{k+1}=f(\hat{\mathbf{x}}_{k})+\mathbf{K}_{k}(\mathbf{y}_{k}-h(\hat{\mathbf{x}}_{k})),\label{eqn: one step forward ekf update}\\
&\bm{\Sigma}_{k+1}=\mathbf{F}_{k}\bm{\Sigma}_{k}\mathbf{F}_{k}^{T}+\mathbf{Q}-\mathbf{K}_{k}(\mathbf{H}_{k}\bm{\Sigma}_{k}\mathbf{H}_{k}^{T}+\mathbf{R})\mathbf{K}_{k}^{T}.\label{eqn: one step forward ekf covariance}
\end{align}
\normalsize

From \eqref{eqn: non y withoutdf} and \eqref{eqn: one step forward ekf update}, the state transition equation for one-step formulation of I-EKF is $\hat{\mathbf{x}}_{k+1}=\widetilde{f}_{k}(\hat{\mathbf{x}}_{k},\mathbf{x}_{k},\mathbf{v}_{k})\doteq f(\hat{\mathbf{x}}_{k})-\mathbf{K}_{k}h(\hat{\mathbf{x}}_{k})+\mathbf{K}_{k}h(\mathbf{x}_{k})+\mathbf{K}_{k}\mathbf{v}_{k}$.
With this state transition, the I-EKF one-step prediction formulation follows directly from EKF's one-step prediction formulation treating $\mathbf{a}_{k}$ as the observation with the Jacobians with respect to state estimate $\widetilde{\mathbf{F}}^{x}_{k}=\nabla_{\mathbf{x}}\widetilde{f}_{k}(\mathbf{x},\mathbf{x}_{k},\mathbf{0})\vert_{\mathbf{x}=\doublehat{\mathbf{x}}_{k}}=\mathbf{F}_{k}-\mathbf{K}_{k}\mathbf{H}_{k}$ and $\mathbf{G}_{k}=\nabla_{\mathbf{x}}g(\mathbf{x})\vert_{\mathbf{x}=\doublehat{\mathbf{x}}_{k}}$, and the process noise covariance matrix $\overline{\mathbf{Q}}_{k}=\mathbf{K}_{k}\mathbf{R}\mathbf{K}_{k}^{T}$.

\vspace{-8pt}
\section{Inverse KF with unknown input}
\label{sec:kfunknown}
For linear Gaussian state-space models, our methods developed in the previous section are useful in extending the I-KF mentioned in \cite{krishnamurthy2019how} to unknown input. Again, the forward KFs employed by the adversary with and without DF are conceptually different \cite{gillijns2007kfb} because of the delay involved in input estimation. The forward KFs with unknown input provide unbiased minimum variance state and input estimates.
\subsection{I-KF-without-DF}
\label{subsec:kfwithoutdf}
Consider the 
system in \eqref{eqn: linear x with input} and \eqref{eqn: linear y withdf} with $\mathbf{D}=\mathbf{0}_{p\times m}$. 
\subsubsection{Forward filter}
Unlike EKF-without-DF, the forward KF-without-DF considers an intermediate state update step using the estimated unknown input before the final state updates. In this step, the unknown input is first estimated (with one-step delay) using the current observation $\mathbf{y}_{k+1}$ and input estimation gain matrix $\mathbf{M}_{k+1}$.  
In the update step, the current state estimate $\hat{\mathbf{x}}_{k+1}$ is computed by again considering the current observation $\mathbf{y}_{k+1}$ as\cite{gillijns2007unknownkf}
\par\noindent\small
\begin{align}
&\textit{Prediction:}\;\hat{\mathbf{x}}_{k+1|k}=\mathbf{F}\hat{\mathbf{x}}_{k},\;\bm{\Sigma}_{k+1|k}=\mathbf{F}\bm{\Sigma}_{k}\mathbf{F}^{T}+\mathbf{Q},\label{eqn: kfwithoutdf predict}\\
&\textit{Unknown input estimation:}\;\mathbf{S}_{k+1}=\mathbf{H}\bm{\Sigma}_{k+1|k}\mathbf{H}^{T}+\mathbf{R},\\
&\bm{M}_{k+1}=(\mathbf{B}^{T}\mathbf{H}^{T}\mathbf{S}_{k+1}^{-1}\mathbf{HB})^{-1}\mathbf{B}^{T}\mathbf{H}^{T}\mathbf{S}_{k+1}^{-1},\\
&\hat{\mathbf{u}}_{k}=\mathbf{M}_{k+1}(\mathbf{y}_{k+1}-\mathbf{H}\hat{\mathbf{x}}_{k+1|k}),\label{eqn: kfwithoutdf update u}\\
&\widetilde{\mathbf{x}}_{k+1|k+1}=\hat{\mathbf{x}}_{k+1|k}+\mathbf{B}\hat{\mathbf{u}}_{k},\label{eqn: kfwithoutdf update x with u}\\
&\widetilde{\bm{\Sigma}}_{k+1|k+1}=(\mathbf{I}_{n\times n}-\mathbf{BM}_{k+1}\mathbf{H})\bm{\Sigma}_{k+1|k}(\mathbf{I}_{n\times n}-\mathbf{BM}_{k+1}\mathbf{H})^{T}\nonumber\\
&\hspace{1.5cm}+\mathbf{BM}_{k+1}\mathbf{RM}_{k+1}^{T}\mathbf{B}^{T},\\
&\textit{Update:}\;\mathbf{K}_{k+1}=\bm{\Sigma}_{k+1|k}\mathbf{H}^{T}\mathbf{S}_{k+1}^{-1},\\
&\hat{\mathbf{x}}_{k+1}=\widetilde{\mathbf{x}}_{k+1|k+1}+\mathbf{K}_{k+1}(\mathbf{y}_{k+1}-\mathbf{H}\widetilde{\mathbf{x}}_{k+1|k+1}),\label{eqn: kfwithoutdf update x}\\
&\bm{\Sigma}_{k+1}=\widetilde{\bm{\Sigma}}_{k+1|k+1}-\mathbf{K}_{k+1}(\widetilde{\bm{\Sigma}}_{k+1|k+1}\mathbf{H}^{T}-\mathbf{BM}_{k+1}\mathbf{R})^{T}.\label{eqn: kfwithoutdf sigma update}
\end{align}
\normalsize
The forward filter exists if $\textrm{rank}(\mathbf{HB})=\textrm{rank}(\mathbf{B})=m$ which implies $n\geq m$ and $p\geq m$\cite{gillijns2007unknownkf}. Here, unlike I-EKFs, the gain matrices $\mathbf{K}_{k+1}$ and $\mathbf{M}_{k+1}$, are deterministic and completely determined by the model parameters and the initial covariance matrix similar to I-KF\cite{krishnamurthy2019how}.

\subsubsection{Inverse filter}
Denote $\widetilde{\mathbf{F}}_{k}=(\mathbf{I}_{n\times n}-\mathbf{K}_{k+1}\mathbf{H})(\mathbf{I}_{n\times n}-\mathbf{BM}_{k+1}\mathbf{H})\mathbf{F}$ and $\mathbf{E}_{k}=\mathbf{BM}_{k+1}-\mathbf{K}_{k+1}\mathbf{HBM}_{k+1}+\mathbf{K}_{k+1}$. From \eqref{eqn: linear y withdf} with $\mathbf{D}=\mathbf{0}_{p\times m}$, and \eqref{eqn: kfwithoutdf predict}-\eqref{eqn: kfwithoutdf update x}, the state transition equation for I-KF-without-DF is
\par\noindent\small
\begin{align}
\label{eqn: state for kfwithoutdf}
\hat{\mathbf{x}}_{k+1}=\widetilde{\mathbf{F}}_{k}\hat{\mathbf{x}}_{k}+\mathbf{E}_{k}\mathbf{Hx}_{k+1}+\mathbf{E}_{k}\mathbf{v}_{k+1}.
\end{align}
\normalsize
Unlike the state transition \eqref{eqn: state transition ekf without df input} and \eqref{eqn: state transition ekf without df state} of I-EKF-without-DF, the state transition for I-KF-without-DF is not an explicit function of the forward filter input estimate and hence, an augmented state is not needed. The difference arises from the forward EKF-without-DF, where the current input estimate explicitly depends on the previous input estimates as observed in \eqref{eqn: ekfwithoutdf update u}, which is not the case in KF-without-DF. The I-KF-without-DF's recursions with observation \eqref{eqn: linear a} 
are:
\par\noindent\small
\begin{align}
&\textit{Prediction:} \;
\doublehat{\mathbf{x}}_{k+1|k}=\widetilde{\mathbf{F}}_{k}\doublehat{\mathbf{x}}_{k}+\mathbf{E}_{k}\mathbf{Hx}_{k+1},\label{eqn: inverse kfwithoutdf state predict}\\
&\overline{\bm{\Sigma}}_{k+1|k}=\widetilde{\mathbf{F}}_{k}\overline{\bm{\Sigma}}_{k}\widetilde{\mathbf{F}}_{k}^{T}+\overline{\mathbf{Q}}_{k},\label{eqn: inverse kfwithoutdf covariance predict}\\
&\textit{Update:}\;\overline{\mathbf{S}}_{k+1}=\mathbf{G}\overline{\bm{\Sigma}}_{k+1|k}\mathbf{G}^{T}+\bm{\Sigma}_{\epsilon},\label{eqn: inverse kfwithoutdf gain}\\
&\doublehat{\mathbf{x}}_{k+1}=\doublehat{\mathbf{x}}_{k+1|k}+\overline{\bm{\Sigma}}_{k+1|k}\mathbf{G}^{T}\overline{\mathbf{S}}_{k+1}^{-1}(\mathbf{a}_{k+1}-\mathbf{G}\doublehat{\mathbf{x}}_{k+1|k}),\label{eqn: inverse kfwithoutdf state update}\\
&\overline{\bm{\Sigma}}_{k+1}=\overline{\bm{\Sigma}}_{k+1|k}-\overline{\bm{\Sigma}}_{k+1|k}\mathbf{G}^{T}\overline{\mathbf{S}}_{k+1}^{-1}\mathbf{G}\overline{\bm{\Sigma}}_{k+1|k},\label{eqn: inverse kfwithoutdf covariance update}
\end{align}
\normalsize
where (inverse) process noise covariance matrix $\overline{\mathbf{Q}}_{k}=\mathbf{E}_{k}\mathbf{R}\mathbf{E}_{k}^{T}$.

\subsection{I-KF-with-DF}
\label{subsec:kfwithdf}
Consider the linear system model with DF given by \eqref{eqn: linear x with input} and \eqref{eqn: linear y withdf}. 
\subsubsection{Forward filter}
Denote the state estimation covariance, input estimation (without delay) covariance, and cross-covariance of state and input estimates by $\bm{\Sigma}^{x}_{k}$, $\bm{\Sigma}^{u}_{k}$ and $\bm{\Sigma}^{xu}_{k}$, respectively. The forward KF-with-DF is  \cite{gillijns2007kfb}:
\par\noindent\small
\begin{align}
&\textit{Prediction:}\;\hat{\mathbf{x}}_{k+1|k}=\mathbf{F}\hat{\mathbf{x}}_{k}+\mathbf{B}\hat{\mathbf{u}}_{k},\label{eqn: kfwithdf predict}\\
&\bm{\Sigma}^{x}_{k+1|k}=\begin{bmatrix}
\mathbf{F} & \mathbf{B}
\end{bmatrix}\begin{bmatrix}
\bm{\Sigma}^{x}_{k} & \bm{\Sigma}^{xu}_{k}\\
\bm{\Sigma}^{ux}_{k} & \bm{\Sigma}^{u}_{k}
\end{bmatrix}\begin{bmatrix}
\mathbf{F}^{T}\\
\mathbf{B}^{T}
\end{bmatrix}+\mathbf{Q},\nonumber\\
&\textit{Gain computation:}\;\mathbf{S}_{k+1}=\mathbf{H}\bm{\Sigma}^{x}_{k+1|k}\mathbf{H}^{T}+\mathbf{R},\nonumber\\
&\mathbf{M}_{k+1}=(\mathbf{D}^{T}\mathbf{S}_{k+1}^{-1}\mathbf{D})^{-1}\mathbf{D}^{T}\mathbf{S}_{k+1}^{-1},\;\;\;\mathbf{K}_{k+1}=\bm{\Sigma}^{x}_{k+1|k}\mathbf{H}^{T}\mathbf{S}_{k+1}^{-1},\nonumber\\
&\textit{Update:}\;\hat{\mathbf{u}}_{k+1}=\mathbf{M}_{k+1}(\mathbf{y}_{k+1}-\mathbf{H}\hat{\mathbf{x}}_{k+1|k}),\label{eqn: kfwithdf update u}\\
&\hat{\mathbf{x}}_{k+1}=\hat{\mathbf{x}}_{k+1|k}+\mathbf{K}_{k+1}(\mathbf{y}_{k+1}-\mathbf{H}\hat{\mathbf{x}}_{k+1|k}-\mathbf{D}\hat{\mathbf{u}}_{k+1}),\label{eqn: kfwithdf update x}\\
&\textit{Covariance updates:}\;\bm{\Sigma}^{u}_{k+1}=(\mathbf{D}^{T}\mathbf{S}_{k+1}^{-1}\mathbf{D})^{-1},\nonumber\\
&\bm{\Sigma}^{x}_{k+1}=\bm{\Sigma}^{x}_{k+1|k}-\mathbf{K}_{k+1}(\mathbf{S}_{k+1}-\mathbf{D}\bm{\Sigma}^{u}_{k+1}\mathbf{D}^{T})\mathbf{K}_{k+1}^{T},\nonumber\\
&\bm{\Sigma}^{xu}_{k+1}=(\bm{\Sigma}^{ux}_{k+1})^{T}=-\mathbf{K}_{k+1}\mathbf{D}\bm{\Sigma}^{u}_{k+1}.\nonumber
\end{align}
\normalsize
The forward filter exists if $\textrm{rank}(\mathbf{D})=m$ (which implies $p\geq m$). 
\subsubsection{Inverse filter}
Consider an augmented state vector $\mathbf{z}_{k}=\begin{bmatrix}
\hat{\mathbf{x}}_{k}^{T} & \hat{\mathbf{u}}_{k}^{T}
\end{bmatrix}^{T}$. Denote $\widetilde{\mathbf{F}}_{k}=(\mathbf{I}_{n\times n}-\mathbf{K}_{k+1}\mathbf{H}+\mathbf{K}_{k+1}\mathbf{DM}_{k+1}\mathbf{H})\mathbf{F}$, $\widetilde{\mathbf{B}}_{k}=(\mathbf{I}_{n\times n}-\mathbf{K}_{k+1}\mathbf{H}+\mathbf{K}_{k+1}\mathbf{DM}_{k+1}\mathbf{H})\mathbf{B}$, $\mathbf{E}_{k}=\mathbf{K}_{k+1}(\mathbf{I}_{p\times p}-\mathbf{DM}_{k+1})$, $\widetilde{\mathbf{H}}_{k}=-\mathbf{M}_{k+1}\mathbf{HF}$ and $\widetilde{\mathbf{D}}_{k}=-\mathbf{M}_{k+1}\mathbf{HB}$. From \eqref{eqn: linear y withdf}, and \eqref{eqn: kfwithdf predict}-
\eqref{eqn: kfwithdf update x}, the  state transition equations for I-KF-with-DF are
\par\noindent\small
\begin{align*}
&\hat{\mathbf{x}}_{k+1}=\widetilde{\mathbf{F}}_{k}\hat{\mathbf{x}}_{k}+\widetilde{\mathbf{B}}_{k}\hat{\mathbf{u}}_{k}+\mathbf{E}_{k}\mathbf{H}\mathbf{x}_{k+1}+\mathbf{E}_{k}\mathbf{D}\mathbf{u}_{k+1}+\mathbf{E}_{k}\mathbf{v}_{k+1},
\end{align*}\normalsize
and\par\noindent\small
\begin{align*}
&\hat{\mathbf{u}}_{k+1}\nonumber\\
&=\widetilde{\mathbf{H}}_{k}\hat{\mathbf{x}}_{k}+\widetilde{\mathbf{D}}_{k}\hat{\mathbf{u}}_{k}+\mathbf{M}_{k+1}\mathbf{H}\mathbf{x}_{k+1}+\mathbf{M}_{k+1}\mathbf{D}\mathbf{u}_{k+1}+\mathbf{M}_{k+1}\mathbf{v}_{k+1}.
\end{align*}
\normalsize
Also, $\begin{bmatrix}
(\mathbf{E}_{k}\mathbf{v}_{k+1})^{T} & (\mathbf{M}_{k+1}\mathbf{v}_{k+1})^{T}
\end{bmatrix}^{T}$ is the augmented noise vector involved in this state transition with noise covariance matrix $\overline{\mathbf{Q}}_{k}=\begin{bmatrix}
\mathbf{E}_{k}\mathbf{R}\mathbf{E}_{k}^{T} & \mathbf{E}_{k}\mathbf{R}\mathbf{M}_{k+1}^{T}\\
\mathbf{M}_{k+1}\mathbf{R}\mathbf{E}_{k}^{T} & \mathbf{M}_{k+1}\mathbf{R}\mathbf{M}_{k+1}^{T}
\end{bmatrix}$. Then, \textit{ceteris paribus}, following similar steps as in I-KF-without-DF, the I-KF-with-DF computes the estimate $\hat{\mathbf{z}}_{k}=\begin{bmatrix}
\doublehat{\mathbf{x}}_{k}^{T} & \doublehat{\mathbf{u}}_{k}^{T}
\end{bmatrix}^{T}$ of the augmented state vector using the observation $\mathbf{a}_{k}$ given by \eqref{eqn: linear a}. The system matrices for the augmented state are $\widetilde{\mathbf{F}}^{z}_{k}=\begin{bmatrix}
\widetilde{\mathbf{F}}_{k} & \widetilde{\mathbf{B}}_{k}\\
\widetilde{\mathbf{H}}_{k} & \widetilde{\mathbf{D}}_{k}
\end{bmatrix}$ and $\overline{\mathbf{G}}=\begin{bmatrix}
\mathbf{G} & \mathbf{0}_{n_{a}\times m}
\end{bmatrix}$. The I-KF-with-DF predicts the augmented state as
\par\noindent\small
\begin{align*}
&\doublehat{\mathbf{x}}_{k+1|k}=\widetilde{\mathbf{F}}_{k}\doublehat{\mathbf{x}}_{k}+\widetilde{\mathbf{B}}_{k}\doublehat{\mathbf{u}}_{k}+\mathbf{E}_{k}\mathbf{Hx}_{k+1}+\mathbf{E}_{k}\mathbf{Du}_{k+1},\\
&\doublehat{\mathbf{u}}_{k+1|k}=\widetilde{\mathbf{H}}_{k}\doublehat{\mathbf{x}}_{k}+\widetilde{\mathbf{D}}_{k}\doublehat{\mathbf{u}}_{k}+\mathbf{M}_{k+1}\mathbf{Hx}_{k+1}+\mathbf{M}_{k+1}\mathbf{Du}_{k+1},\\
&\hat{\mathbf{z}}_{k+1|k}=\begin{bmatrix}
\doublehat{\mathbf{x}}_{k+1|k}^{T} & \doublehat{\mathbf{u}}_{k+1|k}^{T}
\end{bmatrix}^{T},\;
\overline{\bm{\Sigma}}_{k+1|k}=\widetilde{\mathbf{F}}^{z}_{k}\overline{\bm{\Sigma}}_{k}(\widetilde{\mathbf{F}}^{z}_{k})^{T}+\overline{\mathbf{Q}}_{k},
\end{align*}
\normalsize
followed by the update procedure \eqref{eqn: inverse kfwithoutdf gain}-\eqref{eqn: inverse kfwithoutdf covariance update} with $\mathbf{G}$ and $\doublehat{x}_{k+1}$ replaced by $\overline{\mathbf{G}}$ and $\hat{\mathbf{z}}_{k+1}$, respectively.

\begin{remark}
Since the observation $\mathbf{y}_{k}$ explicitly depends on the unknown input $\mathbf{u}_{k}$ for a system with DF, I-KF-with-DF and I-EKF-with-DF require perfect knowledge of the current input $\mathbf{u}_{k}$ as a known exogenous input to obtain their state and input estimates, which is not the case in I-KF-without-DF and I-EKF-without-DF.
\end{remark}
\begin{remark}\label{remark:unknown input diff}
Note that the I-KFs with unknown inputs for linear system models are not special cases of I-EKFs with unknown inputs for non-linear system models. In I-KF-with-unknown-inputs, the adversary employs a forward KF which provides unbiased minimum variance estimates of the state and the unknown inputs\cite{gillijns2007unknownkf,gillijns2007kfb}. On the other hand, in I-EKF-with-unknown-inputs, the forward filter's state and unknown inputs estimates are computed based on a weighted least squared error criterion\cite{pan2010applying,yang2007adaptive}. The different forward filters employed by the adversary results in different inverse filters for the defender to estimate the adversary's state estimate.
\end{remark}
\begin{remark}\label{remark:linear complexity}
As mentioned in Remark~\ref{remark:computation complexity}, the I-KFs with unknown inputs are also derived from standard KF recursions and have similar computational complexity. However, I-KF-without-DF is not formulated using an augmented state and hence, is computationally less complex than I-KF-with-DF.
\end{remark}

\vspace{-8pt}
\section{Performance Analyses}
\label{sec:stability}
For continuous-time non-linear Kalman filtering, some convergence results were mentioned in \cite{krener2003convergenceOfEKF}. In case of EKF, sufficient conditions for stability of non-linear systems with linear output map were described in \cite{la1995conditionsforEKFforfreq}. Recently, the stability of deterministic EKF was studied based on contraction theory in \cite{bonnabel2014contraction}. The asymptotic convergence of EKF for a special class of systems, where EKF is applied for joint state and parameter estimation of linear stochastic systems, was studied in \cite{ljung1979asymptotic,ursin1980asymptotic}. If the non-linearities have known bounds, then the Riccati equation is slightly modified 
to guarantee stability for the continuous-time EKF \cite{reif1998ekf}.

To derive the sufficient conditions for stochastic stability of non-linear filters, 
one of the common approaches is to introduce unknown instrumental matrices to account for the linearization errors \cite{xiong2006performance_ukf}. It does not assume any bound on the estimation error, but its sufficient conditions for stability, especially the bounds assumed on the unknown matrices, are difficult to verify for practical systems.

Alternatively, \cite{reif1999stochastic} considers the one-step prediction formulation of the filter and provides sufficient conditions under which the state prediction error is \textit{exponentially bounded in mean-squared} sense. We restate some definitions and a useful Lemma from \cite{reif1999stochastic}.

\begin{definition}[Exponential mean-squared boundedness \cite{reif1999stochastic}]\label{defn:exponential boundedness} A stochastic process $\{\bm{\zeta}_{k} \}_{k \geq 0}$ is defined to be exponentially bounded in mean-squared sense if there are real numbers $\eta,\nu>0$ and $0<\lambda<1$ such that 
$\mathbb{E}\left[\|\bm{\zeta}_{k}\|_{2}^{2}\right]\leq \eta\mathbb{E}\left[\|\bm{\zeta}_{0}\|_{2}^{2}\right]\lambda^{k}+\nu$  
holds for every $k\geq 0$.
\end{definition}
\begin{definition}[Boundedness with probability one \cite{reif1999stochastic}] A stochastic process $\{\bm{\zeta}_{k} \}_{k \geq 0}$  is defined to be bounded with probability one if 
$\sup_{k\geq 0}\|\bm{\zeta}_{k}\|_{2} < \infty$ 
holds with probability one.  
\end{definition}
\begin{lemma}[Boundedness of stochastic process {\cite[Lemma 2.1]{reif1999stochastic}}]
\label{lemma:exponential boundedness}
Consider a function $V_{k}(\bm{\zeta}_{k})$ of the stochastic process $\bm{\zeta}_{k}$ and real numbers $v_{\textrm{min}}$, $v_{\textrm{max}}$, $\mu>0$, and $0<\lambda\leq 1$ such that for all $k\geq 0$
\par\noindent\small
\begin{align*}
v_{\textrm{min}}\|\bm{\zeta}_{k}\|_{2}^{2}\leq V_{k}(\bm{\zeta}_{k})\leq v_{\textrm{max}}\|\bm{\zeta}_{k}\|_{2}^{2},
\end{align*}
and
\begin{align*}
\mathbb{E}\left[ V_{k+1}(\bm{\zeta}_{k+1})|\bm{\zeta}_{k}\right]-V_{k}(\bm{\zeta}_{k})\leq\mu-\lambda V_{k}(\bm{\zeta}_{k}).
\end{align*}
\normalsize
Then, the stochastic process $\{\bm{\zeta}_{k}\}_{k \geq 0}$ is exponentially bounded in mean-squared sense, i.e.,
\par\noindent\small
\begin{align*}
\mathbb{E}\left[\|\bm{\zeta}_{k}\|_{2}^{2}\right]\leq\frac{v_{\textrm{max}}}{v_{\textrm{min}}}\mathbb{E}\left[\|\bm{\zeta}_{0}\|_{2}^{2}\right](1-\lambda)^{k}+\frac{\mu}{v_{\textrm{min}}}\sum_{i=1}^{k-1}(1-\lambda)^{i},
\end{align*}
\normalsize
for every $k\geq 0$. Further, $\{\bm{\zeta}_{k}\}_{k \geq 0}$ is also bounded with probability one.
\end{lemma}
\begin{remark}\label{remark:reif result remark}
In the bounded mean-squared sense, \cite[Sec. III]{reif1999stochastic} showed that, while the two-step prediction and update recursion (described in previous sections) and one-step formulation of (forward) filters may differ in their performance and transient behaviour, they have similar convergence properties. 
However, the conditions of Lemma \ref{lemma:exponential boundedness} were proved to hold when the error remained within suitable bounds; the guarantees fail if the error exceeds this bound at any instant. However, it was numerically shown \cite[Sec. V]{reif1999stochastic} that the bound on the error was only of theoretical interest and, in practice, the filter remained stable for much larger estimation errors.
\end{remark}

In the following, we first derive stability conditions for I-KF-without-DF in which we rely on the stability of the forward KF-without-DF as proved in \cite{fang2012on}. The procedure is similar for the stability of I-KF-with-DF and I-KF-without-unknown-input \cite{krishnamurthy2019how} and hence, we omit the details for these filters. For I-EKF stability, we employ both unknown matrix and bounded non-linearity approaches. In the process, we also derive the forward EKF stability conditions using unknown matrix approach; note that the same was obtained using bounded non-linearity method in \cite{reif1999stochastic}. Finally, we provide conditions for the consistency of the I-EKF’s estimates. The procedure is similar for the consistency of other proposed filters considering their respective augmented states and hence, we omit the details here.

\subsection{I-KF-with-unknown-input}
\label{subsec:kfstable}
Consider I-KF-without-DF of Section \ref{subsec:kfwithoutdf}, where the forward filter is asymptotically stable under the sufficient conditions provided by \cite{fang2012on}. The following Theorem~\ref{theorem: inverse kf without DF} states conditions for stability of the inverse filter.
\begin{theorem} [Stability of I-KF-without-DF]
\label{theorem: inverse kf without DF}
Consider an asymptotically stable forward KF-without-DF \eqref{eqn: kfwithoutdf predict}-\eqref{eqn: kfwithoutdf sigma update} such that the gain matrices $\mathbf{M}_{k}$ and $\mathbf{K}_{k}$ asymptotically approach to limiting gain matrices $\overline{\mathbf{M}}$ and $\overline{\mathbf{K}}$, respectively. The measurement noise covariance matrix $\bm{\Sigma}_{\epsilon}$ is positive definite (p.d.). Denote the limiting matrices 
$\overline{\mathbf{F}}=(\mathbf{I}-\overline{\mathbf{K}}\mathbf{H})(\mathbf{I}-\mathbf{B}\overline{\mathbf{M}}\mathbf{H})\mathbf{F}$ and $\overline{\mathbf{Q}}=\overline{\mathbf{E}}\mathbf{R}\overline{\mathbf{E}}^{T}$, where $\overline{\mathbf{E}}=\mathbf{B}\overline{\mathbf{M}}-\overline{\mathbf{K}}\mathbf{HB}\overline{\mathbf{M}}+\overline{\mathbf{K}}$. Then, the I-KF-without-DF \eqref{eqn: inverse kfwithoutdf state predict}-\eqref{eqn: inverse kfwithoutdf covariance update} is asymptotically stable under the assumption that pair ($\overline{\mathbf{F}}$,$\mathbf{G}$) is observable and the pair ($\overline{\mathbf{F}}$,$\mathbf{C}$) is controllable for the system given by \eqref{eqn: linear a} and \eqref{eqn: state for kfwithoutdf}, where $\mathbf{C}$ is such that $\overline{\mathbf{Q}}=\mathbf{C}^{T}\mathbf{C}$. 
\end{theorem}

\begin{proof}
See Appendix~\ref{App-thm-inverse kf without DF}.
\end{proof}

Note that, for I-KF-with-DF's stability, the stability conditions of basic KF need to hold for the augmented state considered in inverse filter formulation of Section \ref{subsec:kfwithdf}. For forward KF-with-DF's stability conditions, we refer the reader to \cite{fang2012on}.

\subsection{I-EKF-without-unknown-input: Unknown matrix approach}
\label{subsec:ekf stable unknown}

Consider the I-EKF's two-step prediction and update formulation of Section~\ref{subsec:ekf}, with forward filter as EKF-without-unknown-input.

\subsubsection{Forward EKF stability} 
Denote the forward EKF's state prediction, state estimation and measurement prediction errors by $\widetilde{\mathbf{x}}_{k+1|k}\doteq\mathbf{x}_{k+1}-\hat{\mathbf{x}}_{k+1|k}$, $\widetilde{\mathbf{x}}_{k}\doteq\mathbf{x}_{k}-\hat{\mathbf{x}}_{k}$ and $\widetilde{\mathbf{y}}_{k}\doteq\mathbf{y}_{k}-\hat{\mathbf{y}}_{k}$, with $\hat{\mathbf{y}}_{k}=h(\hat{\mathbf{x}}_{k|k-1})$, respectively. Using \eqref{eqn: ekf x}, \eqref{eqn: ekf predict} and the Taylor series expansion of $f(\cdot)$ at $\hat{\mathbf{x}}_{k}$, we get
\par\noindent\small
\begin{align*}
    \widetilde{\mathbf{x}}_{k+1|k}=\mathbf{F}_{k}(\mathbf{x}_{k}-\hat{\mathbf{x}}_{k})+\mathbf{w}_{k}+\mathcal{O}(\|\mathbf{x}_{k}-\hat{\mathbf{x}}_{k}\|_{2}^{2})\approx\mathbf{F}_{k}\widetilde{\mathbf{x}}_{k}+\mathbf{w}_{k}.
\end{align*}
\normalsize
We consider the general case of time-varying process and measurement noise covariances and denote $\mathbf{Q}$, $\mathbf{R}$ and $\bm{\Sigma}_{\epsilon}$ by $\mathbf{Q}_{k}$, $\mathbf{R}_{k}$ and $\overline{\mathbf{R}}_{k}$, respectively. To account for the residuals and obtain an exact equality, we introduce an unknown instrumental diagonal matrix $\mathbf{U}^{x}_{k} \in\mathbb{R}^{n\times n}$\cite{xiong2006performance_ukf,li2012stochastic_ukf} as
\par\noindent\small
\begin{align}
    \widetilde{\mathbf{x}}_{k+1|k}=\mathbf{U}^{x}_{k}\mathbf{F}_{k}\widetilde{\mathbf{x}}_{k}+\mathbf{w}_{k}\label{eqn:forward EKF predict error with alpha}.
\end{align}
\normalsize
However, using \eqref{eqn: ekf update}, we have $\widetilde{\mathbf{x}}_{k}=\widetilde{\mathbf{x}}_{k|k-1}-\mathbf{K}_{k}\widetilde{\mathbf{y}}_{k}$, which when substituted in \eqref{eqn:forward EKF predict error with alpha} yields $\widetilde{\mathbf{x}}_{k+1|k}=\mathbf{U}^{x}_{k}\mathbf{F}_{k}\widetilde{\mathbf{x}}_{k|k-1}-\mathbf{U}^{x}_{k}\mathbf{F}_{k}\mathbf{K}_{k}\widetilde{\mathbf{y}}_{k}+\mathbf{w}_{k}$.
Similarly, using Taylor series expansion of $h(\cdot)$ at $\hat{\mathbf{x}}_{k+1|k}$ in \eqref{eqn: non y withoutdf} and introducing an unknown diagonal matrix $\mathbf{U}^{y}_{k+1}\in\mathbb{R}^{p\times p}$ gives $\widetilde{\mathbf{y}}_{k+1}=\mathbf{U}^{y}_{k+1}\mathbf{H}_{k+1}\widetilde{\mathbf{x}}_{k+1|k}+\mathbf{v}_{k+1}$.
The prediction error dynamics of the forward EKF becomes
\par\noindent\small
\begin{align}
    \widetilde{\mathbf{x}}_{k+1|k}=\mathbf{U}^{x}_{k}\mathbf{F}_{k}(\mathbf{I}-\mathbf{K}_{k}\mathbf{U}^{y}_{k}\mathbf{H}_{k})\widetilde{\mathbf{x}}_{k|k-1}-\mathbf{U}^{x}_{k}\mathbf{F}_{k}\mathbf{K}_{k}\mathbf{v}_{k}+\mathbf{w}_{k}.\label{eqn:forward EKF prediction error dynamics}
\end{align}
\normalsize

Denote the true prediction covariance by  $\mathbf{P}_{k+1|k}=\mathbb{E}\left[\widetilde{\mathbf{x}}_{k+1|k}\widetilde{\mathbf{x}}_{k+1|k}^{T}\right]$. Define $\delta\mathbf{P}_{k+1|k}$ as the difference of estimated prediction covariance $\bm{\Sigma}_{k+1|k}$ and the true prediction covariance $\mathbf{P}_{k+1|k}$ while $\Delta\mathbf{P}_{k+1|k}$ as the error in the approximation of the expectation
\small
$\mathbb{E}\left[\mathbf{U}^{x}_{k}\mathbf{F}_{k}(\mathbf{I}-\mathbf{K}_{k}\mathbf{U}^{y}_{k}\mathbf{H}_{k})\widetilde{\mathbf{x}}_{k|k-1}\widetilde{\mathbf{x}}_{k|k-1}^{T}(\mathbf{I}-\mathbf{K}_{k}\mathbf{U}^{y}_{k}\mathbf{H}_{k})^{T}\mathbf{F}_{k}^{T}\mathbf{U}^{x}_{k}\right]$
\normalsize
by $\mathbf{U}^{x}_{k}\mathbf{F}_{k}(\mathbf{I}-\mathbf{K}_{k}\mathbf{U}^{y}_{k}\mathbf{H}_{k})\bm{\Sigma}_{k|k-1}(\mathbf{I}-\mathbf{K}_{k}\mathbf{U}^{y}_{k}\mathbf{H}_{k})^{T}\mathbf{F}_{k}^{T}\mathbf{U}^{x}_{k}$. Denoting $\hat{\mathbf{Q}}_{k}=\mathbf{Q}_{k}+\mathbf{U}^{x}_{k}\mathbf{F}_{k}\mathbf{K}_{k}\mathbf{R}_{k}\mathbf{K}_{k}^{T}\mathbf{F}_{k}^{T}\mathbf{U}^{x}_{k}+\delta\mathbf{P}_{k+1|k}+\Delta\mathbf{P}_{k+1|k}$ and following similar steps as in \cite{xiong2006performance_ukf, li2012stochastic_ukf}, we have
\par\noindent\small
\begin{align*}
    &\bm{\Sigma}_{k+1|k}=\\
    &\mathbf{U}^{x}_{k}\mathbf{F}_{k}(\mathbf{I}-\mathbf{K}_{k}\mathbf{U}^{y}_{k}\mathbf{H}_{k})\bm{\Sigma}_{k|k-1}(\mathbf{I}-\mathbf{K}_{k}\mathbf{U}^{y}_{k}\mathbf{H}_{k})^{T}\mathbf{F}_{k}^{T}\mathbf{U}^{x}_{k}+\hat{\mathbf{Q}}_{k}.
\end{align*}
\normalsize

Similarly, denoting the true measurement prediction covariance and true cross-covariance by $\mathbf{P}^{yy}_{k+1}$ and $\mathbf{P}^{xy}_{k+1}$, respectively, we obtain
\par\noindent\small
\begin{align*}
    \mathbf{S}_{k+1}&=\mathbf{U}^{y}_{k+1}\mathbf{H}_{k+1}\bm{\Sigma}_{k+1|k}\mathbf{H}_{k+1}^{T}\mathbf{U}^{y}_{k+1}+\hat{\mathbf{R}}_{k+1},\\
    \bm{\Sigma}^{xy}_{k+1}&=\begin{cases}\bm{\Sigma}_{k+1|k}\mathbf{U}^{xy}_{k+1}\mathbf{H}_{k+1}^{T}\mathbf{U}^{y}_{k+1}, & n\geq p\\
    \bm{\Sigma}_{k+1|k}\mathbf{H}_{k+1}^{T}\mathbf{U}^{y}_{k+1}\mathbf{U}^{xy}_{k+1}, & n<p\end{cases},
\end{align*}
\normalsize
where $\hat{\mathbf{R}}_{k+1}=\mathbf{R}_{k+1}+\Delta\mathbf{P}^{yy}_{k+1}+\delta\mathbf{P}^{yy}_{k+1}$ and $\mathbf{U}^{xy}_{k+1}$ is an unknown instrumental matrix introduced to account for errors in the estimated cross-covariance $\bm{\Sigma}_{k+1}^{xy}$\cite{xiong2007authorreply}.

The following Theorem~\ref{theorem:Forward ekf stable unknown matrix} provides stability conditions for the forward EKF using the unknown matrices $\mathbf{U}^{x}_{k}$, $\mathbf{U}^{y}_{k}$ and $\mathbf{U}^{xy}_{k}$.
\begin{theorem}[Stochastic stability of forward EKF]
\label{theorem:Forward ekf stable unknown matrix}
Consider the non-linear stochastic system in \eqref{eqn: ekf x} and \eqref{eqn: non y withoutdf}. The two-step forward EKF formulation is as in Section~\ref{subsec:ekf}. Let the following assumptions hold true:
\begin{enumerate}
    \item There exist positive real numbers $\overline{f}$, $\overline{h}$, $\overline{\alpha}$, $\overline{\beta}$, $\overline{\gamma}$, $\underline{\sigma}$, $\overline{\sigma}$, $\overline{q}$, $\overline{r}$, $\hat{q}$ and $\hat{r}$ such that the following bounds are fulfilled for all $k\geq 0$.
    \par\noindent\small
    \begin{align*}
        \|\mathbf{F}_{k}\|\leq\overline{f},&\hspace{0.4cm}\|\mathbf{H}_{k}\|\leq\overline{h},\hspace{0.4cm}\|\mathbf{U}^{x}_{k}\|\leq\overline{\alpha}, \hspace{0.4cm}\|\mathbf{U}^{y}_{k}\|\leq\overline{\beta},\\
        \|\mathbf{U}^{xy}_{k}\|\leq\overline{\gamma},&\hspace{0.4cm}\mathbf{Q}_{k}\preceq\overline{q}\mathbf{I},\hspace{0.5cm}     \mathbf{R}_{k}\preceq\overline{r}\mathbf{I},\hspace{0.6cm}\hat{q}\mathbf{I}\preceq\hat{\mathbf{Q}}_{k},\\
        \hat{r}\mathbf{I}\preceq\hat{\mathbf{R}}_{k},&\hspace{0.4cm}
        \underline{\sigma}\mathbf{I}\preceq\bm{\Sigma}_{k|k-1}\preceq\overline{\sigma}\mathbf{I}.
    \end{align*}
    \normalsize
    \item $\mathbf{U}^{x}_{k}$ and $\mathbf{F}_{k}$ are non-singular for every $k\geq 0$.
\end{enumerate}
Then, the prediction error $\widetilde{\mathbf{x}}_{k|k-1}$ and the estimation error $\widetilde{\mathbf{x}}_{k}$ of the forward EKF are exponentially bounded in mean-squared sense and bounded with probability one provided that the constants satisfy the inequality
\par\noindent\small
\begin{align}
    \overline{\sigma}\overline{\gamma}\overline{h}^{2}\overline{\beta}^{2}<\hat{r}.\label{eqn:inequality on constants}
\end{align}
\normalsize
\end{theorem}

\begin{proof}
See Appendix~\ref{App-thm-Forward ekf stable unknown matrix}.
\end{proof}

\subsubsection{Inverse EKF stability} 
For a stable forward EKF in the previous subsection, we prove the stochastic stability of the I-EKF as an extension of Theorem \ref{theorem:Forward ekf stable unknown matrix}. Similar to the forward EKF, we introduce unknown matrices $\overline{\mathbf{U}}^{x}_{k}$ and $\overline{\mathbf{U}}^{a}_{k}$ to account for the errors in the linearization of functions $\widetilde{f}_{k}(\cdot)$ and $g(\cdot)$, respectively, and $\overline{\mathbf{U}}^{xa}_{k}$ for the errors in cross-covariance matrix estimation. Similarly, denote $\hat{\overline{\mathbf{Q}}}_{k}$ and $\hat{\overline{\mathbf{R}}}_{k}$ as the counterparts of $\hat{\mathbf{Q}}_{k}$ and $\hat{\mathbf{R}}_{k}$, respectively, in the I-EKF dynamics. The following Theorem~\ref{theorem: inverse EKF stable unknown matrix} states the stability criteria for I-EKF. Note that, when compared  to Theorem~\ref{theorem:Forward ekf stable unknown matrix}, the following result requires an additional condition $\underline{r}\mathbf{I}\preceq\mathbf{R}_{k}$ for all $k\geq 0$ for some $\underline{r}>0$. 
\begin{theorem}[Stochastic stability of I-EKF]
\label{theorem: inverse EKF stable unknown matrix}
Consider the adversary's forward EKF that is stable as per Theorem \ref{theorem:Forward ekf stable unknown matrix}. Additionally, assume that the following hold true for all $k\geq 0$.
\par\noindent\small
\begin{align*}
    \underline{r}\mathbf{I}&\preceq\mathbf{R}_{k},&\|\mathbf{G}_{k}\|&\leq\overline{g},&\|\overline{\mathbf{U}}^{a}_{k}\|&\leq\overline{c},&
    \|\overline{\mathbf{U}}^{xa}_{k}\|&\leq\overline{d},\\
    \overline{\mathbf{R}}_{k}&\preceq\overline{\epsilon}\mathbf{I},&\hat{c}\mathbf{I}&\preceq\hat{\overline{\mathbf{Q}}}_{k},&   \hat{d}\mathbf{I}&\preceq\hat{\overline{\mathbf{R}}}_{k},&
    \underline{p}\mathbf{I}&\preceq\overline{\bm{\Sigma}}_{k|k-1}\preceq\overline{p}\mathbf{I},
\end{align*}
\normalsize
for some real positive constants $\underline{r}, \overline{g}, \overline{c}, \overline{d}, \overline{\epsilon}, \hat{c}, \hat{d}, \underline{p}, \overline{p}$. Then, the state estimation error of I-EKF is exponentially bounded in mean-squared sense and bounded with probability one provided that the constants satisfy the inequality 
  $\overline{p}\overline{d}\overline{g}^{2}\overline{c}^{2}<{\hat{d}}$. 
\end{theorem}

\begin{proof}
See Appendix~\ref{App-thm-inverse EKF stable unknown matrix}.
\end{proof}
\begin{remark}
Note that Theorem \ref{theorem:Forward ekf stable unknown matrix} requires both $\hat{\mathbf{Q}}_{k}$ and $\hat{\mathbf{R}}_{k}$ to be p.d. In general, the difference matrices $\Delta\mathbf{P}_{k+1|k}$, $\delta\mathbf{P}_{k+1|k}$, $\Delta\mathbf{P}^{yy}_{k+1}$ and $\delta\mathbf{P}^{yy}_{k+1}$ may not be p.d. One could enhance the stability of EKF by enlarging the noise covariance matrices by adding sufficiently large $\Delta\mathbf{Q}_{k}$ and $\Delta\mathbf{R}_{k}$ to $\mathbf{Q}_{k}$ and $\mathbf{R}_{k}$, respectively \cite{xiong2006performance_ukf,xiong2007authorreply}. The same argument also holds true for I-EKF noise covariance matrices.
\end{remark}

\subsection{I-EKF-without-unknown-input: Bounded non-linearity method}
\label{subsec:ekf stable Reif}
Consider the forward EKF's one step prediction formulation \eqref{eqn: one step forward ekf gain}-\eqref{eqn: one step forward ekf covariance}. Using Taylor series expansion around the estimate $\hat{\mathbf{x}}_{k}$, we have
\par\noindent\small
\begin{align*}
&f(\mathbf{x}_{k})-f(\hat{\mathbf{x}}_{k})=\mathbf{F}_{k}(\mathbf{x}_{k}-\hat{\mathbf{x}}_{k})+\phi(\mathbf{x}_{k},\hat{\mathbf{x}}_{k}),\\
&h(\mathbf{x}_{k})-h(\hat{\mathbf{x}}_{k})=\mathbf{H}_{k}(\mathbf{x}_{k}-\hat{\mathbf{x}}_{k})+\chi(\mathbf{x}_{k},\hat{\mathbf{x}}_{k}),
\end{align*}
\normalsize
where $\phi(\cdot)$ and $\chi(\cdot)$ are suitable non-linear functions to account for the higher-order terms of the expansions. Denoting the estimation error by $\mathbf{e}_{k}\doteq\mathbf{x}_{k}-\hat{\mathbf{x}}_{k}$, the error dynamics of the forward filter is
\par\noindent\small
\begin{align}
\mathbf{e}_{k+1}=(\mathbf{F}_{k}-\mathbf{K}_{k}\mathbf{H}_{k})\mathbf{e}_{k}+\mathbf{r}_{k}+\mathbf{s}_{k},\label{eqn: forward ekf error}
\end{align}
\normalsize
where $\mathbf{r}_{k}=\phi(\mathbf{x}_{k},\hat{\mathbf{x}}_{k})-\mathbf{K}_{k}\chi(\mathbf{x}_{k},\hat{\mathbf{x}}_{k})$ and $\mathbf{s}_{k}=\mathbf{w}_{k}-\mathbf{K}_{k}\mathbf{v}_{k}$.

The following Theorem~\ref{theorem: ekf stable Reif} (reproduced from \cite{reif1999stochastic}) provides sufficient conditions for forward EKF's stochastic stability.
\begin{theorem}[Exponential boundedness of forward EKF's error \cite{reif1999stochastic}]\label{theorem: ekf stable Reif} Consider a non-linear stochastic system defined by \eqref{eqn: ekf x} and \eqref{eqn: non y withoutdf}, and the one-step prediction formulation of forward EKF \eqref{eqn: one step forward ekf gain}-\eqref{eqn: one step forward ekf covariance}. Let the following assumptions hold true.
\begin{enumerate}
\item There exist positive real numbers $\overline{f}$,$\overline{h}$,$\underline{\sigma}$,$\overline{\sigma}$,$\underline{q}$,$\underline{r}$, $\delta$ such that the following bounds are fulfilled for all $k\geq 0$.
\par\noindent\small
\begin{align*}
\underline{\sigma}\mathbf{I}&\preceq\bm{\Sigma}_{k}\preceq\overline{\sigma}\mathbf{I},&\underline{q}\mathbf{I}&\preceq \mathbf{Q}_{k}\preceq\delta\mathbf{I},&\\
\underline{r}\mathbf{I}&\preceq \mathbf{R}_{k}\preceq\delta\mathbf{I},&\|\mathbf{F}_{k}\|&\leq\overline{f},\hspace{0.5cm}\|\mathbf{H}_{k}\|\leq\overline{h}.&
\end{align*}
\normalsize
\item $\mathbf{F}_{k}$ is non singular for every $k\geq 0$.
\item There exist positive real numbers $\kappa_{\phi}$, $\epsilon_{\phi}$, $\kappa_{\chi}$, $\epsilon_{\chi}$ such that the non-linear functions $\phi(\cdot)$ and $\chi(\cdot)$ satisfy
\par\noindent\small
\begin{align*}
&\|\phi(\mathbf{x},\hat{\mathbf{x}})\|_{2}\leq \kappa_{\phi}\|\mathbf{x}-\hat{\mathbf{x}}\|_{2}^{2}\hspace{0.2cm}\text{for}\hspace{0.2cm} \|\mathbf{x}-\hat{\mathbf{x}}\|_{2}\leq\epsilon_{\phi},\\
&\|\chi(\mathbf{x},\hat{\mathbf{x}})\|_{2}\leq \kappa_{\chi}\|\mathbf{x}-\hat{\mathbf{x}}\|_{2}^{2}\hspace{0.2cm}\text{for}\hspace{0.2cm} \|\mathbf{x}-\hat{\mathbf{x}}\|_{2}\leq\epsilon_{\chi}.
\end{align*}
\normalsize
\end{enumerate}
Then the estimation error given by \eqref{eqn: forward ekf error} is exponentially bounded in mean-squared sense and bounded with probability one provided that the estimation error is bounded by suitable constant $\epsilon>0$. 
\end{theorem}

\begin{remark}
Theorem~\ref{theorem: ekf stable Reif} guarantees that the estimation error remains exponentially bounded in mean-squared sense as long as the error is within suitable $\epsilon$ bounds. Further, the mean drift $\mathbb{E}[V_{k+1}(\mathbf{e}_{k+1})|\mathbf{e}_{k}]-V_{k}(\mathbf{e}_{k})$ for a suitably defined $V_{k}(\cdot)$ (for application of Lemma \ref{lemma:exponential boundedness}) is negative when $\widetilde{\epsilon}\leq\|\mathbf{e}_{k}\|_{2}\leq\epsilon$, which drives the system towards zero error in an expected sense. However, with some finite probability, the estimation error at some time-steps may be outside the $\epsilon$ bound. In this case, we cannot guarantee with probability one that the error will be within $\epsilon$ bound again at some future time-steps.
\end{remark}
As mentioned in Remark~\ref{remark:reif result remark}, bounded non-linearity approach may not provide theoretical guarantees for the filter to be stable for all time-steps but, practically, the filter remains stable even if the estimation error is outside the $\epsilon$ bound provided that the assumed bounds on the system model are satisfied.

For the inverse filter observations \eqref{eqn: non a}, the Taylor series expansion of $g(\cdot)$ at estimate $\doublehat{\mathbf{x}}_{k}$ of I-EKF's one step prediction formulation of Section~\ref{subsec:ekf}, considering suitable non-linear function $\overline{\chi}(\cdot)$ is
\par\noindent\small
\begin{align*}
&g(\hat{\mathbf{x}}_{k})-g(\doublehat{\mathbf{x}}_{k})=\mathbf{G}_{k}(\hat{\mathbf{x}}_{k}-\doublehat{\mathbf{x}}_{k})+\overline{\chi}(\hat{\mathbf{x}}_{k},\doublehat{\mathbf{x}}_{k}).
\end{align*}
\normalsize
Finally, the error dynamics of the inverse filter, with the estimation error denoted by $\overline{\mathbf{e}}_{k}\doteq\hat{\mathbf{x}}_{k}-\doublehat{\mathbf{x}}_{k}$ and the inverse filter's Kalman gain and estimation error covariance matrix by $\overline{\mathbf{K}}_{k}$ and $\overline{\bm{\Sigma}}_{k}$, respectively, is
\par\noindent\small
\begin{align}
\overline{\mathbf{e}}_{k+1}=(\widetilde{\mathbf{F}}^{x}_{k}-\overline{\mathbf{K}}_{k}\mathbf{G}_{k})\overline{\mathbf{e}}_{k}+\overline{\mathbf{r}}_{k}+\overline{\mathbf{s}}_{k},\label{eqn: inverse ekf error}
\end{align}
\normalsize
where $\overline{\mathbf{r}}_{k}=\overline{\phi}_{k}(\hat{x}_{k},\doublehat{\mathbf{x}}_{k})-\overline{\mathbf{K}}_{k}\overline{\chi}(\hat{\mathbf{x}}_{k},\doublehat{\mathbf{x}}_{k})$ and $\overline{\mathbf{s}}_{k}=\mathbf{K}_{k}\mathbf{v}_{k}-\overline{\mathbf{K}}_{k}\bm{\epsilon}_{k}$ with $\overline{\phi}_{k}(\hat{\mathbf{x}}_{k},\doublehat{\mathbf{x}}_{k})=\phi(\hat{\mathbf{x}}_{k},\doublehat{\mathbf{x}}_{k})-\mathbf{K}_{k}\chi(\hat{\mathbf{x}}_{k},\doublehat{\mathbf{x}}_{k})$.

The following Theorem~\ref{theorem: inverse ekf stable Reif} guarantees the stability of I-EKF. Note the additional assumption of $\mathbf{H}_{k}$ to be full column rank for all $k\geq 0$, which implies $p\geq n$. 
\begin{theorem}[Exponential boundedness of I-EKF's error]
\label{theorem: inverse ekf stable Reif} 
Consider the adversary's forward one-step prediction EKF that is stable as per Theorem \ref{theorem: ekf stable Reif}. Additionally, assume that the following hold true.
\begin{enumerate}
\item There exist positive real numbers $\overline{g}$,  
\underline{$m$}, 
$\overline{m}$, $\underline{\epsilon}$, $\overline{\delta}$ such that the following bounds are fulfilled for all $k\geq 0$.
\par\noindent\small
\begin{align*}
&\|\mathbf{G}_{k}\|\le\overline{g},\;\;\underline{m}\mathbf{I}\preceq\overline{\bm{\Sigma}}_{k}\preceq\overline{m}\mathbf{I},\;\;\underline{\epsilon}\mathbf{I}\preceq \overline{\mathbf{R}}_{k}\preceq\overline{\delta}\mathbf{I}.
\end{align*}
\normalsize
\item $\mathbf{H}_{k}$ is full column rank for every $k\geq 0$.
\item There exist positive real numbers $\kappa_{\bar{\chi}}$ and $\epsilon_{\bar{\chi}}$ such that the non-linear function $\overline{\chi}(\cdot)$ satisfies
\par\noindent\small
\begin{align*}
\|\overline{\chi}(\hat{\mathbf{x}},\doublehat{\mathbf{x}})\|_{2}\leq \kappa_{\bar{\chi}}\|\hat{\mathbf{x}}-\doublehat{\mathbf{x}}\|_{2}^{2}\hspace{0.2cm}\text{for}\hspace{0.2cm} \|\hat{\mathbf{x}}-\doublehat{\mathbf{x}}\|_{2}\leq\epsilon_{\bar{\chi}}.
\end{align*}
\normalsize
\end{enumerate}
Then, the estimation error for I-EKF given by \eqref{eqn: inverse ekf error} is exponentially bounded in mean-squared sense and bounded with probability one provided that the estimation error is bounded by suitable constant $\overline{\epsilon}>0$.
\end{theorem}

\begin{proof}
See Appendix~\ref{App-thm-inverse ekf stable Reif}.
\end{proof}

\begin{remark}\label{remark:observability}
The inequality $\underline{m}\mathbf{I}\preceq\overline{\bm{\Sigma}}_{k}\preceq\overline{m}\mathbf{I}$ assumed in Theorem~\ref{theorem: inverse ekf stable Reif} is closely related to the observability of the non-linear inverse filtering model. In particular, \cite{reif1999stochastic} showed that the condition $\underline{\sigma}\mathbf{I}\preceq\bm{\Sigma}_{k}\preceq\overline{\sigma}\mathbf{I}$ assumed for forward EKF's stability in Theorem~\ref{theorem: ekf stable Reif} is satisfied if the non-linear observability rank condition holds, i.e., the non-linear observability matrix $\mathbf{A}_{k}=\begin{bsmallmatrix}\nabla h(\mathbf{x}_{k})\\\nabla h(\mathbf{x}_{k+1})\nabla f(\mathbf{x}_{k})\\\vdots\\\nabla h(\mathbf{x}_{k+q-1})\nabla f(\mathbf{x}_{k+q-2})\hdots\nabla f(\mathbf{x}_{k})\end{bsmallmatrix}$ has full rank $q$ at $\mathbf{x}_{k}$. Note that the inequality $\underline{m}\mathbf{I}\preceq\overline{\bm{\Sigma}}_{k}\preceq\overline{m}\mathbf{I}$ is a weaker assumption than the observability. The same argument holds for the inequality $\underline{m}\mathbf{I}\preceq\overline{\bm{\Sigma}}_{k}\preceq\overline{m}\mathbf{I}$ for I-EKF's stability.
\end{remark}

\subsection{I-EKF-without-unknown-input: Consistency}
\label{subsec:iekf consistency}
Let us recall the following definition.
\begin{definition}[Consistency of estimator\cite{battistelli2014kullback}]
Consider an unbiased estimate $\hat{\mathbf{x}}$ of random variable $\mathbf{x}$ and its error covariance estimate $\bm{\Sigma}$. The pair ($\hat{\mathbf{x}}$,$\bm{\Sigma}$) are said to be consistent if $\mathbb{E}[(\mathbf{x}-\hat{\mathbf{x}})(\mathbf{x}-\hat{\mathbf{x}})^{T}]\preceq\bm{\Sigma}$, i.e., the estimated covariance $\bm{\Sigma}$ upper bounds the true error covariance. 
\end{definition}
To analyze the consistency of I-EKF's estimates, we consider the statistical linearization technique (SLT)\cite{arasaratnam2007qkf}. Linearize the state transition \eqref{eqn: inverse ekf state transition} and observation \eqref{eqn: non a}, respectively, at $[\hat{\mathbf{x}}_{k}^{T},\mathbf{v}_{k+1}^{T}]^{T}$ and $\hat{\mathbf{x}}_{k}$ as
\par\noindent\small
\begin{align}
    \hat{\mathbf{x}}_{k+1}&=\mathbf{U}^{xv}_{k}\overline{\mathbf{F}}^{x}_{k}\hat{\mathbf{x}}_{k}+\mathbf{U}^{xv}_{k}\overline{\mathbf{F}}^{v}_{k}\mathbf{v}_{k+1},\label{eqn:SLT state transition}\\
    \mathbf{a}_{k}&=\mathbf{U}^{a}_{k}\overline{\mathbf{G}}_{k}\hat{\mathbf{x}}_{k}+\bm{\epsilon}_{k},\label{eqn:SLT observation}
\end{align}
\normalsize
where $\overline{\mathbf{F}}_{k}=[\overline{\mathbf{F}}^{x}_{k},\overline{\mathbf{F}}^{v}_{k}]$ and $\overline{\mathbf{G}}_{k}$ are the respective linear pseudo transition matrices. Also, $\mathbf{U}^{xv}_{k}$ and $\mathbf{U}^{a}_{k}$ are unknown diagonal matrices introduced to account for the approximation errors in SLT. Note that these unknown matrices are different from the ones introduced in Section~\ref{subsec:ekf stable unknown} for the higher-order terms in the Taylor approximation.
\begin{theorem}[I-EKF's consistency]
    \label{thm:consistency}
    Consider an I-EKF initialized with a consistent initial estimate pair $(\doublehat{\mathbf{x}}_{0},\overline{\bm{\Sigma}}_{0})$. Then for any $k\geq 1$, the estimate $(\doublehat{\mathbf{x}}_{k},\overline{\bm{\Sigma}}_{k})$ computed recursively by the I-EKF are also consistent such that $\mathbb{E}[(\hat{\mathbf{x}}_{k}-\doublehat{\mathbf{x}}_{k})(\hat{\mathbf{x}}_{k}-\doublehat{\mathbf{x}}_{k})^{T}]\preceq\overline{\bm{\Sigma}}_{k}$, where $\hat{\mathbf{x}}_{k}$ is the forward EKF's state estimate.
\end{theorem}
\begin{proof}
    See Appendix~\ref{App-thm-inverse EKF consistency}.
\end{proof}

\vspace{-8pt}
\section{Numerical Experiments}\label{sec:simulations}
We illustrate the performance of the proposed inverse filters for different example systems. The efficacy of the inverse filters is demonstrated by comparing the estimation error with RCRLB. The CRLB provides a lower bound on mean-squared error (MSE) and is widely used to assess the performance of an estimator. For the discrete-time non-linear filtering, we employ the RCRLB as $\mathbb{E}\left[(\mathbf{x}_{k}-\hat{\mathbf{x}}_{k})(\mathbf{x}_{k}-\hat{\mathbf{x}}_{k})^{T}\right]\succeq\mathbf{J}_{k}^{-1}$ 
where $\mathbf{J}_{k}=\mathbb{E}\left[-\frac{\partial^{2}\ln{p(Y^{k},X^{k})}}{\partial\mathbf{x}_{k}^{2}}\right]$ is the Fisher information matrix\cite{tichavsky1998posterior}.
Here, $X^{k}=\lbrace\mathbf{x}_{0},\mathbf{x}_{1},\hdots,\mathbf{x}_{k}\rbrace$ is the state vector series while $Y^{k}=\lbrace\mathbf{y}_{1},\mathbf{y}_{2},\hdots,\mathbf{y}_{k}\rbrace$ are the noisy observations. Also, $p(Y^{k},X^{k})$ is the joint probability density of pair $(Y^{k},X^{k})$ and $\hat{\mathbf{x}}_{k}$ (a function of $Y^{k}$) is an estimate of $\mathbf{x}_{k}$ with $\frac{\partial^{2}(\cdot)}{\partial\mathbf{x}^{2}}$ denoting the Hessian with second order partial derivatives. The information matrix $\mathbf{J}_{k}$ can be computed recursively as \cite{tichavsky1998posterior}
\par\noindent\small
\begin{align}
    \mathbf{J}_{k}&=\mathbf{D}_{k}^{22}-\mathbf{D}_{k}^{21}(\mathbf{J}_{k-1}+\mathbf{D}_{k}^{11})^{-1}\mathbf{D}_{k}^{12},\label{eqn: general Jk recursions}\\
    \text{where}\hspace{0.25cm}\mathbf{D}_{k}^{11}&=\mathbb{E}\left[-\frac{\partial^{2}\ln{p(\mathbf{x}_{k}\vert\mathbf{x}_{k-1})}}{\partial\mathbf{x}_{k-1}^{2}}\right],\nonumber\\
    \mathbf{D}_{k}^{12}&=\mathbb{E}\left[-\frac{\partial^{2}\ln{p(\mathbf{x}_{k}\vert\mathbf{x}_{k-1})}}{\partial\mathbf{x}_{k}\partial\mathbf{x}_{k-1}}\right]=(\mathbf{D}_{k}^{21})^{T},\nonumber\\
    \mathbf{D}_{k}^{22}&=\mathbb{E}\left[-\frac{\partial^{2}\ln{p(\mathbf{x}_{k}\vert\mathbf{x}_{k-1})}}{\partial\mathbf{x}_{k}^{2}}\right]+\mathbb{E}\left[-\frac{\partial^{2}\ln{p(\mathbf{y}_{k}\vert\mathbf{x}_{k})}}{\partial\mathbf{x}_{k}^{2}}\right].\nonumber
\end{align}
\normalsize

For the non-linear system given by \eqref{eqn: ekf x} and \eqref{eqn: non y withoutdf}, the forward information matrices $\lbrace\mathbf{J}_{k}\rbrace$ recursions reduces to \cite{xiong2006performance_ukf}
\par\noindent\small
\begin{align}
    &\mathbf{J}_{k+1}=\mathbf{Q}_{k}^{-1}\nonumber\\
    &\;\;+\mathbf{H}_{k+1}^{T}\mathbf{R}_{k+1}^{-1}\mathbf{H}_{k+1}-\mathbf{Q}_{k}^{-1}\mathbf{F}_{k}(\mathbf{J}_{k}+\mathbf{F}_{k}^{T}\mathbf{Q}_{k}^{-1}\mathbf{F}_{k})^{-1}\mathbf{F}_{k}^{T}\mathbf{Q}_{k}^{-1},\label{eqn: additive Jk recursions}
\end{align}
\normalsize
where $\mathbf{F}_{k}=\nabla_{\mathbf{x}}f(\mathbf{x})\vert_{\mathbf{x}=\mathbf{x}_{k}}$ and $\mathbf{H}_{k}=\nabla_{\mathbf{x}}h(\mathbf{x})\vert_{\mathbf{x}=\mathbf{x}_{k}}$. Note that, for the information matrices recursion, the Jacobians $\mathbf{F}_{k}$ and $\mathbf{H}_{k}$ are evaluated at the true state $\mathbf{x}_{k}$ while for forward EKF recursions, these are evaluated at the estimates of the state. These recursions can be trivially extended to other system models considered in this paper and to compute the information matrix $\overline{\mathbf{J}}_{k}$ for inverse filter's estimate $\doublehat{\mathbf{x}}_{k}$. Some recent studies on cognitive radar target tracking instead consider posterior CRLB \cite{bell2015cognitive} as a metric to tune tracking filters.

Throughout all experiments, $100$ time-steps (indexed by $k$) were considered. The initial information matrices $\mathbf{J}_{0}$ and $\overline{\mathbf{J}}_{0}$ were set to $\bm{\Sigma}_{0}^{-1}$ and $\overline{\bm{\Sigma}}_{0}^{-1}$, respectively, 
unless mentioned otherwise. Note that these initial estimates only affect the RCRLB in the transient phase. The steady state RCRLB is independent of the initialization.
\subsection{Inverse KF with unknown inputs}\label{subsec:sim KF with unknown inputs}
Consider a discrete-time linear system without DF\cite{hsieh2000robust}, 

\par\noindent\small
\begin{align*}
&\mathbf{x}_{k+1}=\begin{bmatrix}0.1 & 0.5 & 0.08\\ 0.6 & 0.01 & 0.04\\ 0.1 & 0.7 & 0.05\end{bmatrix}\mathbf{x}_{k}+\begin{bmatrix}0\\ 2\\ 1\end{bmatrix}u_{k}+\mathbf{w}_{k},\\
&\mathbf{y}_{k}=\begin{bmatrix}1 & 1 & 0\\ 0 & 1 & 1\end{bmatrix}\mathbf{x}_{k}+\mathbf{v}_{k},\;\;\;a_{k}=\begin{bmatrix}1 & 1 & 1\end{bmatrix}\hat{\mathbf{x}}_{k}+\epsilon_{k},
\end{align*}
\normalsize
with $\mathbf{w}_{k}\sim\mathcal{N}(\mathbf{0},\mathbf{I}_{3})$, $\mathbf{v}_{k}\sim\mathcal{N}(\mathbf{0},2\mathbf{I}_{2})$ and $\epsilon_{k}\sim\mathcal{N}(0,5)$. The 
unknown input $u_{k}$ was set to $50$ for $1\leq k \leq 50$ and $-50$ thereafter. 
The initial state was $\mathbf{x}_{0}=[1,1,1]^{T}$. For the forward filter, the initial state estimate was set to $[0,0,0]^{T}$ with initial covariance $\bm{\Sigma}_{0}=\mathbf{I}_{3}$. For the inverse filter, the initial state estimate was set to $\mathbf{x}_{0}$ (known to the defender) itself with initial covariance $\overline{\bm{\Sigma}}_{0}=5\mathbf{I}_{3}$. 

For KF-with-DF, we modify the forward filter's observations as\cite{pan2011study}:
\par\noindent\small
\begin{align*}
\mathbf{y}_{k}=\begin{bmatrix}1 & 1 & 0\\ 0 & 1 & 1\end{bmatrix}\mathbf{x}_{k}+\begin{bmatrix}0\\1\end{bmatrix}u_{k}+\mathbf{v}_{k}.
\end{align*}
\normalsize
Here, the initial input estimate was set to $10$ with initial input estimate covariance $\bm{\Sigma}^{u}_{0}=10$ and initial cross-covariance $\bm{\Sigma}^{xu}_{0}=[0,0,0]^{T}$. The inverse filter's initial augmented state estimate $\mathbf{z}_{0}$ was set to $[1,1,1,50]^{T}$ with initial covariance $\overline{\bm{\Sigma}}_{0}=5\mathbf{I}_{4}$. 

Fig. \ref{fig:KF unknown input} shows the time-averaged RMSE (AMSE) $=\sqrt{(\sum_{i=1}^{k}\|\mathbf{x}_{i}-\hat{\mathbf{x}}_{i}\|_{2}^{2})/nk}$ at $k$-th time step for $n$-dimensional actual state $\mathbf{x}_{i}$ and its estimate $\hat{\mathbf{x}}_{i}$, and RCRLB for state estimation for both forward and inverse filters in the two cases, respectively, averaged over 200 runs. For KF-without-DF, we plot the root MSE (RMSE) $=\sqrt{(\|\mathbf{x}_{k}-\hat{\mathbf{x}}_{k}\|^{2}_{2})/n}$ for comparison here but omit it for later plots for clarity. 
Note that in Fig.~\ref{fig:KF unknown input}a, the I-KF-without-DF's RMSE fluctuates about the RCRLB because of a finite number of sample paths; see also similar phenomena in \cite{xiong2006performance_ukf,djuric2008target,vsimandl2001filtering}. The RCRLB value for state estimation is $\sqrt{\textrm{Tr}(\mathbf{J}^{-1})}$ with $\mathbf{J}$ denoting the associated information matrix.

Fig. \ref{fig:KF unknown input} shows that the effect of change in unknown input after 50 time-steps is negligible for KF-without-DF in both forward and inverse filters. However, for KF-with-DF, the sudden change in unknown input leads to an increase in state estimation error of the forward filter and, consequently, of the inverse filter. The estimation error of I-KF-without-DF is less than the corresponding forward filter while for KF-with-DF, the inverse filter has a higher estimation error than the forward filter. 
Only I-KF-without-DF efficiently achieves the RCRLB bound on the estimation error. Note that in this and the following numerical experiments, the forward and inverse filters are compared only to highlight the relative estimation accuracy.
\begin{figure}
  \centering
  \includegraphics[width = 1.0\columnwidth]{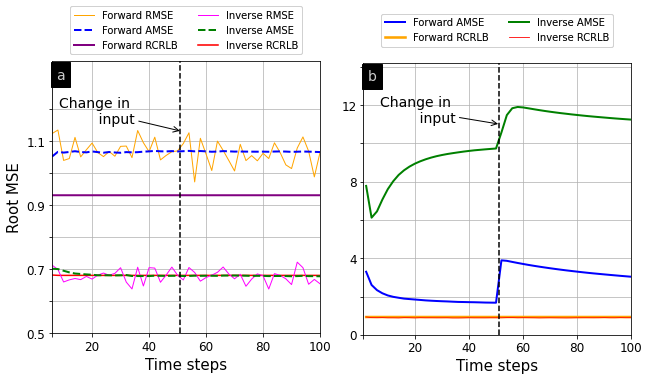}
  \caption{RMSE, AMSE and RCRLB for forward and inverse filters (a) KF-without-DF; (b) KF-with-DF.}
 \label{fig:KF unknown input}
\end{figure}

\subsection{Inverse EKF without unknown inputs}\label{subsec:sim EKF}
Consider the discrete-time non-linear system model of FM demodulator without unknown inputs \cite[Sec. 8.2]{anderson2012optimal}
\par\noindent\small
\begin{align*}
&\mathbf{x}_{k+1}\doteq\begin{bmatrix}\lambda_{k+1}\\\theta_{k+1}\end{bmatrix}=\begin{bmatrix}\exp{(-T/\beta)}&0\\-\beta \exp{(-T/\beta)}-1&1\end{bmatrix}\begin{bmatrix}\lambda_{k}\\\theta_{k}\end{bmatrix}+\begin{bmatrix}1\\-\beta\end{bmatrix}w_{k},\\
&\mathbf{y}_{k}=\sqrt{2}\begin{bmatrix}\sin{\theta_{k}}\\\cos{\theta_{k}}\end{bmatrix}+\mathbf{v}_{k},\;\;
a_{k}=\hat{\lambda}_{k}^{2}+\epsilon_{k},
\end{align*}
\normalsize
with $w_{k}\sim\mathcal{N}(0,0.01)$, $\mathbf{v}_{k}\sim\mathcal{N}(\mathbf{0},\mathbf{I}_{2})$, $\epsilon_{k}\sim\mathcal{N}(0,5)$, $T=2\pi/16$ and $\beta=100$. Here, the observation function $g(\cdot)$ for the inverse filter is quadratic. Also, $\hat{\lambda}_{k}$ is the forward EKF's estimate of $\lambda_{k}$.

The initial state $\mathbf{x}_{0}\doteq[\lambda_{0},\theta_{0}]^{T}$ was set randomly with $\lambda_{0}\sim\mathcal{N}(0,1)$ and $\theta_{0}\sim\mathcal{U}[-\pi,\pi]$. The initial state estimates of forward and inverse EKF were also similarly drawn at random. The initial covariances were set to $\bm{\Sigma}_{0}=10\mathbf{I}_{2}$ and $\overline{\bm{\Sigma}}_{0}=5\mathbf{I}_{2}$ for forward and inverse EKF, respectively. The phase term of the state $\theta$ and its estimates $\hat{\theta}$ and $\doublehat{\theta}$ (for both prediction and measurement updates) were considered to be modulo $2\pi$ \cite{anderson2012optimal}. 
Note that the process covariance $\mathbf{Q}$ is a singular matrix. For numerical stability and to facilitate computation of $\mathbf{Q}^{-1}$ for evaluating information matrices $\mathbf{J}_{k}$, we used an enlarged covariance matrix by adding $10^{-10}\mathbf{I}_{2}$ to $\mathbf{Q}$ in the forward filters. Similarly, we added $10^{-10}\mathbf{I}_{2}$ to $\overline{\mathbf{Q}}_{k}$ in the inverse filter because $\overline{\mathbf{Q}}_{k}$ is time-varying and may be ill-conditioned. The initial $\overline{\mathbf{J}}_{0}$ was taken close to the inverse of the steady state estimation covariance matrix of the forward filter. The initial $\overline{\mathbf{J}}_{0}$ only affects the RCRLB calculated for initial few time-steps. The RCRLB after these initial time-steps (around 20 for the considered system) shows same behaviour irrespective of the initial $\overline{\mathbf{J}}_{0}$.

Fig. \ref{fig:EKF and EKF with unknown inputs}a shows the AMSE and RCRLB for forward and inverse EKF averaged over 200 runs. The I-EKF's estimation error is comparable to that of forward EKF with I-EKF's average error being slightly higher than that of forward EKF. However, the difference between AMSE and RCRLB for I-EKF is less than that for forward EKF. Hence, we conclude that I-EKF is more efficient here. The I-EKF assumes initial covariance $\bm{\Sigma}_{0}$ as $5\mathbf{I}_{2}$ (the true $\bm{\Sigma}_{0}$ of forward EKF is $10\mathbf{I}_{2}$) and a random initial state for these recursions. In spite of this difference in the initial estimates, I-EKF's error performance is comparable to that of the forward EKF.

\subsection{Inverse EKF with unknown inputs}\label{subsec:sim EKF with unknown inputs}
\begin{figure}
  \centering
  \includegraphics[width = 1.0\columnwidth]{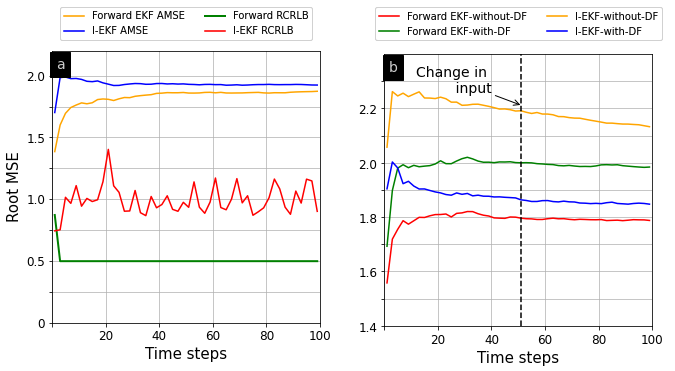}
  \caption{(a) AMSE and RCRLB for forward and inverse EKF; (b) Time-averaged RMSE for forward and inverse EKF with and without DF, averaged over 200 runs.}
 \label{fig:EKF and EKF with unknown inputs}
\end{figure}
For inverse EKF with unknown input, we modified the non-linear system model of Section \ref{subsec:sim EKF} to include an unknown input $u_{k}$ as 
\par\noindent\small
\begin{align*}
\mathbf{x}_{k+1}=\begin{bmatrix}\exp{(-T/\beta)}&0\\-\beta \exp{(-T/\beta)}-1&1\end{bmatrix}\begin{bmatrix}\lambda_{k}\\\theta_{k}\end{bmatrix}+\begin{bmatrix}0.001\\1\end{bmatrix}u_{k}+\begin{bmatrix}1\\-\beta\end{bmatrix}w_{k},
\end{align*}
\normalsize
where $u_{k}$ was set to $\pi/4$ for $1\leq k \leq 50$ and $-\pi/4$ thereafter. 
The observation $\mathbf{y}_{k}$ of the forward EKF-without-DF was same as in Section \ref{subsec:sim EKF}. 
Consider a linear measurement $a_{k}$ for the inverse filter as 
    $a_{k}=\hat{\lambda}_{k}+\epsilon_{k}$.
For the forward filter, the initial input estimate was set to $0$ while the inverse filter initial augmented state estimate consisted of the true state $\mathbf{x}_{0}$ and true input $u_{0}$ (known to the defender) with initial covariance estimate $\overline{\bm{\Sigma}}_{0}=15\mathbf{I}_{3}$. 

Similarly, for system with DF, we again considered the same non-linear system (without any unknown input in $\mathbf{x}_{k}$ state transition) but with a modified forward filter's observation 
$\mathbf{y}_{k}=\sqrt{2}\begin{bmatrix}\sin{(\theta_{k}+u_{k})}\\\cos{(\theta_{k}+u_{k})}\end{bmatrix}+\mathbf{v}_{k}$. 
The input estimates $\hat{u}$ and $\doublehat{u}$ were also, as before, modulo $2\pi$. The Gaussian noise terms in the inverse filter state transitions (\eqref{eqn: state transition ekf without df state} and \eqref{eqn:state transition ekf with df}) are transformed through non-linear functions such that \eqref{eqn: additive Jk recursions} is not applicable. The RCRLB in this case is derived using the general $\mathbf{J}_{k}$ recursions given by \eqref{eqn: general Jk recursions}, which is omitted here. 
Fig. \ref{fig:EKF and EKF with unknown inputs}b shows that for both EKF with and without DF, the change in unknown input after $50$ time-steps does not increase the estimation error (as for KF-with-DF in Fig. \ref{fig:KF unknown input}b). 
The estimation error of I-EKF-without-DF (I-EKF-with-DF) is higher (lower) than that of the corresponding forward filter. 
Any change in unknown input affects the inverse filter's performance only when a significant change occurs in the forward filter's performance.

\vspace{-8pt}
\section{Summary}\label{sec:summary}
We studied the inverse filtering problem for non-linear systems with and without unknown inputs in the context of counter-adversarial applications. 
For systems with unknown inputs, the adversary's observations may or may not be affected by the unknown input known to the defender but not the adversary. 
The stochastic stability of a forward filter with certain additional system assumptions is also sufficient for the stability of the inverse filter. Such a stability analysis of inverse filter has not been considered in the prior work \cite{krishnamurthy2019how}. While \cite{bell2015cognitive,sharaga2015optimal} consider adapting a cognitive radar based on its observations, the proposed inverse filters allow a counter-adversarial defender to infer such a cognitive radar's information by observing its adaptations. 
Our experiments suggested that the impact of the unknown input on inverse filter's performance strongly depends on its impact on the forward filter. For certain systems, the inverse filter may perform more efficiently than the forward filter. In the companion paper (Part II) \cite{singh2022inverse_part2}, we develop I-EKF for second-order, Gaussian sum, and dithered EKFs and consider the case of uncertain information about the forward filter. 

\appendices

\section{Forward EKF-without-DF recursions}
\label{App-forward-EKF-without-DF-recursions}
Here, we provide the detailed steps to derive the forward EKF-without-DF recursions, which were omitted in \cite{pan2010applying}. The forward EKF-without-DF is formulated based on a weighted least-squared error criterion. To this end, similar to EKF, the system model is first linearized locally at the estimates of the previous state and unknown inputs. The linearized models are then used to define a quadratic objective function of an extended state vector consisting of the current state and the unknown inputs at all time instants. Finally, recursive estimates are derived for the extended state vector and then simplified to yield forward EKF-without-DF recursions. As mentioned in Remark~\ref{remark:with and without diff} of the paper, systems without DF induce a one-step delay in input estimation.


Consider the non-linear state transition \eqref{eqn: non x with input} and observation \eqref{eqn: non y withoutdf} without DF. We require estimates $\hat{\mathbf{x}}_{k|k}$ and $\hat{\mathbf{u}}_{k-1|k}$ (represented by $\hat{\mathbf{x}}_{k}$ and $\hat{\mathbf{u}}_{k-1}$ in the main paper) of the state $\mathbf{x}_{k}$ and unknown input $\mathbf{u}_{k-1}$, respectively, given the observations $\{\mathbf{y}_{i}\}_{1\leq i\leq k}$. Linearize the non-linear functions $f(\cdot,\cdot)$ in \eqref{eqn: non x with input} and $h(\cdot)$ in \eqref{eqn: non y withoutdf} with respect to the previous estimates as follows:
\par\noindent\small
\begin{align}
f(\mathbf{x}_{k},\mathbf{u}_{k})&=f(\hat{\mathbf{x}}_{k|k},\hat{\mathbf{u}}_{k-1|k})+\mathbf{F}_{k}(\mathbf{x}_{k}-\hat{\mathbf{x}}_{k|k})+\mathbf{B}_{k}(\mathbf{u}_{k}-\hat{\mathbf{u}}_{k-1|k}),\nonumber\\
    h(\mathbf{x}_{k+1})&=h(\hat{\mathbf{x}}_{k+1|k})+\mathbf{H}_{k+1}(\mathbf{x}_{k+1}-\hat{\mathbf{x}}_{k+1|k}),\label{eqn:supp 14}
\end{align}
\normalsize
Then, \eqref{eqn: non x with input} becomes
\par\noindent\small
\begin{align}
\mathbf{x}_{k+1}=\mathbf{F}_{k}\mathbf{x}_{k}+\mathbf{B}_{k}\mathbf{u}_{k}+\overline{\mathbf{u}}_{k}+\mathbf{w}_{k},\label{eqn:supp 20}
\end{align}
\normalsize
where $\overline{\mathbf{u}}_{k}=f(\hat{\mathbf{x}}_{k|k},\hat{\mathbf{u}}_{k-1|k})-\mathbf{F}_{k}\hat{\mathbf{x}}_{k|k}-\mathbf{B}_{k}\hat{\mathbf{u}}_{k-1|k}$.

Define the least-squared error objective function $J_{k+1}=\overline{\bm{\Delta}}_{k+1}^{T}\mathbf{W}_{k+1}\overline{\bm{\Delta}}_{k+1}$ where $\overline{\bm{\Delta}}_{k+1}=[\bm{\Delta}_{1}^{T},\bm{\Delta}_{2}^{T},\hdots,\bm{\Delta}_{k+1}^{T}]\in\mathbb{R}^{p(k+1)\times 1}$ with $\bm{\Delta}_{i}=\mathbf{y}_{i}-h(\mathbf{x}_{i})$. The weighting matrix $\mathbf{W}_{k+1}\in\mathbb{R}^{p(k+1)\times p(k+1)}$ is defined using the inverse of process and measurement noise covariance matrices as in \cite[Eq.~(33)]{pan2010applying}.

Define a extended state vector $\mathbf{z}_{k}=[\mathbf{x}_{k}^{T},\mathbf{u}_{1}^{T},\mathbf{u}_{2}^{T},\hdots,\mathbf{u}_{k-1}^{T}]^{T}$. We first represent the objective function $J_{k+1}$ in terms of $\mathbf{z}_{k+1}$. Rearranging \eqref{eqn:supp 20}, we obtain $\mathbf{x}_{k}=\mathbf{F}_{k}^{-1}\mathbf{x}_{k+1}-\mathbf{F}_{k}^{-1}(\mathbf{B}_{k}\mathbf{u}_{k}+\overline{\mathbf{u}}_{k}+\mathbf{w}_{k})$. Replacing $k$ by $k-1$, we have $\mathbf{x}_{k-1}=\mathbf{F}_{k-1}^{-1}\mathbf{x}_{k}-\mathbf{F}_{k-1}^{-1}(\mathbf{B}_{k-1}\mathbf{u}_{k-1}+\overline{\mathbf{u}}_{k-1}+\mathbf{w}_{k-1})$ such that $\mathbf{x}_{k-1}=\mathbf{F}_{k-1}^{-1}\mathbf{F}_{k}^{-1}\mathbf{x}_{k+1}-\mathbf{F}_{k-1}^{-1}\mathbf{F}_{k}^{-1}(\mathbf{B}_{k}\mathbf{u}_{k}+\overline{\mathbf{u}}_{k}+\mathbf{w}_{k})-\mathbf{F}_{k-1}^{-1}(\mathbf{B}_{k-1}\mathbf{u}_{k-1}+\overline{\mathbf{u}}_{k-1}+\mathbf{w}_{k-1})$. Repeating the procedure for $i=k-2,k-3,\hdots,1$, we obtain
\par\noindent\small
\begin{align}
    \mathbf{x}_{i}=\bm{\Phi}_{k+1,i}^{-1}\mathbf{x}_{k+1}-\left(\sum_{j=i}^{k}\bm{\Phi}_{j+1,i}^{-1}(\mathbf{B}_{j}\mathbf{u}_{j}+\overline{\mathbf{u}}_{j}+\mathbf{w}_{j})\right),\label{eqn:supp 23}
\end{align}
\normalsize
for $i=1,2,\hdots,k$ with $\bm{\Phi}_{q,s}^{-1}\doteq\mathbf{F}_{s}^{-1}\mathbf{F}_{s+1}^{-1}\hdots\mathbf{F}_{q-1}^{-1}$ for $q>s$ and $\bm{\Phi}_{s,s}^{-1}=\mathbf{I}$. Using \eqref{eqn: non y withoutdf}, \eqref{eqn:supp 14} and \eqref{eqn:supp 23}, we have
\par\noindent\small
\begin{align}
    \bm{\Delta}_{i}=\mathbf{y}_{i}-\mathbf{H}_{i}\bm{\Phi}_{k+1,i}^{-1}\mathbf{x}_{k+1}+\mathbf{H}_{i}\left(\sum_{j=i}^{k}(\bm{\Phi}_{j+1,i}^{-1}\mathbf{B}_{j}\mathbf{u}_{j})\right)-\widetilde{\mathbf{u}}_{i|i-1},\label{eqn:supp 28}
\end{align}
\normalsize
for $i=1,2,\hdots,k+1$ with
\par\noindent\small
\begin{align}
    \widetilde{\mathbf{u}}_{i|i-1}=h(\hat{\mathbf{x}}_{i|i-1})-\mathbf{H}_{i}\left(\sum_{j=i}^{k}\bm{\Phi}_{j+1,i}^{-1}\overline{\mathbf{u}}_{j}+\hat{\mathbf{x}}_{i|i-1}\right).\label{eqn: supp 26}
\end{align}
\normalsize
Using \eqref{eqn:supp 28}, we express $\overline{\bm{\Delta}}_{k+1}$ as
\par\noindent\small
\begin{align}
    \overline{\bm{\Delta}}_{k+1}=\mathbf{Y}_{k+1}-\mathbf{A}_{z,k+1}\mathbf{z}_{k+1},\nonumber
\end{align}
\normalsize
where $\mathbf{Y}_{k+1}=[(\mathbf{y}_{1}-\widetilde{\mathbf{u}}_{1|0})^{T},(\mathbf{y}_{2}-\widetilde{\mathbf{u}}_{2|1})^{T},\hdots,(\mathbf{y}_{k+1}-\widetilde{\mathbf{u}}_{k+1|k})^{T}]$ and $\mathbf{A}_{z,k+1}=\begin{bmatrix}\widetilde{\mathbf{L}}_{k+1}&\widetilde{\mathbf{N}}_{k+1}\\\widetilde{\mathbf{H}}_{k+1}&\mathbf{0}_{p\times m}\end{bmatrix}$. Here,
\par\noindent\small
\begin{align}
    \widetilde{\mathbf{H}}_{k+1}=[\mathbf{H}_{k+1},\mathbf{0}_{p\times m(k-1)}],\label{eqn: supp 31}
\end{align}
\normalsize
while $\widetilde{\mathbf{N}}_{k+1}$ and $\widetilde{\mathbf{L}}_{k+1}$ are given by \cite[Eq.~32]{pan2010applying}.

Assume $p\geq m$ (condition for the existence of forward EKF-without-DF) and minimize the objective function $J_{k+1}$ with respect to the extended state vector $\mathbf{z}_{k+1}$ to yield the estimate $\hat{\mathbf{z}}_{k+1|k+1}=[\hat{\mathbf{x}}_{k+1|k+1}^{T},\hat{\mathbf{u}}_{1|k+1}^{T},\hat{\mathbf{u}}_{2|k+1}^{T},\hdots,\hat{\mathbf{u}}_{k|k+1}^{T}]$ given observations $\{\mathbf{y}_{i}\}_{1\leq i\leq k+1}$ as
\par\noindent\small
\begin{align}
    \hat{\mathbf{z}}_{k+1|k+1}=\mathbf{P}_{z,k+1}(\mathbf{A}_{z,k+1}^{T}\mathbf{W}_{k+1}\mathbf{Y}_{k+1}),\nonumber
\end{align}
where
\begin{align}
    \mathbf{P}_{z,k+1}=(\mathbf{A}_{z,k+1}^{T}\mathbf{W}_{k+1}\mathbf{A}_{z,k+1})^{-1}\label{eqn:supp 34}.
\end{align}
\normalsize

\subsection{Recursive Solutions for Extended State}\label{sec:extended}
In the following, we restate the updates to compute $\hat{\mathbf{z}}_{k+1|k+1}$ recursively as obtained in \cite[Appendix~A.1]{pan2010applying}. The procedure involves first expressing $\mathbf{A}_{z,k+1}$, $\mathbf{Y}_{k+1}$ and $\mathbf{W}_{k+1}$ in terms of $\mathbf{A}_{z,k}$, $\mathbf{Y}_{k}$ and $\mathbf{W}_{k}$ as follows
\par\noindent\small
\begin{align}
    &\mathbf{A}_{z,k+1}=\begin{bmatrix}
        \mathbf{A}_{z,k}\overline{\bm{\Phi}}_{k+1,k}^{-1}&-\mathbf{A}_{z,k}\overline{\bm{\Phi}}_{k+1,k}^{-1}\hat{\mathbf{B}}_{k}\\
        \widetilde{\mathbf{H}}_{k+1}&\mathbf{0}_{p\times m}
    \end{bmatrix},\label{eqn: A1 first part}\\
    &\mathbf{Y}_{k+1}=\begin{bmatrix}
        \mathbf{Y}_{k}+\mathbf{A}_{z,k}\overline{\bm{\Phi}}_{k+1,k}^{-1}\widetilde{\mathbf{U}}_{k}\\
        \mathbf{y}_{k+1}-\widetilde{\mathbf{u}}_{k+1|k}
    \end{bmatrix},\label{eqn: A2 first part}\\
    &\mathbf{W}_{k+1}=\begin{bmatrix}
        \widetilde{\mathbf{W}}_{k}&\mathbf{0}_{pk\times p}\\
        \mathbf{0}_{p\times pk}&\mathbf{R}_{k+1}^{-1}
    \end{bmatrix},\label{eqn: A3 above first part}
\end{align}
\normalsize
where $\mathbf{R}_{k}$ is the (time-varying) noise covariance matrix of measurement noise $\mathbf{v}_{k}$ in \eqref{eqn: non y withoutdf}. Also,
\par\noindent\small
\begin{align}
    &\overline{\bm{\Phi}}_{k+1,k}^{-1}=\begin{bmatrix}
        \bm{\Phi}_{k+1|k}^{-1}&\mathbf{0}_{n\times m(k-1)}\\
        \mathbf{0}_{m(k-1)\times n}&\mathbf{I}_{m(k-1)}
    \end{bmatrix},\label{eqn: A1 second part}\\
    &\widetilde{\mathbf{U}}_{k}=\begin{bmatrix}
        \overline{\mathbf{u}}_{k}\\
        \mathbf{0}_{m(k-1)\times 1}
    \end{bmatrix},\label{eqn: A2 second part}\\    &\widetilde{\mathbf{W}}_{k}=\left(\mathbf{W}_{k}^{-1}+\mathbf{A}_{z,k}\overline{\bm{\Phi}}_{k+1,k}^{-1}\widetilde{\mathbf{Q}}_{k}\overline{\bm{\Phi}}_{k+1,k}^{-T}\mathbf{A}_{z,k}^{T}\right)^{-1},\label{eqn: A3 above second part}\\
    &\widetilde{\mathbf{Q}}_{k}=\begin{bmatrix}
        \mathbf{Q}_{k}&\mathbf{0}_{n\times m(k-1)}\\
        \mathbf{0}_{m(k-1)\times n}&\mathbf{0}_{m(k-1)\times m(k-1)}
    \end{bmatrix},\label{eqn: A3 first part}\\
    &\hat{\mathbf{B}}_{k}=\begin{bmatrix}
        \mathbf{B}_{k}\\\mathbf{0}_{m(k-1)\times m}
    \end{bmatrix}.\label{eqn: A3 second part}
\end{align}
\normalsize
and $\mathbf{Q}_{k}$ is the (time-varying) noise covariance matrix of process noise $\mathbf{w}_{k}$ in \eqref{eqn: non x with input}. Here, $(\cdot)^{-T}$ denotes $((\cdot)^{-1})^{T}$. Basically, the relation between $\mathbf{A}_{z,k+1}$ and $\mathbf{A}_{z,k}$ is through $\bm{\Phi}_{k+1,k}$. The relations \eqref{eqn: A1 first part}-\eqref{eqn: A3 above first part} are then obtained through comparison following a similar procedure as in \cite{yang2007adaptive} for EKF-with-DF case.

Finally, consider the following matrix inversion formulas
\par\noindent\small
\begin{align}
    &\begin{bmatrix}
        \mathbf{A}_{1}&\mathbf{A}_{2}\\
        \mathbf{A}_{3}&\mathbf{A}_{4}
    \end{bmatrix}^{-1}=\begin{bmatrix}
    \mathbf{A}_{1}^{-1}+\mathbf{A}_{1}^{-1}\mathbf{A}_{2}\mathbf{A}_{q}^{-1}\mathbf{A}_{3}\mathbf{A}_{1}^{-1}&-\mathbf{A}_{1}^{-1}\mathbf{A}_{2}\mathbf{A}_{q}^{-1}\\
    -\mathbf{A}_{q}^{-1}\mathbf{A}_{3}\mathbf{A}_{1}^{-1}&\mathbf{A}_{q}^{-1}
    \end{bmatrix},\label{eqn: A4}
\end{align}    
and
\begin{align}
    &(\mathbf{C}_{1}+\mathbf{C}_{2}\mathbf{C}_{3}\mathbf{C}_{4})^{-1}\nonumber\\
    &=\mathbf{C}_{1}^{-1}-\mathbf{C}_{1}^{-1}\mathbf{C}_{2}(\mathbf{C}_{3}^{-1}+\mathbf{C}_{4}\mathbf{C}_{1}^{-1}\mathbf{C}_{2})^{-1}\mathbf{C}_{4}\mathbf{C}_{1}^{-1},\label{eqn: A5}
\end{align}
\normalsize
where $\mathbf{A}_{q}=\mathbf{A}_{4}-\mathbf{A}_{3}\mathbf{A}_{1}^{-1}\mathbf{A}_{2}$.

Substituting \eqref{eqn: A1 first part} and \eqref{eqn: A3 above first part} in \eqref{eqn:supp 34}, we obtain
\par\noindent\small
\begin{align}
    \mathbf{P}_{z,k+1}=\begin{bmatrix}
        \mathbf{P}_{1,k+1}&\mathbf{P}_{2,k+1}\\
       \mathbf{P}_{3,k+1}&\mathbf{P}_{4,k+1}        
    \end{bmatrix}^{-1},\nonumber
\end{align}
\normalsize
where\\
$\mathbf{P}_{1,k+1}=\overline{\bm{\Phi}}_{k+1,k}^{-T}\mathbf{A}_{z,k}^{T}\widetilde{\mathbf{W}}_{k}\mathbf{A}_{z,k}\overline{\bm{\Phi}}_{k+1,k}^{-1}+\widetilde{\mathbf{H}}_{k+1}^{T}\mathbf{R}_{k+1}^{-1}\widetilde{\mathbf{H}}_{k+1}$,\\
$\mathbf{P}_{2,k+1}=-\overline{\bm{\Phi}}_{k+1,k}^{-T}\mathbf{A}_{z,k}^{T}\widetilde{\mathbf{W}}_{k}\mathbf{A}_{z,k}\overline{\bm{\Phi}}_{k+1,k}^{-1}\hat{\mathbf{B}}_{k}$, $\mathbf{P}_{3,k+1}=-\hat{\mathbf{B}}_{k}^{T}\overline{\bm{\Phi}}_{k+1,k}^{-T}\mathbf{A}_{z,k}^{T}\widetilde{\mathbf{W}}_{k}\mathbf{A}_{z,k}\overline{\bm{\Phi}}_{k+1,k}^{-1}$ and $\mathbf{P}_{4,k+1}=\hat{\mathbf{B}}_{k}^{T}\overline{\bm{\Phi}}_{k+1,k}^{-T}\mathbf{A}_{z,k}^{T}\widetilde{\mathbf{W}}_{k}\mathbf{A}_{z,k}\overline{\bm{\Phi}}_{k+1,k}^{-1}\hat{\mathbf{B}}_{k}$. This is then simplified using \eqref{eqn: A4} and \eqref{eqn: A5}, similar to the procedure followed in \cite[Appendix~A]{yang2007adaptive}, to obtain the final recursive solutions for $\hat{\mathbf{z}}_{k+1|k+1}$. Define $\overline{\mathbf{P}}_{z,k+1}\doteq[\overline{\bm{\Phi}}_{k+1,k}^{-T}\mathbf{A}_{z,k}^{T}\widetilde{\mathbf{W}}_{k}\mathbf{A}_{z,k}\overline{\bm{\Phi}}_{k+1,k}^{-1}+\widetilde{\mathbf{H}}_{k+1}^{T}\mathbf{R}_{k+1}^{-1}\widetilde{\mathbf{H}}_{k+1}]^{-1}$. The recursive updates for $\hat{\mathbf{z}}_{k+1|k+1}$ are
\par\noindent\small
\begin{align}
    &\widetilde{\mathbf{P}}_{z,k+1}=\overline{\bm{\Phi}}_{k+1,k}\mathbf{P}_{z,k}\overline{\bm{\Phi}}_{k+1,k}^{T}+\widetilde{\mathbf{Q}}_{k},\label{eqn: A8 first part}\\
    &\mathbf{K}_{z,k+1}=\widetilde{\mathbf{P}}_{z,k+1}\widetilde{\mathbf{H}}_{k+1}^{T}(\mathbf{R}_{k+1}+\widetilde{\mathbf{H}}_{k+1}\widetilde{\mathbf{P}}_{z,k+1}\widetilde{\mathbf{H}}_{k+1}^{T})^{-1},\label{eqn: A8 second part}\\
    &\overline{\mathbf{P}}_{z,k+1}=(\mathbf{I}-\mathbf{K}_{z,k+1}\widetilde{\mathbf{H}}_{k+1})\widetilde{\mathbf{P}}_{z,k+1},\label{eqn: A9 first part}\\
    &\mathbf{S}_{k+1}=(\hat{\mathbf{B}}_{k}^{T}\widetilde{\mathbf{P}}_{z,k+1}^{-1}\mathbf{K}_{z,k+1}\widetilde{\mathbf{H}}_{k+1}\hat{\mathbf{B}}_{k})^{-1}\label{eqn: A9 second part}\\
    &\overline{\mathbf{z}}_{k+1}=(\overline{\bm{\Phi}}_{k+1,k}\hat{\mathbf{z}}_{k|k}+\widetilde{\mathbf{U}}_{k})\nonumber\\
    &\;\;\;+\mathbf{K}_{z,k+1}(\mathbf{y}_{k+1}-\widetilde{\mathbf{u}}_{k+1|k}-\widetilde{\mathbf{H}}_{k+1}(\overline{\bm{\Phi}}_{k+1,k}\hat{\mathbf{z}}_{k|k}+\widetilde{\mathbf{U}}_{k}))\label{eqn: A7}\\
    &\hat{\mathbf{u}}_{k|k+1}=-\mathbf{S}_{k+1}\hat{\mathbf{B}}_{k}^{T}\widetilde{\mathbf{P}}_{z,k+1}^{-1}(\overline{\bm{\Phi}}_{k+1,k}\hat{\mathbf{z}}_{k|k}+\widetilde{\mathbf{U}}_{k}-\overline{\mathbf{z}}_{k+1}),\label{eqn:A6}\\
    &\hat{\mathbf{z}}_{k+1|k+1}=\begin{bmatrix}
        \overline{\mathbf{z}}_{k+1}+\overline{\mathbf{P}}_{z,k+1}\widetilde{\mathbf{P}}_{z,k+1}^{-1}\hat{\mathbf{B}}_{k}\hat{\mathbf{u}}_{k|k+1}\\
        \hat{\mathbf{u}}_{k|k+1}
    \end{bmatrix}.\label{eqn: supp 35}
\end{align}
\normalsize

\subsection{Forward EKF-without-DF recursions}\label{sec:recursions}
From the extended state estimate $\hat{\mathbf{z}}_{k+1|k+1}$, we are only interested in $\hat{\mathbf{x}}_{k+1|k+1}$ and $\hat{\mathbf{u}}_{k|k+1}$ estimates of the current state $\mathbf{x}_{k+1}$ and unknown input $\mathbf{u}_{k}$, respectively. Hence, in this section, we simplify \eqref{eqn: A8 first part}-\eqref{eqn: supp 35} to obtain the forward EKF-without-DF recursions for computing $\hat{\mathbf{x}}_{k+1|k+1}$ and $\hat{\mathbf{u}}_{k|k+1}$. By definition, $\hat{\mathbf{z}}_{k+1|k+1}$ can be partitioned as $\hat{\mathbf{z}}_{k+1|k+1}=[\hat{\mathbf{x}}_{k+1|k+1}^{T},\hat{\mathbf{U}}_{k|k+1}^{T},\hat{\mathbf{u}}_{k|k+1}^{T}]$, where $\hat{\mathbf{U}}_{k|k+1}=[\hat{\mathbf{u}}_{1|k+1}^{T},\hat{\mathbf{u}}_{2|k+1}^{T},\hdots,\hat{\mathbf{u}}_{k-1|k+1}^{T}]$. Comparing with \eqref{eqn: supp 35}, we have
\par\noindent\small
\begin{align}
    \begin{bmatrix}
        \hat{\mathbf{x}}_{k+1|k+1}\\
        \hat{\mathbf{U}}_{k|k+1}    \end{bmatrix}=\overline{\mathbf{z}}_{k+1}+\overline{\mathbf{P}}_{z,k+1}\widetilde{\mathbf{P}}_{z,k+1}^{-1}\hat{\mathbf{B}}_{k}\hat{\mathbf{u}}_{k|k+1},\label{eqn:A11}
\end{align}
\normalsize
and $\hat{\mathbf{u}}_{k|k+1}$ is given by \eqref{eqn:A6}. Also, $\mathbf{P}_{z,k+1}$ can be partitioned as $\mathbf{P}_{z,k+1}=\begin{bmatrix}
    \mathbf{P}_{x,k+1|k+1}&\mathbf{P}_{xu,k+1|k+1}\\
    \mathbf{P}_{ux,k+1|k+1}&\mathbf{P}_{u,k+1|k+1}
\end{bmatrix}$. We can observe that the recursive solution for $\hat{\mathbf{x}}_{k+1|k+1}$ can be obtained from the top $n$ elements of the right side of \eqref{eqn:A11} while the recursive solution for $\hat{\mathbf{u}}_{k|k+1}$ is obtained by simplifying \eqref{eqn:A6}. Hence, we first obtain appropriate partitions for $\widetilde{\mathbf{P}}_{z,k+1}$, $\mathbf{K}_{z,k+1}$, $\overline{\mathbf{z}}_{k+1}$ and $\overline{\mathbf{P}}_{z,k+1}$ in the following section. The simplified forward EKF-without-DF recursions are then obtained in Section~\ref{subsec: recursive estimates} using these partitions. For simplicity, in the following, we omit the dimensions of the zero and identity matrices and represent the (appropriate-size) matrices as $\mathbf{0}$ and $\mathbf{I}$, respectively.

\subsubsection{Partitions for $\widetilde{\mathbf{P}}_{z,k+1}$, $\mathbf{K}_{z,k+1}$, $\overline{\mathbf{z}}_{k+1}$ and $\overline{\mathbf{P}}_{z,k+1}$}\label{subsec:partitions}
\textbf{1) Partition for $\widetilde{\mathbf{P}}_{z,k+1}$:} Applying \eqref{eqn: A4} to \eqref{eqn: A1 second part}, we obtain
\par\noindent\small
\begin{align}
    \overline{\bm{\Phi}}_{k+1,k}=\begin{bmatrix}
    \bm{\Phi}_{k+1,k}&\mathbf{0}\\ \mathbf{0}&\mathbf{I}
\end{bmatrix},\label{i}
\end{align}
\normalsize
Substituting \eqref{i}, \eqref{eqn: A3 first part} and $\mathbf{P}_{z,k}$ in \eqref{eqn: A8 first part}, we have
\par\noindent\small
\begin{align}
    &\widetilde{\mathbf{P}}_{z,k+1}=\begin{bmatrix}
        \widetilde{\mathbf{P}}_{z11,k+1}&\widetilde{\mathbf{P}}_{z12,k+1}\\
        \widetilde{\mathbf{P}}_{z21,k+1}&\widetilde{\mathbf{P}}_{z22,k+1}
    \end{bmatrix}\nonumber\\
    &=\begin{bmatrix}
        \bm{\Phi}_{k+1,k}\mathbf{P}_{x,k|k}\bm{\Phi}_{k+1,k}^{T}+\mathbf{Q}_{k}&\bm{\Phi}_{k+1,k}\mathbf{P}_{xu,k|k}\\
        \mathbf{P}_{ux,k|k}\bm{\Phi}_{k+1,k}^{T}&\mathbf{P}_{u,k|k}
    \end{bmatrix}.\label{A12}
\end{align}
\normalsize

\textbf{2) Partition for $\mathbf{K}_{z,k+1}$:} Using \eqref{eqn: supp 31} and the partitioned form of $\widetilde{\mathbf{P}}_{z,k+1}$ from \eqref{A12}, we have
\par\noindent\small
\begin{align}
    &\widetilde{\mathbf{P}}_{z,k+1}\widetilde{\mathbf{H}}^{T}_{k+1}=\begin{bmatrix}        \widetilde{\mathbf{P}}_{z11,k+1}\mathbf{H}^{T}_{k+1}\\\widetilde{\mathbf{P}}_{z21,k+1}\mathbf{H}^{T}_{k+1}\end{bmatrix},\nonumber\\    &\widetilde{\mathbf{H}}_{k+1}\widetilde{\mathbf{P}}_{z,k+1}\widetilde{\mathbf{H}}^{T}_{k+1}=\mathbf{H}_{k+1}\widetilde{\mathbf{P}}_{z11,k+1}\mathbf{H}^{T}_{k+1},\label{eqn:HpH}
\end{align}
\normalsize
Substituting this in \eqref{eqn: A8 second part}, we have
\par\noindent\small
\begin{align}
    &\mathbf{K}_{z,k+1}=\begin{bmatrix}
        \mathbf{K}_{x,k+1}\\
        \mathbf{K}_{u,k+1}
    \end{bmatrix}\nonumber\\
    &=\begin{bmatrix}
        \widetilde{\mathbf{P}}_{z11,k+1}\mathbf{H}^{T}_{k+1}(\mathbf{R}_{k+1}+\mathbf{H}_{k+1}\widetilde{\mathbf{P}}_{z11,k+1}\mathbf{H}^{T}_{k+1})^{-1}\\
        \widetilde{\mathbf{P}}_{z21,k+1}\mathbf{H}^{T}_{k+1}(\mathbf{R}_{k+1}+\mathbf{H}_{k+1}\widetilde{\mathbf{P}}_{z11,k+1}\mathbf{H}^{T}_{k+1})^{-1}
    \end{bmatrix}.\label{A13}
\end{align}
\normalsize

\textbf{3) Partition for $\overline{\mathbf{z}}_{k+1}$:} By definition, $\hat{\mathbf{z}}_{k|k}=[\hat{\mathbf{x}}_{k|k}^{T},\hat{\mathbf{U}}_{k|k}^{T}]^{T}$ with $\hat{\mathbf{U}}_{k|k}=[\hat{\mathbf{u}}_{1|k}^{T},\hat{\mathbf{u}}_{2|k}^{T},\hdots,\hat{\mathbf{u}}_{k-1|k}^{T}]$. Hence, using \eqref{i} and \eqref{eqn: A2 second part}, we have $\overline{\bm{\Phi}}_{k+1,k}\hat{\mathbf{z}}_{k|k}+\widetilde{\mathbf{U}}_{k}=\begin{bmatrix}
    \bm{\Phi}_{k+1,k}\hat{\mathbf{x}}_{k|k}+\overline{\mathbf{u}}_{k}\\\hat{\mathbf{U}}_{k|k}
\end{bmatrix}$. Denote $\widetilde{\mathbf{z}}_{k}\doteq\bm{\Phi}_{k+1,k}\hat{\mathbf{x}}_{k|k}+\overline{\mathbf{u}}_{k}$. Now, using \eqref{eqn: supp 31}, we have
\par\noindent\small
\begin{align}
    \widetilde{\mathbf{H}}_{k+1}(\overline{\bm{\Phi}}_{k+1,k}\hat{\mathbf{z}}_{k|k}+\widetilde{\mathbf{U}}_{k})=\mathbf{H}_{k+1}\widetilde{\mathbf{z}}_{k}.\label{ii}
\end{align}
\normalsize
Hence, \eqref{eqn: A7} becomes
\par\noindent\small
\begin{align}
    \overline{\mathbf{z}}_{k}=\begin{bmatrix}
        \widetilde{\mathbf{z}}_{k}\\\hat{\mathbf{U}}_{k|k}
    \end{bmatrix}+\mathbf{K}_{z,k+1}(\mathbf{y}_{k+1}-\widetilde{\mathbf{u}}_{k+1|k}-\mathbf{H}_{k+1}\widetilde{\mathbf{z}}_{k}),\nonumber
\end{align}
\normalsize
which on using the partitioned form of $\mathbf{K}_{z,k+1}$ from \eqref{A13} yields
\par\noindent\small
\begin{align}
    \overline{\mathbf{z}}_{k}=\begin{bmatrix}
        \widetilde{\mathbf{z}}_{k}+\mathbf{K}_{x,k+1}(\mathbf{y}_{k+1}-\widetilde{\mathbf{u}}_{k+1|k}-\mathbf{H}_{k+1}\widetilde{\mathbf{z}}_{k})\\
        \hat{\mathbf{U}}_{k|k}+\mathbf{K}_{u,k+1}(\mathbf{y}_{k+1}-\widetilde{\mathbf{u}}_{k+1|k}-\mathbf{H}_{k+1}\widetilde{\mathbf{z}}_{k})
    \end{bmatrix},\label{A14}
\end{align}
\normalsize
which is the corrected \cite[Eq.~A14]{pan2010applying}.

\textbf{4) Partition for $\overline{\mathbf{P}}_{z,k+1}$:} Substituting for $\mathbf{K}_{z,k+1}$ from \eqref{A13}, $\widetilde{\mathbf{H}}_{k+1}$ from \eqref{eqn: supp 31} and $\widetilde{\mathbf{P}}_{z,k+1}$ from \eqref{A12} in \eqref{eqn: A9 first part}, we have
\par\noindent\small
\begin{align}
    \overline{\mathbf{P}}_{z,k+1}=\begin{bmatrix}
       \overline{\mathbf{P}}_{1,k+1}&\overline{\mathbf{P}}_{2,k+1}\\
        \overline{\mathbf{P}}_{3,k+1}&\overline{\mathbf{P}}_{4,k+1}
    \end{bmatrix},\label{A15}
\end{align}
\normalsize
where $\overline{\mathbf{P}}_{1,k+1}=(\mathbf{I}-\mathbf{K}_{x,k+1}\mathbf{H}_{k+1})\widetilde{\mathbf{P}}_{z11,k+1}$, $\overline{\mathbf{P}}_{2,k+1}=(\mathbf{I}-\mathbf{K}_{x,k+1}\mathbf{H}_{k+1})\widetilde{\mathbf{P}}_{z12,k+1}$, $\overline{\mathbf{P}}_{3,k+1}=\widetilde{\mathbf{P}}_{e21,k+1}-\mathbf{K}_{u,k+1}\mathbf{H}_{k+1}\widetilde{\mathbf{P}}_{z11,k+1}$ and $\overline{\mathbf{P}}_{4,k+1}=\widetilde{\mathbf{P}}_{z22,k+1}-\mathbf{K}_{u,k+1}\mathbf{H}_{k+1}\widetilde{\mathbf{P}}_{z12,k+1}$.

\subsubsection{Recursions for $\hat{\mathbf{x}}_{k+1|k+1}$ and $\hat{\mathbf{u}}_{k|k+1}^{T}$}\label{subsec: recursive estimates}
\textbf{1) $\mathbf{S}_{k+1}$ update:} From \eqref{eqn: A8 second part}, we have
\par\noindent\small
\begin{align}
    \widetilde{\mathbf{P}}_{z,k+1}^{-1}\mathbf{K}_{z,k+1}=\widetilde{\mathbf{H}}^{T}_{k+1}(\mathbf{R}_{k+1}+\widetilde{\mathbf{H}}_{k+1}\widetilde{\mathbf{P}}_{z,k+1}\widetilde{\mathbf{H}}^{T}_{k+1})^{-1},\label{v}
\end{align}
\normalsize
Substituting in \eqref{eqn: A9 second part}, we have $\mathbf{S}_{k+1}=(\hat{\mathbf{B}}^{T}_{k}\widetilde{\mathbf{H}}^{T}_{k+1}(\mathbf{R}_{k+1}+\widetilde{\mathbf{H}}_{k+1}\widetilde{\mathbf{P}}_{z,k+1}\widetilde{\mathbf{H}}^{T}_{k+1})^{-1}\widetilde{\mathbf{H}}_{k+1}\hat{\mathbf{B}}_{k})^{-1}$.

Next, using \eqref{eqn: supp 31} and \eqref{eqn: A3 second part}, we have $\widetilde{\mathbf{H}}_{k+1}\hat{\mathbf{B}}_{k}=\mathbf{H}_{k+1}\mathbf{B}_{k}$. Using this and \eqref{eqn:HpH}, we have 
\par\noindent\small
\begin{align}
    \mathbf{S}_{k+1}=\left(\mathbf{B}_{k}^{T}\mathbf{H}_{k+1}^{T}(\mathbf{R}_{k+1}+\mathbf{H}_{k+1}\widetilde{\mathbf{P}}_{z11,k+1}\mathbf{H}_{k+1}^{T})^{-1}\mathbf{H}_{k+1}\mathbf{B}_{k}\right)^{-1}.\label{viS}
\end{align}
\normalsize
Now, from \eqref{A13}, $\mathbf{K}_{x,k+1}=\widetilde{\mathbf{P}}_{z11,k+1}\mathbf{H}_{k+1}^{T}(\mathbf{R}_{k+1}+\mathbf{H}_{k+1}\widetilde{\mathbf{P}}_{z11,k+1}\mathbf{H}_{k+1}^{T})^{-1}$ which implies $\mathbf{I}-\mathbf{H}_{k+1}\mathbf{K}_{x,k+1}=\mathbf{I}-\mathbf{H}_{k+1}\widetilde{\mathbf{P}}_{z11,k+1}\mathbf{H}_{k+1}^{T}(\mathbf{R}_{k+1}+\mathbf{H}_{k+1}\widetilde{\mathbf{P}}_{z11,k+1}\mathbf{H}_{k+1}^{T})^{-1}$. Comparing with \eqref{eqn: A5} with $\mathbf{C}_{1}^{-1}=\mathbf{I}$, $\mathbf{C}_{2}=\mathbf{H}_{k+1}\widetilde{\mathbf{P}}_{z11,k+1}\mathbf{H}_{k+1}^{T}$, $\mathbf{C}_{3}^{-1}=\mathbf{R}_{k+1}$ and $\mathbf{C}_{4}=\mathbf{I}$, we have $\mathbf{I}-\mathbf{H}_{k+1}\mathbf{K}_{x,k+1}=(\mathbf{I}+\mathbf{H}_{k+1}\widetilde{\mathbf{P}}_{z11,k+1}\mathbf{H}_{k+1}^{T}\mathbf{R}_{k+1}^{-1})^{-1}$ which implies $\mathbf{R}_{k+1}^{-1}(\mathbf{I}-\mathbf{H}_{k+1}\mathbf{K}_{x,k+1})=(\mathbf{R}_{k+1}+\mathbf{H}_{k+1}\widetilde{\mathbf{P}}_{z11,k+1}\mathbf{H}_{k+1}^{T})^{-1}$. Substituting in \eqref{viS}, we obtain
\par\noindent\small
\begin{align}
        \mathbf{S}_{k+1}=\left(\mathbf{B}_{k}^{T}\mathbf{H}_{k+1}^{T}\mathbf{R}_{k+1}^{-1}(\mathbf{I}-\mathbf{H}_{k+1}\mathbf{K}_{x,k+1})\mathbf{H}_{k+1}\mathbf{B}_{k}\right)^{-1}.\label{40}
\end{align}
\normalsize
Representing $\mathbf{S}_{k+1}$ by $\bm{\Sigma}^{u}_{k}$ and $\mathbf{K}_{x,k+1}$ by $\mathbf{K}^{x}_{k+1}$, \eqref{40} is the $\bm{\Sigma}^{u}_{k}$ update step of forward EKF-without-DF in Section~\ref{subsubsec:forward EKF without DF}.

\textbf{2) $\hat{\mathbf{u}}_{k|k+1}$ update:} From \eqref{eqn: A7} and \eqref{ii}, we have $\overline{\bm{\Phi}}_{k+1,k}\hat{\mathbf{z}}_{k|k}+\widetilde{\mathbf{U}}_{k}-\overline{\mathbf{z}}_{k+1}=-\mathbf{K}_{z,k+1}(\mathbf{y}_{k+1}-\widetilde{\mathbf{u}}_{k+1|k}-\mathbf{H}_{k+1|k}\widetilde{\mathbf{z}}_{k})$. Hence, from \eqref{eqn:A6}, $\hat{\mathbf{u}}_{k|k+1}=\mathbf{S}_{k+1}\hat{\mathbf{B}}_{k}^{T}\widetilde{\mathbf{P}}_{z,k+1}^{-1}\mathbf{K}_{z,k+1}(\mathbf{y}_{k+1}-\widetilde{\mathbf{u}}_{k+1|k}-\mathbf{H}_{k+1}\widetilde{\mathbf{z}}_{k})$. Again, substituting for $\widetilde{\mathbf{P}}_{z,k+1}^{-1}\mathbf{K}_{z,k+1}$ from \eqref{v}, we have $\hat{\mathbf{u}}_{k|k+1}=\mathbf{S}_{k+1}\hat{\mathbf{B}}_{k}^{T}\widetilde{\mathbf{H}}^{T}_{k+1}(\mathbf{R}_{k+1}+\widetilde{\mathbf{H}}_{k+1}\widetilde{\mathbf{P}}_{z,k+1}\widetilde{\mathbf{H}}^{T}_{k+1})^{-1}(\mathbf{y}_{k+1}-\widetilde{\mathbf{u}}_{k+1|k}-\mathbf{H}_{k+1}\widetilde{\mathbf{z}}_{k})$. Now, substituting $\widetilde{\mathbf{H}}_{k+1}\hat{\mathbf{B}}_{k}=\mathbf{H}_{k+1}\mathbf{B}_{k}$ and $(\mathbf{R}_{k+1}+\widetilde{\mathbf{H}}_{k+1}\widetilde{\mathbf{P}}_{z,k+1}\widetilde{\mathbf{H}}^{T}_{k+1})^{-1}=(\mathbf{R}_{k+1}+\mathbf{H}_{k+1}\widetilde{\mathbf{P}}_{z11,k+1}\mathbf{H}_{k+1}^{T})^{-1}=\mathbf{R}_{k+1}^{-1}(\mathbf{I}-\mathbf{H}_{k+1}\mathbf{K}_{x,k+1})$ as obtained in the previous step, we have
\par\noindent\small
\begin{align}
    \hat{\mathbf{u}}_{k|k+1}&=\mathbf{S}_{k+1}\mathbf{B}_{k}^{T}\mathbf{H}_{k+1}^{T}\mathbf{R}_{k+1}^{-1}(\mathbf{I}-\mathbf{H}_{k+1}\mathbf{K}_{x,k+1})\nonumber\\
    &\;\;\;\times(\mathbf{y}_{k+1}-\widetilde{\mathbf{u}}_{k+1|k}-\mathbf{H}_{k+1}\widetilde{\mathbf{z}}_{k}).\label{viu}
\end{align}
\normalsize
    
Now, we simplify $(\mathbf{y}_{k+1}-\widetilde{\mathbf{u}}_{k+1|k}-\mathbf{H}_{k+1}\widetilde{\mathbf{z}}_{k})$. From \eqref{eqn:supp 23}, using $i=k$, we have $\mathbf{x}_{k}=\bm{\Phi}^{-1}_{k+1,k}\mathbf{x}_{k+1}-\bm{\Phi}^{-1}_{k+1,k}(\mathbf{B}_{k}\mathbf{u}_{k}+\overline{\mathbf{u}}_{k}+\mathbf{w}_{k})$ which implies $\mathbf{x}_{k+1}=\bm{\Phi}_{k+1,k}\mathbf{x}_{k}+\mathbf{B}_{k}\mathbf{u}_{k}+\overline{\mathbf{u}}_{k}+\mathbf{w}_{k}$. But by definition of $\bm{\Phi}_{k+1,k}$, we have $\bm{\Phi}_{k+1,k}^{-1}=\mathbf{F}_{k}^{-1}$ such that $\bm{\Phi}_{k+1,k}=\mathbf{F}_{k}$. Hence, $\mathbf{x}_{k+1}=\mathbf{F}_{k}\mathbf{x}_{k}+\mathbf{B}_{k}\mathbf{u}_{k}+\overline{\mathbf{u}}_{k}+\mathbf{w}_{k}$. From \eqref{eqn:supp 20}, this is the linearized form of \eqref{eqn: non x with input}. We obtain the predicted state $\hat{\mathbf{x}}_{k+1|k}$ by substituting $\hat{\mathbf{x}}_{k|k}$ and $\hat{\mathbf{u}}_{k-1|k}$ (previous estimates) in place of $\mathbf{x}_{k}$ and $\mathbf{u}_{k}$, respectively, with the noise $\mathbf{w}_{k}$ taken as $\mathbf{0}$. Hence,
\par\noindent\small
\begin{align}
    \hat{\mathbf{x}}_{k+1|k}=\mathbf{F}_{k}\hat{\mathbf{x}}_{k|k}+\mathbf{B}_{k}\hat{\mathbf{u}}_{k-1|k}+\overline{\mathbf{u}}_{k}.\label{vii}
\end{align}
\normalsize
But similar to EKF, instead of the linearized approximation, we use the non-linear state transition function itself for state prediction for reduced errors, i.e.,
\par\noindent\small
\begin{align}
    \hat{\mathbf{x}}_{k+1|k}=f(\hat{\mathbf{x}}_{k|k},\hat{\mathbf{u}}_{k-1|k}).\label{38}
\end{align}
\normalsize
This is the prediction step \eqref{eqn: ekfwithoutdf predict} of forward EKF-without-DF.

Now, by definition, $\widetilde{\mathbf{z}}_{k}=\bm{\Phi}_{k+1,k}\hat{\mathbf{x}}_{k|k}+\overline{\mathbf{u}}_{k}=\mathbf{F}_{k}\hat{\mathbf{x}}_{k|k}+\overline{\mathbf{u}}_{k}$. Using \eqref{vii}, we have $\widetilde{\mathbf{z}}_{k}=\hat{\mathbf{x}}_{k+1|k}-\mathbf{B}_{k}\hat{\mathbf{u}}_{k-1|k}$. Using this along with \eqref{eqn: supp 26} with $i=k+1$, we have $\mathbf{y}_{k+1}-\widetilde{\mathbf{u}}_{k+1|k}-\mathbf{H}_{k+1}\widetilde{\mathbf{z}}_{k}=\mathbf{y}_{k+1}-(h(\hat{\mathbf{x}}_{k+1|k})-\mathbf{H}_{k+1}\hat{\mathbf{x}}_{k+1|k})-\mathbf{H}_{k+1}\widetilde{\mathbf{z}}_{k}=\mathbf{y}_{k+1}-h(\hat{\mathbf{x}}_{k+1|k})+\mathbf{H}_{k+1}\mathbf{B}_{k}\hat{\mathbf{u}}_{k-1|k}$. Substituting in \eqref{viu}, we have
\par\noindent\small
\begin{align}
    \hat{\mathbf{u}}_{k|k+1}&=\mathbf{S}_{k+1}\mathbf{B}_{k}^{T}\mathbf{H}_{k+1}^{T}\mathbf{R}_{k+1}^{-1}(\mathbf{I}-\mathbf{H}_{k+1}\mathbf{K}_{x,k+1})\nonumber\\
    &\;\;\times(\mathbf{y}_{k+1}-h(\hat{\mathbf{x}}_{k+1|k})+\mathbf{H}_{k+1}\mathbf{B}_{k}\hat{\mathbf{u}}_{k-1|k}).\label{43}
\end{align}
\normalsize
Denoting $\mathbf{K}^{u}_{k}=\mathbf{S}_{k+1}\mathbf{B}_{k}^{T}\mathbf{H}_{k+1}^{T}\mathbf{R}_{k+1}^{-1}(\mathbf{I}-\mathbf{H}_{k+1}\mathbf{K}_{x,k+1})$, \eqref{43} is the update step \eqref{eqn: ekfwithoutdf update u} of forward EKF-without-DF.

\textbf{3) $\hat{\mathbf{x}}_{k+1|k+1}$ update:} From \eqref{eqn: A9 first part}, \eqref{eqn: supp 31} and \eqref{A13}, we have $\overline{\mathbf{P}}_{z,k+1}\widetilde{\mathbf{P}}_{z,k+1}^{-1}=\begin{bmatrix}
    \mathbf{I}-\mathbf{K}_{x,k+1}\mathbf{H}_{k+1}&\mathbf{0}\\-\mathbf{K}_{u,k+1}\mathbf{H}_{k+1}&\mathbf{I}
    \end{bmatrix}$. Hence, using \eqref{eqn: A3 second part}, $\overline{\mathbf{P}}_{z,k+1}\widetilde{\mathbf{P}}_{z,k+1}^{-1}\hat{\mathbf{B}}_{k}=\begin{bmatrix}
    (\mathbf{I}-\mathbf{K}_{x,k+1}\mathbf{H}_{k+1})\mathbf{B}_{k}\\-\mathbf{K}_{u,k+1}\mathbf{H}_{k+1}\mathbf{B}_{k}\end{bmatrix}$. Substituting this and \eqref{A14} in \eqref{eqn:A11}, we have
\par\noindent\small
\begin{align}
    \begin{bmatrix}
        \hat{\mathbf{x}}_{k+1|k+1}\\\hat{\mathbf{U}}_{k|k+1}
    \end{bmatrix}&=\begin{bmatrix}
        \widetilde{\mathbf{z}}_{k}+\mathbf{K}_{x,k+1}(\mathbf{y}_{k+1}-\widetilde{\mathbf{u}}_{k+1|k}-\mathbf{H}_{k+1}\widetilde{\mathbf{z}}_{k})\\
        \hat{\mathbf{U}}_{k|k}+\mathbf{K}_{u,k+1}(\mathbf{y}_{k+1}-\widetilde{\mathbf{u}}_{k+1|k}-\mathbf{H}_{k+1}\widetilde{\mathbf{z}}_{k})
    \end{bmatrix}\nonumber\\
    &+\begin{bmatrix}
        (\mathbf{I}-\mathbf{K}_{x,k+1}\mathbf{H}_{k+1})\mathbf{B}_{k}\hat{\mathbf{u}}_{k|k+1}\\-\mathbf{K}_{u,k+1}\mathbf{H}_{k+1}\mathbf{B}_{k}\hat{\mathbf{u}}_{k|k+1}
    \end{bmatrix},\nonumber
\end{align}
\normalsize
which implies $\hat{\mathbf{x}}_{k+1|k+1}= \widetilde{\mathbf{z}}_{k}+\mathbf{K}_{x,k+1}(\mathbf{y}_{k+1}-\widetilde{\mathbf{u}}_{k+1|k}-\mathbf{H}_{k+1}\widetilde{\mathbf{z}}_{k})+(\mathbf{I}-\mathbf{K}_{x,k+1}\mathbf{H}_{k+1})\mathbf{B}_{k}\hat{\mathbf{u}}_{k|k+1}$.

Again, $\widetilde{\mathbf{z}}_{k}=\hat{\mathbf{x}}_{k+1|k}-\mathbf{B}_{k}\hat{\mathbf{u}}_{k-1|k}$ and $\mathbf{y}_{k+1}-\widetilde{\mathbf{u}}_{k+1|k}-\mathbf{H}_{k+1}\widetilde{\mathbf{z}}_{k}=\mathbf{y}_{k+1}-h(\hat{\mathbf{x}}_{k+1|k})+\mathbf{H}_{k+1}\mathbf{B}_{k}\hat{\mathbf{u}}_{k-1|k}$. Hence,
\par\noindent\small
\begin{align}
    &\hat{\mathbf{x}}_{k+1|k+1}=\hat{\mathbf{x}}_{k+1|k}-\mathbf{B}_{k}\hat{\mathbf{u}}_{k-1|k}+(\mathbf{I}-\mathbf{K}_{x,k+1}\mathbf{H}_{k+1})\mathbf{B}_{k}\hat{\mathbf{u}}_{k|k+1}\nonumber\\
    &\;\;\;+\mathbf{K}_{x,k+1}(\mathbf{y}_{k+1}-h(\hat{\mathbf{x}}_{k+1|k})+\mathbf{H}_{k+1}\mathbf{B}_{k}\hat{\mathbf{u}}_{k-1|k})\nonumber\\
    &=\hat{\mathbf{x}}_{k+1|k}+\mathbf{K}_{x,k+1}(\mathbf{y}_{k+1}-h(\hat{\mathbf{x}}_{k+1|k}))\nonumber\\
    &\;\;\;+(\mathbf{I}-\mathbf{K}_{x,k+1}\mathbf{H}_{k+1})\mathbf{B}_{k}(\hat{\mathbf{u}}_{k|k+1}-\hat{\mathbf{u}}_{k-1|k})\nonumber
\end{align}
\normalsize
Now, \cite{pan2010applying} assumed $\hat{\mathbf{u}}_{k|k+1}-\hat{\mathbf{u}}_{k-1|k}\approx\mathbf{0}$, i.e., the filter's unknown input estimate does not change abruptly (by a large value in one step) and hence,
\par\noindent\small
\begin{align}
    \hat{\mathbf{x}}_{k+1|k+1}=\hat{\mathbf{x}}_{k+1|k}+\mathbf{K}_{x,k+1}(\mathbf{y}_{k+1}-h(\hat{\mathbf{x}}_{k+1|k})).\label{42}
\end{align}
\normalsize
This is the update step \eqref{eqn: ekfwithoutdf update x} of the forward EKF-without-DF.

Also, define $\bm{\Sigma}^{x}_{k+1|k}\doteq\widetilde{\mathbf{P}}_{z11,k+1}$. From \eqref{A13}, we have
\par\noindent\small
\begin{align}
    \mathbf{K}_{x,k+1}=\bm{\Sigma}^{x}_{k+1|k}\mathbf{H}_{k+1}^{T}(\mathbf{R}_{k+1}+\mathbf{H}_{k+1}\bm{\Sigma}^{x}_{k+1|k}\mathbf{H}_{k+1}^{T})^{-1},\label{39}
\end{align}
\normalsize
which is the $\mathbf{K}^{x}_{k+1}$ (another notation for $\mathbf{K}_{x,k+1}$) update of the forward EKF-without-DF in Section~\ref{subsubsec:forward EKF without DF}.

From \eqref{A12}, $\bm{\Sigma}^{x}_{k+1|k}=\widetilde{\mathbf{P}}_{z11,k+1}=\bm{\Phi}_{k+1,k}\mathbf{P}_{x,k|k}\bm{\Phi}_{k+1,k}^{T}+\mathbf{Q}_{k}$. But $\bm{\Phi}_{k+1,k}=\mathbf{F}_{k}$. Hence,
\par\noindent\small
\begin{align}
    \bm{\Sigma}^{x}_{k+1|k}=\mathbf{F}_{k}\mathbf{P}_{x,k|k}\mathbf{F}_{k}^{T}+\mathbf{Q}_{k}.\label{41}
\end{align}
\normalsize
Representing $\bm{\Sigma}^{x}_{k}=\mathbf{P}_{x,k|k}$, \eqref{41} is the $\bm{\Sigma}^{x}_{k+1|k}$ update of the forward EKF-without-DF in Section~\ref{subsubsec:forward EKF without DF}.

\textbf{4) $\mathbf{P}_{x,k|k}$ update:} By definition, $\mathbf{P}_{z,k+1}$ can be partitioned as
\par\noindent\small
\begin{align}
    \mathbf{P}_{z,k+1}=\begin{bmatrix}
        \mathbf{P}_{x,k+1|k+1}&\mathbf{P}_{xu,k+1|k+1}\\\mathbf{P}_{ux,k+1|k+1}&\mathbf{P}_{u,k+1|k+1}.\label{viii}
    \end{bmatrix}
\end{align}
\normalsize
Substituting \eqref{eqn: A1 first part} and \eqref{eqn: A3 above first part} in \eqref{eqn:supp 34}, we have
\par\noindent\small
\begin{align}
    \mathbf{P}_{z,k+1}=\begin{bmatrix}
        \mathbf{P}^{1}_{z,k+1}&\mathbf{P}^{2}_{z,k+1}\\\mathbf{P}^{3}_{z,k+1}&\mathbf{P}^{4}_{z,k+1}
    \end{bmatrix}^{-1},\label{inv}
\end{align}
\normalsize
where
\par\noindent\small
\begin{align}
 &\mathbf{P}^{1}_{z,k+1}=\overline{\bm{\Phi}}_{k+1,k}^{-T}\mathbf{A}_{z,k}^{T}\widetilde{\mathbf{W}}_{k}\mathbf{A}_{z,k}\overline{\bm{\Phi}}_{k+1,k}^{-1}+\widetilde{\mathbf{H}}_{k+1}^{T}\mathbf{R}_{k+1}^{-1}\widetilde{\mathbf{H}}_{k+1},\nonumber\\
 &\mathbf{P}^{2}_{z,k+1}=-\overline{\bm{\Phi}}_{k+1,k}^{-T}\mathbf{A}_{z,k}^{T}\widetilde{\mathbf{W}}_{k}\mathbf{A}_{z,k}\overline{\bm{\Phi}}_{k+1,k}^{-1}\hat{\mathbf{B}}_{k},\nonumber\\
 &\mathbf{P}^{3}_{z,k+1}=-\hat{\mathbf{B}}_{k}^{T}\overline{\bm{\Phi}}_{k+1,k}^{-T}\mathbf{A}_{z,k}^{T}\widetilde{\mathbf{W}}_{k}\mathbf{A}_{z,k}\overline{\bm{\Phi}}_{k+1,k}^{-1},\nonumber\\
 &\mathbf{P}^{4}_{z,k+1}=\hat{\mathbf{B}}_{k}^{T}\overline{\bm{\Phi}}_{k+1,k}^{-T}\mathbf{A}_{z,k}^{T}\widetilde{\mathbf{W}}_{k}\mathbf{A}_{z,k}\overline{\bm{\Phi}}_{k+1,k}^{-1}\hat{\mathbf{B}}_{k}.\nonumber
\end{align}
\normalsize
Let
\par\noindent\small
\begin{align}
    \mathbf{P}_{z,k+1}=\begin{bmatrix}
        \mathbf{P}_{z11,k+1}&\mathbf{P}_{z12,k+1}\\\mathbf{P}_{z21,k+1}&\mathbf{P}_{z22,k+1}
    \end{bmatrix}.\label{ix}
\end{align}
\normalsize
In \cite[Appendix~A.1]{pan2010applying} (detailed steps in \cite[Appendix~A]{yang2007adaptive}) to obtain \eqref{eqn: A8 first part}-\eqref{eqn: supp 35}, \eqref{eqn: A4} is used to simplify \eqref{inv} after defining $\mathbf{S}_{k+1}\doteq\mathbf{P}_{z22,k+1}$. This yields
\par\noindent\small
\begin{align}
    \mathbf{P}_{z11,k+1}&=\overline{\mathbf{P}}_{z,k+1}+\overline{\mathbf{P}}_{z,k+1}\overline{\bm{\Phi}}_{k+1,k}^{-T}\mathbf{A}_{z,k}^{T}\widetilde{\mathbf{W}}_{k}\mathbf{A}_{z,k}\overline{\bm{\Phi}}_{k+1,k}^{-1}\hat{\mathbf{B}}_{k}\mathbf{S}_{k+1}\nonumber\\
    &\;\;\;\times\hat{\mathbf{B}}^{T}_{k}\overline{\bm{\Phi}}_{k+1,k}^{-T}\mathbf{A}_{z,k}^{T}\widetilde{\mathbf{W}}_{k}\mathbf{A}_{z,k}\overline{\bm{\Phi}}_{k+1,k}^{-1}\overline{\mathbf{P}}_{z,k+1}\nonumber
\end{align}
\normalsize

Now, \eqref{viii} and \eqref{ix} are two different partitions of $\mathbf{P}_{z,k+1}$. Comparing the dimensions, we observe that $\mathbf{P}_{x,k+1|k+1}$ is the upper-left submatrix of $\mathbf{P}_{z11,k+1}$. Denote $\mathbf{M}:=\overline{\mathbf{P}}_{z,k+1}\overline{\bm{\Phi}}_{k+1,k}^{-T}\mathbf{A}_{z,k}^{T}\widetilde{\mathbf{W}}_{k}\mathbf{A}_{z,k}\overline{\bm{\Phi}}_{k+1,k}^{-1}$. Then, $\mathbf{P}_{z11,k+1}=\overline{\mathbf{P}}_{z,k+1}+\mathbf{M}\hat{\mathbf{B}}_{k}\mathbf{S}_{k+1}\hat{\mathbf{B}}^{T}_{k}\mathbf{M}^{T}$. Denote the partitions of $\mathbf{M}=\begin{bmatrix}
        \mathbf{M}_{11}&\mathbf{M}_{12}\\\mathbf{M}_{21}&\mathbf{M}_{22}
\end{bmatrix}$. Hence, substituting for $\hat{\mathbf{B}}_{k}$ from \eqref{eqn: A3 second part}, we have
\par\noindent\small
\begin{align}
    &\mathbf{M}\hat{\mathbf{B}}_{k}\mathbf{S}_{k+1}\hat{\mathbf{B}}^{T}_{k}\mathbf{M}^{T}\nonumber\\
    &=\begin{bmatrix}
    \mathbf{M}_{11}\mathbf{B}_{k}\mathbf{S}_{k+1}\mathbf{B}_{k}^{T}\mathbf{M}_{11}^{T}&\mathbf{M}_{11}\mathbf{B}_{k}\mathbf{S}_{k+1}\mathbf{B}_{k}^{T}\mathbf{M}_{21}^{T}\\
    \mathbf{M}_{21}\mathbf{B}_{k}\mathbf{S}_{k+1}\mathbf{B}_{k}^{T}\mathbf{M}_{11}^{T}&\mathbf{M}_{21}\mathbf{B}_{k}\mathbf{S}_{k+1}\mathbf{B}_{k}^{T}\mathbf{M}_{21}^{T},\nonumber
    \end{bmatrix}
\end{align}
\normalsize
whose upper-left submatrix is $\mathbf{M}_{11}\mathbf{B}_{k}\mathbf{S}_{k+1}\mathbf{B}_{k}^{T}\mathbf{M}_{11}^{T}$. Also, from \eqref{A15}, the upper-left submatrix of $\overline{\mathbf{P}}_{z,k+1}$ is $(\mathbf{I}-\mathbf{K}_{x,k+1}\mathbf{H}_{k+1})\widetilde{\mathbf{P}}_{z11,k+1}$. Hence,
\par\noindent\small
\begin{align}
    \mathbf{P}_{x,k+1|k+1}=(\mathbf{I}-\mathbf{K}_{x,k+1}\mathbf{H}_{k+1})\widetilde{\mathbf{P}}_{z11,k+1}+\mathbf{M}_{11}\mathbf{B}_{k}\mathbf{S}_{k+1}\mathbf{B}_{k}^{T}\mathbf{M}_{11}^{T},\label{x}
\end{align}
\normalsize
where $\mathbf{M}_{11}$ is the upper-left submatrix of $\overline{\mathbf{P}}_{z,k+1}\overline{\bm{\Phi}}_{k+1,k}^{-T}\mathbf{A}_{z,k}^{T}\widetilde{\mathbf{W}}_{k}\mathbf{A}_{z,k}\overline{\bm{\Phi}}_{k+1,k}^{-1}$.

Now, we compute $\mathbf{M}_{11}$. Using \eqref{eqn: A5} in \eqref{eqn: A3 above second part}, we have
\par\noindent\small
\begin{align}
    \widetilde{\mathbf{W}}_{k}&=\mathbf{W}_{k}-\mathbf{W}_{k}\mathbf{A}_{z,k}\overline{\bm{\Phi}}_{k+1,k}^{-1}\left(\widetilde{\mathbf{Q}}^{-1}_{k}+\overline{\bm{\Phi}}_{k+1,k}^{-T}\mathbf{A}_{z,k}^{T}\mathbf{W}_{k}\mathbf{A}_{z,k}\overline{\bm{\Phi}}_{k+1,k}^{-1}\right)^{-1}\nonumber\\
    &\;\;\;\times\overline{\bm{\Phi}}_{k+1,k}^{-T}\mathbf{A}_{z,k}^{T}\mathbf{W}_{k}.\nonumber
\end{align}
\normalsize
Now, from \eqref{eqn:supp 34}, $\mathbf{A}_{z,k}^{T}\mathbf{W}_{k}\mathbf{A}_{z,k}=\mathbf{P}_{z,k}^{-1}$ such that
\par\noindent\small
\begin{align}
    &\widetilde{\mathbf{W}}_{k}=\mathbf{W}_{k}-\mathbf{W}_{k}\mathbf{A}_{z,k}\overline{\bm{\Phi}}_{k+1,k}^{-1}\left(\widetilde{\mathbf{Q}}^{-1}_{k}+\overline{\bm{\Phi}}_{k+1,k}^{-T}\mathbf{P}_{z,k}^{-1}\overline{\bm{\Phi}}_{k+1,k}^{-1}\right)^{-1}\nonumber\\
    &\;\;\;\;\;\times\overline{\bm{\Phi}}_{k+1,k}^{-T}\mathbf{A}_{z,k}^{T}\mathbf{W}_{k}\nonumber\\
    &=\mathbf{W}_{k}-\mathbf{W}_{k}\mathbf{A}_{z,k}\left[\overline{\bm{\Phi}}_{k+1,k}^{T}(\widetilde{\mathbf{Q}}^{-1}_{k}+\overline{\bm{\Phi}}_{k+1,k}^{-T}\mathbf{P}_{z,k}^{-1}\overline{\bm{\Phi}}_{k+1,k}^{-1})\overline{\bm{\Phi}}_{k+1,k}\right]^{-1}\nonumber\\
    &\;\;\;\;\;\times\mathbf{A}_{z,k}^{T}\mathbf{W}_{k}\nonumber\\
    &=\mathbf{W}_{k}-\mathbf{W}_{k}\mathbf{A}_{z,k}\left[\overline{\bm{\Phi}}_{k+1,k}^{T}\widetilde{\mathbf{Q}}^{-1}_{k}\overline{\bm{\Phi}}_{k+1,k}+\mathbf{P}_{z,k}^{-1}\right]^{-1}\mathbf{A}_{z,k}^{T}\mathbf{W}_{k}.\nonumber
\end{align}
\normalsize
Hence,
\par\noindent\small
\begin{align}
    &\mathbf{A}_{z,k}^{T}\widetilde{\mathbf{W}}_{k}\mathbf{A}_{z,k}=\mathbf{A}_{z,k}^{T}\mathbf{W}_{k}\mathbf{A}_{z,k}\nonumber\\
    &-\mathbf{A}_{z,k}^{T}\mathbf{W}_{k}\mathbf{A}_{z,k}\left[\overline{\bm{\Phi}}_{k+1,k}^{T}\widetilde{\mathbf{Q}}^{-1}_{k}\overline{\bm{\Phi}}_{k+1,k}+\mathbf{P}_{z,k}^{-1}\right]^{-1}\mathbf{A}_{z,k}^{T}\mathbf{W}_{k}\mathbf{A}_{z,k}\nonumber
\end{align}
\normalsize
Again, using $\mathbf{A}_{z,k}^{T}\mathbf{W}_{k}\mathbf{A}_{z,k}=\mathbf{P}_{z,k}^{-1}$, we have $\mathbf{A}_{z,k}^{T}\widetilde{\mathbf{W}}_{k}\mathbf{A}_{z,k}=\mathbf{P}_{z,k}^{-1}-\mathbf{P}_{z,k}^{-1}\left[\overline{\bm{\Phi}}_{k+1,k}^{T}\widetilde{\mathbf{Q}}^{-1}_{k}\overline{\bm{\Phi}}_{k+1,k}+\mathbf{P}_{z,k}^{-1}\right]^{-1}\mathbf{P}_{z,k}^{-1}$. Comparing with \eqref{eqn: A5} with $\mathbf{C}_{1}^{-1}=\mathbf{P}_{z,k}^{-1}$, $\mathbf{C}_{2}=\mathbf{I}$, $\mathbf{C}_{3}^{-1}=\overline{\bm{\Phi}}_{k+1,k}^{T}\widetilde{\mathbf{Q}}^{-1}_{k}\overline{\bm{\Phi}}_{k+1,k}$ and $\mathbf{C}_{4}=\mathbf{I}$, we have $\mathbf{A}_{z,k}^{T}\widetilde{\mathbf{W}}_{k}\mathbf{A}_{z,k}=\left(\mathbf{P}_{z,k}+(\overline{\bm{\Phi}}_{k+1,k}^{T}\widetilde{\mathbf{Q}}^{-1}_{k}\overline{\bm{\Phi}}_{k+1,k})^{-1}\right)^{-1}=(\mathbf{P}_{z,k}+\overline{\bm{\Phi}}_{k+1,k}^{-1}\widetilde{\mathbf{Q}}_{k}\overline{\bm{\Phi}}_{k+1,k}^{-T})^{-1}$. Now,
\par\noindent\small
\begin{align}
    &\overline{\bm{\Phi}}_{k+1,k}^{-T}\mathbf{A}_{z,k}^{T}\widetilde{\mathbf{W}}_{k}\mathbf{A}_{z,k}\overline{\bm{\Phi}}_{k+1,k}^{-1}\nonumber\\
    &=\overline{\bm{\Phi}}_{k+1,k}^{-T}(\mathbf{P}_{z,k}+\overline{\bm{\Phi}}_{k+1,k}^{-1}\widetilde{\mathbf{Q}}_{k}\overline{\bm{\Phi}}_{k+1,k}^{-T})^{-1}\overline{\bm{\Phi}}_{k+1,k}^{-1}\nonumber\\
    &=\left[\overline{\bm{\Phi}}_{k+1,k}(\mathbf{P}_{z,k}+\overline{\bm{\Phi}}_{k+1,k}^{-1}\widetilde{\mathbf{Q}}_{k}\overline{\bm{\Phi}}_{k+1,k}^{-T})\overline{\bm{\Phi}}_{k+1,k}^{T}\right]^{-1}\nonumber\\
    &=[\overline{\bm{\Phi}}_{k+1,k}\mathbf{P}_{z,k}\overline{\bm{\Phi}}_{k+1,k}^{T}+\widetilde{\mathbf{Q}}_{k}]^{-1}=\widetilde{\mathbf{P}}_{z,k+1}^{-1},\nonumber
\end{align}
\normalsize
using \eqref{eqn: A8 first part}. Hence, $\mathbf{M}=\overline{\mathbf{P}}_{z,k+1}\overline{\bm{\Phi}}_{k+1,k}^{-T}\mathbf{A}_{z,k}^{T}\widetilde{\mathbf{W}}_{k}\mathbf{A}_{z,k}\overline{\bm{\Phi}}_{k+1,k}^{-1}=\overline{\mathbf{P}}_{z,k+1}\widetilde{\mathbf{P}}_{z,k+1}^{-1}=\mathbf{I}-\mathbf{K}_{z,k+1}\widetilde{\mathbf{H}}_{k+1}$ using \eqref{eqn: A9 first part}. Hence, substituting for $\mathbf{K}_{z,k+1}$ from \eqref{A13} and $\widetilde{\mathbf{H}}_{k+1}$ from \eqref{eqn: supp 31}, we obtain the submatrix $\mathbf{M}_{11}=\mathbf{I}-\mathbf{K}_{x,k+1}\mathbf{H}_{k+1}$. Using this in \eqref{x}, we have
\par\noindent\small
\begin{align}
    &\mathbf{P}_{x,k+1|k+1}=(\mathbf{I}-\mathbf{K}_{x,k+1}\mathbf{H}_{k+1})\widetilde{\mathbf{P}}_{z11,k+1}\nonumber\\
    &+(\mathbf{I}-\mathbf{K}_{x,k+1}\mathbf{H}_{k+1})\mathbf{B}_{k}\mathbf{S}_{k+1}\mathbf{B}_{k}^{T}(\mathbf{I}-\mathbf{K}_{x,k+1}\mathbf{H}_{k+1})^{T}\nonumber
\end{align}
\normalsize
Using $\bm{\Sigma}^{x}_{k+1|k}=\widetilde{\mathbf{P}}_{z11,k+1}$, we obtain
\par\noindent\small
\begin{align}
    \mathbf{P}_{x,k+1|k+1}&=(\mathbf{I}-\mathbf{K}_{x,k+1}\mathbf{H}_{k+1})\nonumber\\
    &\;\;\;\times[\bm{\Sigma}^{x}_{k+1|k}+\mathbf{B}_{k}\mathbf{S}_{k+1}\mathbf{B}_{k}^{T}(\mathbf{I}-\mathbf{K}_{x,k+1}\mathbf{H}_{k+1}^{T}].\label{44}
\end{align}
\normalsize
With $\bm{\Sigma}^{x}_{k+1}=\mathbf{P}_{x,k+1|k+1}$, $\mathbf{K}^{x}_{k+1}=\mathbf{K}_{x,k+1}$ and $\bm{\Sigma}^{u}_{k}=\mathbf{S}_{k+1}$, \eqref{44} is the $\bm{\Sigma}^{x}_{k+1}$ update of forward EKF-without-DF in Section~\ref{subsubsec:forward EKF without DF}. The updates \eqref{38}, \eqref{39}, \eqref{40}, \eqref{41}, \eqref{42}, \eqref{43} and \eqref{44} are the final forward EKF-without-DF recursions.

\section{Proof of Theorem \ref{theorem: inverse kf without DF}}
\label{App-thm-inverse kf without DF}
Under the stability assumption of the forward filter, $\widetilde{\mathbf{F}}_{k}$ and $\mathbf{E}_{k}$ converge to $\overline{\mathbf{F}}$ and $\overline{\mathbf{E}}$, respectively, where $\overline{\mathbf{F}}=(\mathbf{I}-\overline{\mathbf{K}}\mathbf{H})(\mathbf{I}-\mathbf{B}\overline{\mathbf{M}}\mathbf{H})\mathbf{F}$ and $\overline{\mathbf{E}}=\mathbf{B}\overline{\mathbf{M}}-\overline{\mathbf{K}}\mathbf{HB}\overline{\mathbf{M}}+\overline{\mathbf{K}}$, obtained by replacing $\mathbf{K}_{k+1}$ and $\mathbf{M}_{k+1}$ by the limiting matrices $\overline{\mathbf{K}}$ and $\overline{\mathbf{M}}$, respectively, in $\widetilde{\mathbf{F}}_{k}$ and $\mathbf{E}_{k}$. In this limiting case, the state transition equation \eqref{eqn: state for kfwithoutdf} becomes $\hat{\mathbf{x}}_{k+1}=\overline{\mathbf{F}}\hat{\mathbf{x}}_{k}+\overline{\mathbf{E}}\mathbf{Hx}_{k+1}+\overline{\mathbf{E}}\mathbf{v}_{k+1}$.
From \eqref{eqn: inverse kfwithoutdf covariance predict}, \eqref{eqn: inverse kfwithoutdf gain}, and \eqref{eqn: inverse kfwithoutdf covariance update} and substituting the limiting matrices, the Riccati equation $\overline{\bm{\Sigma}}_{k+1|k}=\overline{\mathbf{F}}\left[\overline{\bm{\Sigma}}_{k|k-1}-\overline{\bm{\Sigma}}_{k|k-1}\mathbf{G}^{T}(\mathbf{G}\overline{\bm{\Sigma}}_{k|k-1}\mathbf{G}^{T}+\overline{\mathbf{R}})^{-1}\mathbf{G}\overline{\bm{\Sigma}}_{k|k-1}\right]\overline{\mathbf{F}}^{T}+\overline{\bm{Q}}$ is obtained,
where $\overline{\mathbf{Q}}=\overline{\mathbf{E}}\mathbf{R}\overline{\mathbf{E}}^{T}$. For the forward filter to be stable, covariance $\mathbf{R}$ needs to be p.d.\cite{fang2012on} and hence, $\overline{\mathbf{Q}}$ is a p.s.d. matrix. With $\overline{\mathbf{R}}$ being p.d. and the observability and controllability assumptions, $\overline{\bm{\Sigma}}_{k|k-1}$ tends to a unique p.d. matrix $\overline{\bm{\Sigma}}$ satisfying $\overline{\bm{\Sigma}}=\overline{\mathbf{F}}[\overline{\bm{\Sigma}}-\overline{\bm{\Sigma}}\mathbf{G}^{T}\left(\mathbf{G}\overline{\bm{\Sigma}}\mathbf{G}^{T}+\overline{\mathbf{R}}\right)^{-1}\mathbf{G}\overline{\bm{\Sigma}}]\overline{\mathbf{F}}^{T}+\overline{\mathbf{Q}}$,
and $\overline{\mathbf{F}}-\overline{\mathbf{F}}\overline{\bm{\Sigma}}\mathbf{G}^{T}(\mathbf{G}\overline{\bm{\Sigma}}\mathbf{G}^{T}+\overline{\mathbf{R}})^{-1}\mathbf{G}$ has eigenvalues strictly within the unit circle. These results follow directly from the application of \cite[Proposition 4.1, Sec. 4.1]{bertsekas1995dynamic} similar to the stability and convergence results for the standard KF for linear systems \cite[Appendix E.4]{bertsekas1995dynamic}.

In this limiting case, the inverse filter prediction and update equations take the following asymptotic form
\par\noindent\small
\begin{align*}
&\doublehat{\mathbf{x}}_{k+1|k}=\overline{\mathbf{F}}\doublehat{\mathbf{x}}_{k}+\overline{\mathbf{E}}\mathbf{Hx}_{k+1},\\
&\doublehat{\mathbf{x}}_{k+1}=\doublehat{\mathbf{x}}_{k+1|k}+\overline{\bm{\Sigma}}\mathbf{G}^{T}(\mathbf{G}\overline{\bm{\Sigma}}\mathbf{G}^{T}+\overline{\mathbf{R}})^{-1}(\mathbf{a}_{k+1}-\mathbf{G}\doublehat{\mathbf{x}}_{k+1|k}).
\end{align*}
\normalsize
Denoting the inverse filter's one-step prediction error as $\overline{\mathbf{e}}_{k+1|k}\doteq\hat{\mathbf{x}}_{k+1}-\doublehat{\mathbf{x}}_{k+1|k}$, the error dynamics for the inverse filter is obtained from this asymptotic form using \eqref{eqn: linear a} as
\par\noindent\small
\begin{align*}
\overline{\mathbf{e}}_{k+1|k}&=\left(\overline{\mathbf{F}}-\overline{\mathbf{F}}\overline{\bm{\Sigma}}\mathbf{G}^{T}(\mathbf{G}\overline{\bm{\Sigma}}\mathbf{G}^{T}+\overline{\mathbf{R}})^{-1}\mathbf{G}\right)\overline{\mathbf{e}}_{k|k-1}\nonumber\\ &\hspace{4mm}-\overline{\mathbf{F}}\overline{\bm{\Sigma}}\mathbf{G}^{T}(\mathbf{G}\overline{\bm{\Sigma}}\mathbf{G}^{T}+\overline{\mathbf{R}})^{-1}\bm{\epsilon}_{k}+\overline{\mathbf{E}}\mathbf{v}_{k+1}.
\end{align*}
\normalsize

Since $\overline{\mathbf{F}}-\overline{\mathbf{F}}\overline{\bm{\Sigma}}\mathbf{G}^{T}(\mathbf{G}\overline{\bm{\Sigma}}\mathbf{G}^{T}+\overline{\mathbf{R}})^{-1}\mathbf{G}$ has eigenvalues strictly within the unit circle, this error dynamics is asymptotically stable.

\vspace{-8pt}
\section{Proof of Theorem~\ref{theorem:Forward ekf stable unknown matrix}}
\label{App-thm-Forward ekf stable unknown matrix}
For simplicity, we consider the case of $n\geq p$ with $\mathbf{U}^{xy}_{k+1}\in\mathbb{R}^{n\times n}$. It is trivial to show that the proof remains valid for $n<p$ as well. Using the expressions for $\bm{\Sigma}^{xy}_{k+1}$ and $\mathbf{S}_{k+1}$, we have
\par\noindent\small
\begin{align*}
\mathbf{K}_{k+1}&=\bm{\Sigma}_{k+1|k}\mathbf{U}^{xy}_{k+1}\mathbf{H}_{k+1}^{T}\mathbf{U}^{y}_{k+1}\\
&\hspace{1.0cm}\times\left(\mathbf{U}^{y}_{k+1}\mathbf{H}_{k+1}\bm{\Sigma}_{k+1|k}\mathbf{H}_{k+1}^{T}\mathbf{U}^{y}_{k+1}+\hat{\mathbf{R}}_{k+1}\right)^{-1},\\
\bm{\Sigma}_{k+1}&=\bm{\Sigma}_{k+1|k}-\bm{\Sigma}_{k+1|k}\mathbf{U}^{xy}_{k+1}\mathbf{H}_{k+1}^{T}\mathbf{U}^{y}_{k+1}\\
&\times\left(\mathbf{U}^{y}_{k+1}\mathbf{H}_{k+1}\bm{\Sigma}_{k+1|k}\mathbf{H}_{k+1}^{T}\mathbf{U}^{y}_{k+1}+\hat{\mathbf{R}}_{k+1}\right)^{-1}\\
&\times\mathbf{U}^{y}_{k+1}\mathbf{H}_{k+1}(\mathbf{U}^{xy}_{k+1})^{T}\bm{\Sigma}_{k+1|k}.
\end{align*}
\normalsize

Define $V_{k}(\widetilde{\mathbf{x}}_{k|k-1})=\widetilde{\mathbf{x}}_{k|k-1}^{T}\bm{\Sigma}_{k|k-1}^{-1}\widetilde{\mathbf{x}}_{k|k-1}$. Using the bounds assumed on $\bm{\Sigma}_{k|k-1}$, we have for all $k\geq 0$
\par\noindent\small
\begin{align*}
    \frac{1}{\overline{\sigma}}\|\widetilde{\mathbf{x}}_{k|k-1}\|_{2}^{2}\leq V_{k}(\widetilde{\mathbf{x}}_{k|k-1})\leq\frac{1}{\underline{\sigma}}\|\widetilde{\mathbf{x}}_{k|k-1}\|_{2}^{2}.
\end{align*}
\normalsize
Hence, the first condition of Lemma \ref{lemma:exponential boundedness} is satisfied with $v_{\textrm{min}}=1/\overline{\sigma}$ and $v_{\textrm{max}}=1/\underline{\sigma}$.

Using \eqref{eqn:forward EKF prediction error dynamics} and the independence of noise terms, we have
\par\noindent\small
\begin{align}
&\mathbb{E}\left[V_{k+1}(\widetilde{\mathbf{x}}_{k+1|k})\vert\widetilde{\mathbf{x}}_{k|k-1}\right]\nonumber\\
&= \widetilde{\mathbf{x}}_{k|k-1}^{T}(\mathbf{U}^{x}_{k}\mathbf{F}_{k}(\mathbf{I}-\mathbf{K}_{k}\mathbf{U}^{y}_{k}\mathbf{H}_{k}))^{T}\nonumber\\
&\hspace{3mm} \times \bm{\Sigma}_{k+1|k}^{-1}(\mathbf{U}^{x}_{k}\mathbf{F}_{k}(\mathbf{I}-\mathbf{K}_{k}\mathbf{U}^{y}_{k}\mathbf{H}_{k}))\widetilde{\mathbf{x}}_{k|k-1}\nonumber\\
& +\mathbb{E}\left[\mathbf{v}_{k}^{T}(\mathbf{U}^{x}_{k}\mathbf{F}_{k}\mathbf{K}_{k})^{T}\bm{\Sigma}_{k+1|k}^{-1}(\mathbf{U}^{x}_{k}\mathbf{F}_{k}\mathbf{K}_{k})\mathbf{v}_{k}\vert\widetilde{\mathbf{x}}_{k|k-1}\right]\nonumber\\
& +\mathbb{E}\left[\mathbf{w}_{k}^{T}\bm{\Sigma}_{k+1|k}^{-1}\mathbf{w}_{k}\vert\widetilde{\mathbf{x}}_{k|k-1}\right].\label{eqn: Vk expectation}
\end{align}
\normalsize
The difference of two matrices $\mathbf{A}-\mathbf{B}$ is invertible if maximum singular value of $\mathbf{B}$ is strictly less than the minimum singular value of $\mathbf{A}$.  Using the assumed bounds, we have $\|\mathbf{K}_{k}\|\leq\overline{k}=(\overline{\sigma}\overline{\gamma}\overline{h}\overline{\beta})/\hat{r}$. Hence, maximum singular value of $\mathbf{K}_{k}\mathbf{U}^{y}_{k}\mathbf{H}_{k}$ is upper-bounded by $(\overline{\sigma}\overline{\gamma}\overline{h}^{2}\overline{\beta}^{2})/\hat{r}$  and the inequality \eqref{eqn:inequality on constants} guarantees that $\mathbf{I}-\mathbf{K}_{k}\mathbf{U}^{y}_{k}\mathbf{H}_{k}$ is invertible (singular value of $\mathbf{I}$ is 1) such that
\par\noindent\small
\begin{align*}
&\bm{\Sigma}_{k+1|k} \nonumber\\
&=\mathbf{U}^{x}_{k}\mathbf{F}_{k}(\mathbf{I}-\mathbf{K}_{k}\mathbf{U}^{y}_{k}\mathbf{H}_{k}) (\bm{\Sigma}_{k|k-1}+(\mathbf{U}^{x}_{k}\mathbf{F}_{k}(\mathbf{I}-\mathbf{K}_{k}\mathbf{U}^{y}_{k}\mathbf{H}_{k}))^{-1}\\
&\hspace{3mm} \times\hat{\mathbf{Q}}_{k}((\mathbf{U}^{x}_{k}\mathbf{F}_{k}(\mathbf{I}-\mathbf{K}_{k}\mathbf{U}^{y}_{k}\mathbf{H}_{k}))^{-1})^{T}) (\mathbf{I}-\mathbf{K}_{k}\mathbf{U}^{y}_{k}\mathbf{H}_{k})^{T}\mathbf{F}_{k}^{T}\mathbf{U}^{x}_{k},
\end{align*}
\normalsize
because $\mathbf{U}^{x}_{k}$ and $\mathbf{F}_{k}$ are also assumed to be invertible. Again with the assumed bounds, we have $\|\mathbf{U}^{x}_{k}\mathbf{F}_{k}(\mathbf{I}-\mathbf{K}_{k}\mathbf{U}^{y}_{k}\mathbf{H}_{k})\|\leq\overline{\alpha}\overline{f}(1+\overline{k}\overline{\beta}\overline{h})$ which implies
\par\noindent\small
\begin{align*}
    (\mathbf{U}^{x}_{k}\mathbf{F}_{k}(\mathbf{I}-\mathbf{K}_{k}\mathbf{U}^{y}_{k}\mathbf{H}_{k}))^{-1}\hat{\mathbf{Q}}_{k}((\mathbf{U}^{x}_{k}\mathbf{F}_{k}(\mathbf{I}-\mathbf{K}_{k}\mathbf{U}^{y}_{k}\mathbf{H}_{k}))^{-1})^{T}\\
    \succeq\frac{\hat{q}}{(\overline{\alpha}\overline{f}(1+\overline{k}\overline{\beta}\overline{h}))^{2}}\mathbf{I}.
\end{align*}
\normalsize
Using this bound in the expression of $\bm{\Sigma}_{k+1|k}$ as in \cite{li2012stochastic_ukf}, we have
\par\noindent\small
\begin{align*}
(\mathbf{U}^{x}_{k}\mathbf{F}_{k}(\mathbf{I}-\mathbf{K}_{k}\mathbf{U}^{y}_{k}\mathbf{H}_{k}))^{T}\bm{\Sigma}_{k+1|k}^{-1}(\mathbf{U}^{x}_{k}\mathbf{F}_{k}(\mathbf{I}-\mathbf{K}_{k}\mathbf{U}^{y}_{k}\mathbf{H}_{k}))\\
\preceq(1-\lambda)\bm{\Sigma}_{k|k-1}^{-1},
\end{align*}
\normalsize
where $1-\lambda=\left(1+\frac{\hat{q}}{\overline{\sigma}(\overline{\alpha}\overline{f}(1+\overline{k}\overline{\beta}\overline{h}))^{2}}\right)^{-1}$ with $0<\lambda<1$. The last two expectation terms in \eqref{eqn: Vk expectation} can be bounded by $\mu=(\overline{r}p\overline{\alpha}^{2}\overline{f}^{2}\overline{k}^{2}/\underline{\sigma})+(\overline{q}n/\underline{\sigma})>0$ following similar steps as in \cite{li2012stochastic_ukf} such that
\par\noindent\small
\begin{align*}
\mathbb{E}\left[ V_{k+1}(\widetilde{\mathbf{x}}_{k+1|k})|\widetilde{\mathbf{x}}_{k|k-1}\right]-V_{k}(\widetilde{\mathbf{x}}_{k|k-1})\leq-\lambda V_{k}(\widetilde{\mathbf{x}}_{k|k-1})+\mu.
\end{align*}
\normalsize
Hence, the second condition of Lemma \ref{lemma:exponential boundedness} is also satisfied and the prediction error $\widetilde{\mathbf{x}}_{k|k-1}$ is exponentially bounded in mean-squared sense and bounded with probability one.

Furthermore, with the bounds assumed on various matrices, it is straightforward to show that
\par\noindent\small
\begin{align*}
\mathbb{E}\left[\|\widetilde{\mathbf{x}}_{k}\|^{2}_{2}\right]\leq(1+\overline{k}\overline{\beta}\overline{h})^{2}\mathbb{E}\left[\|\widetilde{\mathbf{x}}_{k|k-1}\|^{2}_{2}\right]+\overline{k}^{2}\overline{r}p.
\end{align*}
\normalsize
Finally, the exponential boundedness of $\widetilde{\mathbf{x}}_{k|k-1}$ leads to $\widetilde{\mathbf{x}}_{k}$ also being exponentially bounded in mean-squared sense as well as bounded with probability one.

\vspace{-8pt}
\section{Proof of Theorem~\ref{theorem: inverse EKF stable unknown matrix}}
\label{App-thm-inverse EKF stable unknown matrix}
We will show that the I-EKF's dynamics also satisfies the assumptions of Theorem \ref{theorem:Forward ekf stable unknown matrix}. For this, the following conditions \textbf{C1}-\textbf{C13} need to hold true for all $k\geq 0$ for some real positive constants $\overline{a},\overline{g},\overline{b},\overline{c},\overline{d},\hat{q},\overline{\epsilon},\hat{c},\hat{d},\underline{p},\overline{p}$. 
\begin{description}
    \item [C1]
    $\|\widetilde{\mathbf{F}}^{x}_{k}\|\leq\overline{a}$; 
    \item [C2]
    $\|\overline{\mathbf{U}}^{x}_{k}\|\leq\overline{b}$; 
    \item [C3]
    $\overline{\mathbf{U}}^{x}_{k}$ is non-singular; 
    \item [C4]
    $\widetilde{\mathbf{F}}^{x}_{k}$ is non-singular; 
    \item [C5]
    $\overline{\mathbf{Q}}_{k}\preceq\widetilde{q}\mathbf{I}$; 
    \item [C6]
    $\|\mathbf{G}_{k}\|\leq\overline{g}$; 
    \item [C7]
    $\|\overline{\mathbf{U}}^{a}_{k}\|\leq\overline{c}$; 
    \item [C8]
    $\|\overline{\mathbf{U}}^{xa}_{k}\|\leq\overline{d}$; 
    \item [C9]
    $\overline{\mathbf{R}}_{k}\preceq\overline{\epsilon}\mathbf{I}$;
    \item [C10]
    $\hat{c}\mathbf{I}\preceq\hat{\overline{\mathbf{Q}}}_{k}$; 
    \item [C11]
    $\hat{d}\mathbf{I}\preceq\hat{\overline{\mathbf{R}}}_{k}$; 
    \item [C12]
    $\underline{p}\mathbf{I}\preceq\overline{\bm{\Sigma}}_{k|k-1}\preceq\overline{p}\mathbf{I}$; and 
    \item [C13]
    the constants satisfy the inequality $\overline{p}\overline{d}\overline{g}^{2}\overline{c}^{2}<\hat{d}$.
\end{description}

The conditions \textbf{C6-C13} are assumed to hold true in Theorem \ref{theorem: inverse EKF stable unknown matrix}. Next, we prove that under the assumptions of Theorem \ref{theorem: inverse EKF stable unknown matrix}, \textbf{C1}-\textbf{C5} are also satisfied for the I-EKF's error dynamics such that Theorem \ref{theorem:Forward ekf stable unknown matrix} is applicable for the I-EKF as well. From the I-EKF's state transition \eqref{eqn: inverse ekf state transition}, the Jacobians $\widetilde{\mathbf{F}}^{x}_{k}=\mathbf{F}_{k}-\mathbf{K}_{k+1}\mathbf{H}_{k+1}\mathbf{F}_{k}$ and $\widetilde{\mathbf{F}}^{v}_{k}=\mathbf{K}_{k+1}$ such that $\overline{\mathbf{Q}}_{k}=\mathbf{K}_{k+1}\mathbf{R}_{k+1}\mathbf{K}_{k+1}^{T}$.

For \textbf{C1}, using $\|\mathbf{K}_{k+1}\|\leq\overline{k}$ (as proved in Theorem \ref{theorem:Forward ekf stable unknown matrix}) and the bounds on $\mathbf{F}_{k}$ and $\mathbf{H}_{k+1}$ from the assumptions of Theorem \ref{theorem:Forward ekf stable unknown matrix}, it is trivial to show that $\|\widetilde{\mathbf{F}}^{x}_{k}\|=\|\mathbf{F}_{k}-\mathbf{K}_{k+1}\mathbf{H}_{k+1}\mathbf{F}_{k}\|\leq\overline{f}+\overline{k}\overline{h}\overline{f}$.
Hence, \textbf{C1} is satisfied with $\overline{a}=\overline{f}+\overline{k}\overline{h}\overline{f}$.

For \textbf{C2}-\textbf{C4}, consider the unknown matrix $\overline{\mathbf{U}}^{x}_{k}$ introduced to account for the residuals in linearization of $\widetilde{f}_{k}(\cdot)$. Let $\hat{\widetilde{\mathbf{x}}}_{k+1|k}$ and $\hat{\widetilde{\mathbf{x}}}_{k}$ denote the state prediction error and state estimation error of I-EKF. Similar to forward EKF with the introduction of the unknown matrix, we have
\par\noindent\small
\begin{align}
    \hat{\widetilde{\mathbf{x}}}_{k+1|k}=\overline{\mathbf{U}}^{x}_{k}(\mathbf{F}_{k}-\mathbf{K}_{k+1}\mathbf{H}_{k+1}\mathbf{F}_{k})\hat{\widetilde{\mathbf{x}}}_{k}+\mathbf{K}_{k+1}\mathbf{v}_{k+1}\label{eqn:inverse EKF error unknown alpha}.
\end{align}
\normalsize
Also, $\hat{\widetilde{\mathbf{x}}}_{k+1|k}=f(\hat{\mathbf{x}}_{k})-f(\doublehat{\mathbf{x}}_{k})-\mathbf{K}_{k+1}(h(f(\hat{\mathbf{x}}_{k}))-h(f(\doublehat{\mathbf{x}}_{k})))+\mathbf{K}_{k+1}\mathbf{v}_{k+1}$.
Using the unknown matrices $\mathbf{U}^{x}_{k}$ and $\mathbf{U}^{y}_{k}$ introduced in the linearization of $f(\cdot)$ and $h(\cdot)$, respectively, we have
\par\noindent\small
\begin{align*}
    \hat{\widetilde{\mathbf{x}}}_{k+1|k}=(\mathbf{U}^{x}_{k}\mathbf{F}_{k}-\mathbf{K}_{k+1}\mathbf{U}^{y}_{k+1}\mathbf{H}_{k+1}\mathbf{U}^{x}_{k}\mathbf{F}_{k})\hat{\widetilde{\mathbf{x}}}_{k}+\mathbf{K}_{k+1}\mathbf{v}_{k+1}.
\end{align*}
\normalsize
Comparing with \eqref{eqn:inverse EKF error unknown alpha}, we have
\par\noindent\small
\begin{align}
    \overline{\mathbf{U}}^{x}_{k}(\mathbf{I}-\mathbf{K}_{k+1}\mathbf{H}_{k+1})\mathbf{F}_{k}=(\mathbf{I}-\mathbf{K}_{k+1}\mathbf{U}^{y}_{k+1}\mathbf{H}_{k+1})\mathbf{U}^{x}_{k}\mathbf{F}_{k}\label{eqn: unknown alpha for inverse EKF}.
\end{align}
\normalsize
With the additional assumption of $\underline{r}\mathbf{I}\preceq\mathbf{R}_{k}$ and using matrix inversion lemma as in proof of \cite[Lemma 3.1]{reif1999stochastic}, we have
\par\noindent\small
\begin{align*}
    (\mathbf{I}-\mathbf{K}_{k+1}\mathbf{H}_{k+1})\bm{\Sigma}_{k+1|k}=\left(\bm{\Sigma}_{k+1|k}^{-1}+\mathbf{H}_{k+1}^{T}\mathbf{R}_{k+1}^{-1}\mathbf{H}_{k+1}\right)^{-1}.
\end{align*}
\normalsize
Since $\bm{\Sigma}_{k+1|k}$ is invertible by the assumptions of Theorem \ref{theorem:Forward ekf stable unknown matrix}, $\mathbf{I}-\mathbf{K}_{k+1}\mathbf{H}_{k+1}$ is invertible for all $k\geq 0$ and
\par\noindent\small
\begin{align*}
    (\mathbf{I}-\mathbf{K}_{k+1}\mathbf{H}_{k+1})^{-1}=\mathbf{I}+\bm{\Sigma}_{k+1|k}\mathbf{H}_{k+1}^{T}\mathbf{R}_{k+1}^{-1}\mathbf{H}_{k+1}.
\end{align*}
\normalsize
With the bounds assumed on various matrices, we have $ \|(\mathbf{I}-\mathbf{K}_{k+1}\mathbf{H}_{k+1})^{-1}\|\leq 1+\frac{\overline{\sigma}\overline{h}^{2}}{\underline{r}}$.
Furthermore, using this bound and the invertibility of $\mathbf{I}-\mathbf{K}_{k+1}\mathbf{H}_{k+1}$ in \eqref{eqn: unknown alpha for inverse EKF}, it is straightforward to show that $\overline{\mathbf{U}}^{x}_{k}=(\mathbf{I}-\mathbf{K}_{k+1}\mathbf{U}^{y}_{k+1}\mathbf{H}_{k+1})\mathbf{U}^{x}_{k}(\mathbf{I}-\mathbf{K}_{k+1}\mathbf{H}_{k+1})^{-1}$ is non-singular (both $\mathbf{U}^{x}_{k}$ and $\mathbf{I}-\mathbf{K}_{k+1}\mathbf{U}^{y}_{k+1}\mathbf{H}_{k+1}$ are invertible under the assumptions of Theorem \ref{theorem:Forward ekf stable unknown matrix}) and satisfies $\|\overline{\mathbf{U}}^{x}_{k}\|\leq\overline{\alpha}(1+\overline{k}\overline{\beta}\overline{h})(1+(\overline{\sigma}\overline{h}^{2})/\underline{r})$.
Also, since both $\mathbf{I}-\mathbf{K}_{k+1}\mathbf{H}_{k+1}$ and $\mathbf{F}_{k}$ are invertible, $\widetilde{\mathbf{F}}^{x}_{k}=\mathbf{F}_{k}(\mathbf{I}-\mathbf{K}_{k+1}\mathbf{H}_{k+1})$ is non-singular. Hence, \textbf{C2}-\textbf{C4} are also satisfied with $\overline{b}=\overline{\alpha}(1+\overline{k}\overline{\beta}\overline{h})(1+(\overline{\sigma}\overline{h}^{2})/\underline{r})$.

For \textbf{C5}, using the upper bound on $\mathbf{R}_{k}$ from assumptions of Theorem \ref{theorem:Forward ekf stable unknown matrix}, we have $\overline{\mathbf{Q}}_{k}\preceq\overline{r}\mathbf{K}_{k+1}\mathbf{K}_{k+1}^{T}$. Since, $\|\mathbf{K}_{k+1}\|\leq\overline{k}$, the maximum eigenvalue of $\mathbf{K}_{k+1}\mathbf{K}_{k+1}^{T}$ is bounded by $\overline{k}^{2}$ such that $\overline{\mathbf{Q}}_{k}\preceq\overline{k}^{2}\overline{r}\mathbf{I}$. Hence, \textbf{C5} is satisfied with $\title{q}=\overline{k}^{2}\overline{r}$.

\vspace{-8pt}
\section{Proof of Theorem~\ref{theorem: inverse ekf stable Reif}}
\label{App-thm-inverse ekf stable Reif}
We will show that the error dynamics of the I-EKF given by \eqref{eqn: inverse ekf error} satisfies the following conditions for all $k\geq 0$ for some real positive constants $\underline{c},\kappa_{\bar{\phi}}, \epsilon_{\bar{\phi}}$.
\begin{description}
    \item [C1]
    $\underline{c}\mathbf{I}\preceq\overline{\mathbf{Q}}_{k}$.
    \item [C2]
    $\widetilde{\mathbf{F}}^{x}_{k}$ is non-singular matrix for all $k\geq 0$.
    \item [C3]
    $\|\overline{\phi}_{k}(\hat{\mathbf{x}},\doublehat{\mathbf{x}})\|_{2}\leq\kappa_{\bar{\phi}}\|\hat{\mathbf{x}}-\doublehat{\mathbf{x}}\|^{2}_{2}$ for all $\|\hat{\mathbf{x}}-\doublehat{\mathbf{x}}\|_{2}\leq\epsilon_{\bar{\phi}}$ for some $\kappa_{\bar{\phi}}>0$ and $\epsilon_{\bar{\phi}}>0$.
\end{description}

All other conditions of Theorem \ref{theorem: ekf stable Reif} can be proved to hold true for the I-EKF's error dynamics under the assumptions of Theorem \ref{theorem: inverse ekf stable Reif} following similar approach as in proof of Theorem \ref{theorem: inverse EKF stable unknown matrix}, such that the estimation error given by \eqref{eqn: inverse ekf error} is exponentially bounded in mean-squared sense and bounded with probability one provided that the estimation error is bounded with $\overline{\epsilon}>0$ where $\overline{\epsilon}$ depends on the various bounds in the same manner as $\epsilon$ depends in the forward filter case.

For \textbf{C1}, using the bound on $\mathbf{R}_{k}$ from one of the assumptions of Theorem \ref{theorem: ekf stable Reif}, we have $\overline{\mathbf{Q}}_{k}=\mathbf{K}_{k}\mathbf{R}_{k}\mathbf{K}_{k}^{T}\succeq\underline{r}\mathbf{K}_{k}\mathbf{K}_{k}^{T}$.
Substituting for $\mathbf{K}_{k}$, we have
\par\noindent\small
\begin{align*}
\mathbf{K}_{k}\mathbf{K}_{k}^{T}=\mathbf{F}_{k}\bm{\Sigma}_{k}\mathbf{H}_{k}^{T}(\mathbf{H}_{k}\bm{\Sigma}_{k}\mathbf{H}_{k}^{T}+\mathbf{R}_{k})^{-2}\mathbf{H}_{k}\bm{\Sigma}_{k}\mathbf{F}_{k}^{T}.
\end{align*}
\normalsize
With the assumption that $\mathbf{H}_{k}$ is full column rank, $\mathbf{K}_{k}\mathbf{K}_{k}^{T}$ is p.d. as $\mathbf{F}_{k}$ is assumed to be non-singular in Theorem \ref{theorem: ekf stable Reif}. Hence, there exists a constant $\widetilde{q}>0$ which is the minimum eigenvalue of $\mathbf{K}_{k}\mathbf{K}_{k}^{T}$ such that $\mathbf{K}_{k}\mathbf{K}_{k}^{T}\succeq\widetilde{q}\mathbf{I}$ and $\overline{\mathbf{Q}}_{k}\succeq\underline{r}\widetilde{q}\mathbf{I}$. Hence, \textbf{C1} is satisfied with $\underline{c}=\underline{r}\widetilde{q}$.

For \textbf{C2}, $\widetilde{\mathbf{F}}^{x}_{k}=\mathbf{F}_{k}-\mathbf{K}_{k}\mathbf{H}_{k}$ is proved to be invertible for all $k\geq 0$ as an intermediate result in the proof of Theorem \ref{theorem: ekf stable Reif} in \cite[Lemma 3.1]{reif1999stochastic}.

For \textbf{C3}, using $\|\mathbf{K}_{k}\|\leq(\overline{f}\overline{\sigma}\overline{h}/\underline{r})$ (proved in \cite[Lemma 3.1]{reif1999stochastic}) and the bounds on functions $\phi(\cdot)$ and $\chi(\cdot)$ from the assumptions of Theorem \ref{theorem: ekf stable Reif}, we have $\|\overline{\phi}_{k}(\hat{\mathbf{x}},\doublehat{\mathbf{x}})\|_{2}\leq\|\phi(\hat{\mathbf{x}},\doublehat{\mathbf{x}})\|_{2}+\frac{\overline{f}\overline{\sigma}\overline{h}}{\underline{r}}\|\chi(\hat{\mathbf{x}},\doublehat{\mathbf{x}})\|_{2}\leq\left(\kappa_{\phi}+\frac{\overline{f}\overline{\sigma}\overline{h}}{\underline{r}}\kappa_{\chi}\right)\|\hat{\mathbf{x}}-\doublehat{\mathbf{x}}\|_{2}^{2}$,
for $\|\hat{\mathbf{x}}-\doublehat{\mathbf{x}}\|_{2}\leq \textrm{min}(\epsilon_{\phi},\epsilon_{\chi})$. Hence, \textbf{C3} is satisfied with $\kappa_{\bar{\phi}}=\kappa_{\phi}+(\overline{f}\overline{\sigma}\overline{h}/\underline{r})\kappa_{\chi}$ and $\epsilon_{\bar{\phi}}=\textrm{min}(\epsilon_{\phi},\epsilon_{\chi})$.

\section{Proof of Theorem~\ref{thm:consistency}}
\label{App-thm-inverse EKF consistency}
We prove the theorem by the principle of mathematical induction. Define the prediction and estimation errors as $\hat{\widetilde{\mathbf{x}}}_{k|k-1}\doteq\hat{\mathbf{x}}_{k}-\doublehat{\mathbf{x}}_{k|k-1}$ and $\hat{\widetilde{\mathbf{x}}}_{k}\doteq\hat{\mathbf{x}}_{k}-\doublehat{\mathbf{x}}_{k}$, respectively. Assume $\mathbb{E}[\hat{\widetilde{\mathbf{x}}}_{k}\hat{\widetilde{\mathbf{x}}}_{k}^{T}]\preceq\overline{\bm{\Sigma}}_{k}$. We show that the inequality also holds for $(k+1)$-th time step. Substituting \eqref{eqn:SLT state transition} in the I-EKF's recursions, we have $\doublehat{\mathbf{x}}_{k+1|k}=\mathbf{U}^{xv}_{k}\overline{\mathbf{F}}^{x}_{k}\doublehat{\mathbf{x}}_{k}$ and $\overline{\bm{\Sigma}}_{k+1|k}=\mathbf{U}^{xv}_{k}\overline{\mathbf{F}}^{x}_{k}\overline{\bm{\Sigma}}_{k}(\overline{\mathbf{F}}^{x}_{k})^{T}\mathbf{U}^{xv}_{k}+\mathbf{U}^{xv}_{k}\overline{\mathbf{F}}^{v}_{k}\mathbf{R}_{k+1}(\overline{\mathbf{F}}^{v}_{k})^{T}\mathbf{U}^{xv}_{k}$. Hence, $\hat{\widetilde{\mathbf{x}}}_{k+1|k}=\mathbf{U}^{xv}_{k}\overline{\mathbf{F}}^{x}_{k}\hat{\widetilde{\mathbf{x}}}_{k}+\mathbf{U}^{xv}_{k}\overline{\mathbf{F}}^{v}_{k}\mathbf{v}_{k+1}$ such that $\mathbb{E}[\hat{\widetilde{\mathbf{x}}}_{k+1|k}\hat{\widetilde{\mathbf{x}}}_{k+1|k}^{T}]=\mathbf{U}^{xv}_{k}\overline{\mathbf{F}}^{x}_{k}\mathbb{E}[\hat{\widetilde{\mathbf{x}}}_{k}\hat{\widetilde{\mathbf{x}}}_{k}^{T}](\overline{\mathbf{F}}^{x}_{k})^{T}\mathbf{U}^{xv}_{k}+\mathbf{U}^{xv}_{k}\overline{\mathbf{F}}^{v}_{k}\mathbf{R}_{k+1}(\overline{\mathbf{F}}^{v}_{k})^{T}\mathbf{U}^{xv}_{k}$. Since $\mathbb{E}[\hat{\widetilde{\mathbf{x}}}_{k}\hat{\widetilde{\mathbf{x}}}_{k}^{T}]\preceq\overline{\bm{\Sigma}}_{k}$, we have $\mathbb{E}[\hat{\widetilde{\mathbf{x}}}_{k+1|k}\hat{\widetilde{\mathbf{x}}}_{k+1|k}^{T}]\preceq\overline{\bm{\Sigma}}_{k+1|k}$. Similarly, using \eqref{eqn:SLT observation}, we predict observation $\mathbf{a}_{k+1}$ as $\hat{\mathbf{a}}_{k+1|k}=\mathbf{U}^{a}_{k+1}\overline{\mathbf{G}}_{k+1}\doublehat{\mathbf{x}}_{k+1|k}$ and $\overline{\mathbf{S}}_{k+1}=\mathbf{U}^{a}_{k+1}\overline{\mathbf{G}}_{k+1}\overline{\bm{\Sigma}}_{k+1|k}\overline{\mathbf{G}}_{k+1}^{T}\mathbf{U}^{a}_{k+1}+\overline{\mathbf{R}}_{k+1}$ with I-EKF's gain matrix $\overline{\mathbf{K}}_{k+1}=\overline{\bm{\Sigma}}_{k+1|k}\overline{\mathbf{G}}_{k+1}^{T}\mathbf{U}^{a}_{k+1}\overline{\mathbf{S}}_{k+1}^{-1}$. Again, $\hat{\widetilde{\mathbf{x}}}_{k+1}=(\mathbf{I}-\overline{\mathbf{K}}_{k+1}\mathbf{U}^{a}_{k+1}\overline{\mathbf{G}}_{k+1})\hat{\widetilde{\mathbf{x}}}_{k+1|k}-\overline{\mathbf{K}}_{k+1}\bm{\epsilon}_{k+1}$, which implies $\mathbb{E}[\hat{\widetilde{\mathbf{x}}}_{k+1}\hat{\widetilde{\mathbf{x}}}_{k+1}^{T}]=(\mathbf{I}-\overline{\mathbf{K}}_{k+1}\mathbf{U}^{a}_{k+1}\overline{\mathbf{G}}_{k+1})\mathbb{E}[\hat{\widetilde{\mathbf{x}}}_{k+1|k}\hat{\widetilde{\mathbf{x}}}_{k+1|k}^{T}](\mathbf{I}-\overline{\mathbf{K}}_{k+1}\mathbf{U}^{a}_{k+1}\overline{\mathbf{G}}_{k+1})^{T}+\overline{\mathbf{K}}_{k+1}\overline{\mathbf{R}}_{k+1}\overline{\mathbf{K}}_{k+1}^{T}$. Finally, using $\mathbb{E}[\hat{\widetilde{\mathbf{x}}}_{k+1|k}\hat{\widetilde{\mathbf{x}}}_{k+1|k}^{T}]\preceq\overline{\bm{\Sigma}}_{k+1|k}$, we have $\mathbb{E}[\hat{\widetilde{\mathbf{x}}}_{k+1}\hat{\widetilde{\mathbf{x}}}_{k+1}^{T}]\preceq\overline{\bm{\Sigma}}_{k+1}$.

\bibliographystyle{IEEEtran}
\bibliography{main}
\end{document}